\documentclass[11pt]{amsart}

\usepackage{a4wide}
\usepackage[utf8]{inputenc}
\usepackage[T1]{fontenc} 
\usepackage{amsmath}
\usepackage{amsfonts}
\usepackage{amssymb} 
\usepackage{amsthm}
\usepackage{graphicx}
\usepackage{subcaption}
\usepackage{enumerate}
\usepackage{color}
\usepackage{bbm}

\usepackage{mathtools}
\usepackage{xstring}
\usepackage{bm}
\usepackage{verbatim}

\usepackage[ruled,vlined]{algorithm2e}
\usepackage{nomencl}
\usepackage{pdfpages}

\usepackage{tikz}
\usepackage{psfrag}
\usepackage{subcaption}

\usepackage{ifthen}
\usepackage{xstring}

\usepackage{float}
\usepackage{setspace}

\graphicspath{./}

\newcommand{\corrl}[1]{{\color{black} #1}} 
\newcommand{\red}{\color{black}}

\newtheorem{theorem}{Theorem}[section]

\newtheorem{rem}[theorem]{Remark}
\newtheorem{lemma}[theorem]{Lemma}
\newtheorem{corollar}[theorem]{Corollary}

\newcommand{\todo}[1]{}

\newcommand{\eps}{\varepsilon}
\newcommand{\epsgamma}[1][\gamma]{\ifthenelse{\equal{#1}{2\gamma+1}}{\varepsilon}{}}

\newcommand{\WW}{\mathbf{W}}
\newcommand{\Wl}[1][l]{\beta_{#1}} 

\newcommand{\calc}[1]{}

\newcommand{\E}[1]{\mathbb{E}\left[#1\right]} 
\renewcommand{\d}{\mathrm{d}}
\newcommand{\D}{\mathcal{D}}

\newcommand{\Lp}[1][p]{{\mathbb{L}^{#1}}} 
\newcommand{\Hk}[1][k]{{\mathbb{H}^{#1}}} 
\newcommand{\Wkp}[2]{{\mathbb{W}^{#1,#2}}} 

\newcommand{\Ltwo}{\Lp[2]}

\newcommand{\Hone}{\Hk[1]}
\newcommand{\Hmone}{\Hk[-1]}
\newcommand{\Hmonecirc}{\Hk[-1]}

\newcommand{\mDeltaN}{(-\Delta)}

\newcommand{\Vh}{\mathbb{V}_h}
\newcommand{\Vht}{{\mathbb{V}}_h}
\newcommand{\Vhc}{{\mathbb{V}}_h}

\newcommand{\dual}[1]{
\noexpandarg
\StrLen{#1}[\mystringlen]
\ifthenelse{\mystringlen=1}{{#1}'}{(#1)'}
}
\newcommand{\dpairing}[3]{\left\langle #1, #2\right \rangle_{
\noexpandarg
\ifthenelse{\equal{#3}{\Hk[1]}}{}{\dual{#3}\times #3}
}}

\newcommand{\sigmahnm}{\sigma^{n-1}}
\newcommand{\sigmahjm}{\sigma^{j-1}}
\newcommand{\sigmajm}{\sigma^{j-1}}

\begin{document}


\title{Robust a posteriori estimates for the stochastic Cahn-Hilliard equation}

\author{\v{L}ubom\'{i}r Ba\v{n}as}
\address{Department of Mathematics, Bielefeld University, 33501 Bielefeld, Germany}
\email{banas@math.uni-bielefeld.de}
\author{Christian Vieth}
\address{Department of Mathematics, Bielefeld University, 33501 Bielefeld, Germany}
 \email{cvieth@math.uni-bielefeld.de}

\thanks{Funded by the Deutsche Forschungsgemeinschaft (DFG, German Research Foundation) – SFB 1283/2 2021 – 317210226.}

\begin{abstract}
We derive a posteriori error estimates for a fully discrete finite element approximation of the stochastic Cahn-Hilliard equation.
The a posteriori bound is obtained by a splitting of the equation into a linear stochastic partial differential equation (SPDE)
and a nonlinear random partial differential equation (RPDE). 
The resulting estimate is robust with respect to the interfacial width parameter
and is computable since it involves the discrete principal eigenvalue of a linearized (stochastic) Cahn-Hilliard operator.
Furthermore, the estimate is robust with respect to topological changes as well as the intensity of the stochastic noise.
We provide numerical simulations to demonstrate the practicability of the proposed adaptive algorithm.
\end{abstract}


\maketitle

\section{Introduction}\label{Sec_sCH_Intro}
We study a posteriori error estimates for the numerical approximation of the stochastic Cahn-Hilliard equation
\begin{subequations}\label{StochCH}
\begin{alignat}{2}
\d u &= \Delta w \d t + \epsgamma \sigma \d \WW &&\quad \text{ in }(0,T)\times \D, \\
w &= -\varepsilon \Delta u +\varepsilon^{-1}f(u) &&\quad \text{ in }(0,T)\times \D, \\
\partial_{\vec{n}}u &= \partial_{\vec{n}}w = 0&&\quad \text{ on }(0,T)\times \partial\D,\\
u(0,\cdot)&=u_0^\varepsilon &&\quad \text{ in } \D,
\end{alignat}
\end{subequations}
where $T>0$, $\mathcal{D}\subset \mathbb{R}^d$, $d = 1,2,3$ is an open bounded domain and $\epsgamma\sigma \d\WW$ is a noise term (trace-class Wiener process)
which will be specified below.
The constant $0<\varepsilon\ll 1$ is called interfacial width parameter and is usually taken to be small.
The nonlinear term in (\ref{StochCH}b) is given as $f(u)=F'(u)=u^3-u$ where the function $F(u)=\frac{1}{4}(u^2-1)^2$ is a double-well potential.
Without loss of generality, for simplicity we assume throughout the paper that the initial condition $u_0 \equiv u_0^\varepsilon\in \Hk[1]$ satisfies $\int_\D u_0^\varepsilon\d x=0$. 
For further details on the stochastic Cahn-Hilliard equation we refer the reader to \cite{Banas19} and the references therein.

The theoretical and numerical aspects of the deterministic version of the Cahn-Hilliard equation are well studied, see for instance \cite{BloweyElliott91}, \cite{BloweyElliott92}, \cite{fp04}, \cite{fp05}, \cite{BartelsMueller2011}
and the recent review paper \cite{df20}. 
One of the main difficulties in the approximation of the Cahn-Hilliard equation 
is to derive numerical schemes that are robust with respect to the interfacial width parameter $\varepsilon$.
In \cite{fp04, fp05}  the authors propose robust and convergent numerical approximation schemes for the deterministic Cahn-Hilliard equation 
using the lower bound for the (analytic) principal eigenvalue of the linearized Cahn-Hilliard operator; the use of the principal eigenvalue
allows for estimates that depend on $\varepsilon^{-1}$ polynomially. 
The technique that employs the principal eigenvalue goes back to the seminal work \cite{abc94} which shows that
the solutions of the deterministic Cahn-Hilliard equation converge to the Hele-Shaw problem for $\varepsilon \rightarrow 0$, i.e., in the so-called sharp interface limit.
The corresponding sharp interface limit of the numerical approximation of the deterministic Cahn-Hilliard equation has been obtained in \cite{fp05}.
The ideas of \cite{fp04, fp05} have been adopted in \cite{BartelsMueller2011, BartelsMuellerQO} to derive robust computable a posteriori error bounds for the numerical approximation,
which involve a discrete version of the principal eigenvalue. 

Fewer results are available for the stochastic Cahn-Hilliard equation (\ref{StochCH}).  For the proof of existence of a unique (stochastically) strong solution we refer to the earlier result \cite{DaPratoDebussche}. 
In the recent paper \cite{Banas19} robust error estimates for the numerical approximation 
of the stochastic Cahn-Hilliard equations are shown for asymptotically small noise in the form $\varepsilon^\gamma \d\WW$ with sufficiently large scaling factor $\gamma>0$. 
We note that the existence and further properties of the principal eigenvalue in the stochastic setting are not clear.
This issue has been circumvented in \cite{Banas19} by a linearization approach around the corresponding deterministic problems which imposes the restrictive condition on the scaling $\gamma$.
Hence, \cite{Banas19} shows that the numerical approximation of the stochastic problem with asymptotically small noise converges uniformly to the deterministic Hele-Shaw problem in spatial dimension $d=2$.
{We also mention the recent analytical work \cite{BanasZhu2019} that studies the sharp interface limit of the stochastic Cahn-Hilliard equation
that leads to a stochastic version of the Hele-Shaw problem for suitable scaling of the noise 
and \cite{BanasZhu2019} which obtains the (deterministic) sharp interface limit with singular noise. For a more detailed review of existing literature  we refer the reader to the aforementioned papers.
}

Adaptivity for SPDEs is a recent area of research. 
Few results exist on practical adaptive algorithms for SPDEs, see \cite{sllg_book}, \cite{PS19} and the references therein. 
As far as we are aware, apart from the present work, there exist only two other very recent contributions
which derive rigorous a posteriori estimates for SPDEs:
the paper \cite{MajeeProhl} studies a posteriori estimates for (linear) SPDEs
and \cite{bw21} considers a monotone nonlinear  SPDE related to the stochastic total variation flow.

In order to derive the a posteriori error estimate for the numerical approximation of (\ref{StochCH}) we split the solution as $u=\tilde{u}+\hat{u}$ where $\tilde{u}$ solves the linear SPDE (\ref{CHstoch_lin})
and $\hat{u}$ solves the (nonlinear) random PDE (RPDE) (\ref{CHstoch_nonlin}) and proceed as follows:
\begin{itemize}
\item We derive a posteriori estimates for a mixed finite element approximation of the fourth order linear SPDE (\ref{CHstoch_lin}) using an analogue of the transformation approach which was employed in \cite{MajeeProhl} to derive a posteriori estimate
for linear second order SPDEs.
\item To derive a posteriori estimates for a mixed finite element approximation of the nonlinear (fourth order) RPDE (\ref{CHstoch_nonlin}) we generalize the approach for the deterministic Cahn-Hilliard
equation (see \cite{BartelsMueller2011} and the references therein) which relies on the use of the discrete counterpart of the principal eigenvalue of the corresponding linearized Cahn-Hilliard operator, cf., \cite{abc94}, \cite{fp05}.
We derive the pathwise estimate for the nonlinear RPDE, which holds on a suitable probability subset,
using the (random) linearized Cahn-Hilliard operator (\ref{Eigenvalue_sCH}).
The derived estimate involves the (computable) principal eigenvalue of the (random)
Cahn-Hilliard operator linearized at the numerical solution of (\ref{StochCH}).
The size of the probability subset for the pathwise estimate depends on the accuracy of the approximation of the linear SPDE (\ref{CHstoch_lin}), i.e., the size of the subset can be controlled in an a posteriori fashion
by computable quantities.
\item By combining the estimates for the linear SPDE and the nonlinear RPDE in Theorem~\ref{StochCH_este} we 
obtain an error estimate for the numerical approximation of the stochastic Cahn-Hilliard equation (\ref{StochCH}).
\end{itemize}

As a byproduct we obtain several generalizations of existing results.
In contrast to \cite{Banas19}, the restriction of asymptotically small noise and $d=2$ is not explicitly required in the present work, i.e., the derived estimate
is robust w.r.t. the noise intensity and also holds for $d=3$ under additional assumption on the boundedness of the solution. 
The derived estimate retains the robustness properties of its deterministic counterpart \cite{BartelsMueller2011}, i.e.,
the estimate only depends polynomially on the interfacial width parameter $\varepsilon$.
We also obtain the following two generalizations of \cite{MajeeProhl} for linear second order SPDEs: we derive an a posteriori estimate
for a mixed finite element approximation of linear fourth order parabolic SPDEs and analyze the (a posteriori) error due to the truncation of the infinite dimensional
Wiener process. Furthermore, \cite{MajeeProhl} only provides an error estimate for the numerical approximation of the transformed RPDE and does not relate it to the error of the numerical approximation of the original SPDE;
in this paper we also obtain estimates for the numerical approximation of the original SPDE problem which is the actual quantity of interest in simulation.

The rest of the paper is organized as follows.
In Section \ref {Sec_sCH_Notation} we introduce the notation and the assumptions,
and formulate the splitting of the nonlinear SPDE (\ref{StochCH}) into a linear SPDE and a nonlinear random PDE which is used to obtain the a posteriori estimates. 
In Section~\ref{Sec_sCH_Approx} we introduce a fully discrete mixed finite element approximation of (\ref{StochCH})
as well as the discrete counterpart of the continuous splitting into a discrete linear stochastic equation and  a discrete random nonlinear equation.
A posteriori estimate for the linear part of the splitting is derived in Section \ref{Subsec_sCH_EstLin}. 
In Section \ref{Subsec_sCH_EstNonLin} we derive a pathwise a posteriori estimate for the random nonlinear part of the splitting.
Finally, in  Section \ref{Subsec_sCH_Est} we combine the respective a posteriori estimates for the linear SPDE and the nonlinear RPDE
to derive an error estimate for the numerical approximation of (\ref{StochCH}).
We conclude with numerical results in Section \ref{Sec_sCH_NumResults} to illustrate the efficiency and robustness of the adaptive algorithm which is based on the derived estimates.

\section{Notation and preliminaries}\label{Sec_sCH_Notation}
Let  $\D\subset \mathbb{R}^d$  be an open bounded polyhedral domain with boundary $\partial\D$.
We denote the standard Lebesgue space of $p$-th order integrable functions on $\D$ as $\mathbb{L}^p$ and
$\Hk$ denotes the standard Sobolov space $\Wkp{k}{p}$ with $p=2$. We denote the $\mathbb{L}^2$ scalar product as $(\cdot, \cdot)= (\cdot, \cdot)_{\mathbb{L}^2}$
and the corresponding norm by $\|\cdot \| = \|\cdot \|_{\mathbb{L}^2}$.
The duality pairing between $\Hk[1]$ and its dual $\Hk[-1]$ is denoted as $\dpairing{\,\cdot\,}{\,\cdot\,}{\Hk[1]}=\dpairing{\,\cdot\,}{\,\cdot\,}{\Hk[ 1 ]}$.

For $v\in \Hk[-1]$ with $\dpairing{v}{1}{\Hk[1]}=0$, we define the  inverse Laplacian $\mDeltaN^{-1}v=:\tilde{v}$ with $\int_\D \tilde{v} = 0$
to be the unique weak solution of the Poisson equation
$$
(\nabla \tilde{v},\nabla \varphi)\equiv (\nabla \left(\mDeltaN^{-1}v\right),\nabla \varphi)=\dpairing{v}{\varphi}{\Hk[1]} \qquad \forall \varphi \in \Hone\,.
$$
Below we denote $\|v\|_{\Hk[-1]}= \| \nabla \mDeltaN^{-1}v\|$.

The noise term in (\ref{StochCH}) is assumed to be a standard $\mathbf{Q}$-Wiener process\nomenclature[B]{$\WW$}{Wiener process} 
on a filtered probability space ($\Omega$, $\mathcal{F}$, $\{\mathcal{F}_t\}_t$, $\mathbb{P}$), i.e.: 
\begin{equation}\label{wiener}
\d \WW(t,x) = \sum_{l=1}^\infty \nu_l e_l(x) \d \Wl(t),
\end{equation}
where $(e_l)_{l\in \mathbb{N}}$ is an orthonormal basis of $\Hk[4]\cap \mathbb{W}^{1,\infty}$ consisting of the eigenvector of the operator $\mathbf{Q}$ with corresponding eigenvalues $(\nu_l^2)_{l\in \mathbb{N}}$
that satisfy $\sum_{l=1}^\infty \nu_l \|e_l\|_{\Hk[4]} < \infty$, {\red $\sum_{l=1}^\infty \nu_l \|e_l\|_{W^{1,\infty}} < \infty$}.
The processes $(\Wl(t))_{l\in \mathbb{N}}$ are independent real valued Brownian motions.
Furthermore, we assume that $\sigma$ is a time-continuous, $\{\mathcal{F}_t \}_{t\in[0,T]}$-adapted,
$\Hk[4]\cap \mathbb{W}^{1,\infty}$-valued stochastic process that satisfies, 
$\mathbb{P}$-a.s. \corrl{$\int_\D \sigma(t,x)e_l(x) \d x = 0$}, 
and 
$\partial_{\vec{n}} \big(\sigma(t) e_l\big)=\partial_{\vec{n}}\big(\Delta[\sigma(t) e_l]\big)=0$ on $\partial\D$ for $l=1,\ldots,\infty$,  $t\in [0,T]$.

We recall that the weak formulation of (\ref{StochCH}) is given as (cf., \cite{DaPratoDebussche}, \cite{Banas19}): 
\begin{subequations}
\begin{alignat}{2}
\left(u(t),\varphi\right) + \int_0^t \left(\nabla w(s),\nabla \varphi\right) \d s&= \left(u_0^\varepsilon,\varphi\right)  + \epsgamma \int_0^t \left(\sigma(s)\d \WW(s),\varphi\right),\\
\left(w(t),\psi \right) &= \varepsilon \left(\nabla u(t),\nabla \psi\right) +\varepsilon^{-1}\left(f(u(t)),\psi\right),
\end{alignat}
\end{subequations}
for all $\varphi,\psi \in \Hone$ and $t\in [0,T]$.

For the analysis below it is convenient to adopt the approach introduced in \cite{DaPratoDebussche} and split the solution of \eqref{StochCH} as $u = \tilde{u}+\hat{u}$
where $\tilde{u}$ solves the following linear SPDE
\begin{subequations}\label{CHstoch_lin}
\begin{alignat}{2}
\d \tilde{u} &= \Delta \tilde{w} \d t + \epsgamma \sigma \d \WW &&\quad \text{ in }(0,T)\times \D, \\
\tilde{w} &= -\varepsilon \Delta \tilde{u} &&\quad \text{ in }(0,T)\times \D, \\
\partial_{\vec{n}}\tilde{u} &= \partial_{\vec{n}}\tilde{w} = 0&&\quad \text{ on }(0,T)\times \partial\D,\\
\tilde{u}(0,\cdot)&=0 &&\quad \text{ in } \D,
\end{alignat}
\end{subequations}
and $\hat{u}$ solves the random PDE
\begin{subequations}\label{CHstoch_nonlin}
\begin{alignat}{2}
\partial_t \hat{u} &= \Delta \hat{w}  &&\quad \text{ in }(0,T)\times \D, \\
\hat{w} &= -\varepsilon \Delta \hat{u} +\varepsilon^{-1}f(u) &&\quad \text{ in }(0,T)\times \D, \\
\partial_{\vec{n}}\hat{u} &= \partial_{\vec{n}}\hat{w} = 0&&\quad \text{ on }(0,T)\times \partial\D,\\
\hat{u}(0,\cdot)&=u_0^\varepsilon &&\quad \text{ in } \D.
\end{alignat}
\end{subequations}
\todo{$\partial_t \hat{u}$ exists in what sense?}

\section{Fully discrete finite element approximation}\label{Sec_sCH_Approx}

We consider a possibly non-uniform partition $0=t_{0}< t_1<\dots < t_N = T$ of the time interval $[0,T]$ with
time step sizes $\tau_n = t_n-t_{n-1}$. 
At time level $t_n$ we consider a quasi-uniform partition $\mathcal{T}_h^n$ of the domain $\mathcal{D}$ into simplices 
and the associated finite element space of continuous piecewise linear functions
$$
\Vh^n=\{\varphi_h\in\mathcal{C}(\bar{\D}): \varphi_h|_T \in\mathcal{P}^1(T)\quad \forall T\in \mathcal{T}_h^n\}\,.
$$
Throughout the paper we assume for simplicity that $\Vh^{n-1} \subset \Vh^n$ but this condition
can be relaxed, see Remark~\ref{rem_coarsening}. 

For an element $T\in \mathcal{T}_h^n$  we denote by $\mathcal{E}_T$ 
the set of all faces of $\partial T$.
The set of all faces of the elements of the mesh $\mathcal{T}_h^n$ is denoted as
$\mathcal{E}_h^n = \bigcup_{T \in \mathcal{T}_h^n}\mathcal{E}_T$; the diameter of $T \in \mathcal{T}_h^n$ and $e \in \mathcal{E}_h^n$ is denoted as $h_T$ and $h_e$, respectively
and $h:= \max_{T\in \mathcal{T}_h^n} h_T$. 
We split $\mathcal{E}_h^n$ into the set of all interior and boundary faces
$\mathcal{E}_h^n=\mathcal{E}_{h,\D}^n \cup\mathcal{E}_{h,\partial \D}^n$,
where
$\mathcal{E}_{h, \partial\D}^n= \{ e \in \mathcal{E}_h^n : e \subset \partial \D \}$.
Given an $e \in \mathcal{E}_h^n$ we denote by $\mathcal{N}(e)$ the set of  its nodes and for $T \in \mathcal{T}_h^n, e \in \mathcal{E}_h^n$ 
and define the local patches $\omega_T=\bigcup_{\mathcal{E}(T)\cap\mathcal{E}(T')\neq \emptyset }T'$, $\omega_e=\bigcup_{e \in \mathcal{E}(T') }T'$.

We define the $\Ltwo$-projection $P_h^n:\Ltwo\rightarrow \Vh^n$: 
\begin{align}\label{l2proj}
(P_h^nv-v,\varphi_h)&=0\quad\forall\varphi_h\in\Vh^n\,,
\end{align}
with the approximation property
$$
\|v - P_h^nv \| \leq Ch \|\nabla v\| \qquad \forall v \in \Hone\,.
$$
Furthermore, we consider the Cl\'ement-Scott-Zhang interpolation operator $C_h :\Hone \rightarrow \mathbb{V}_h^n$
with the following local approximation properties for $\psi \in \Hone$:
\begin{align}
\label{ch_r}
\|\psi -C_h^n\psi\|_{L^2(T)} +h_T\|\nabla[\psi-C_h^n\psi]\|_{L^2(T)} & \leq C^*h_T\|\nabla \psi\|_{L^2(\omega_T)}  \qquad \forall T\in\mathcal{T}_h^n\,,\\
\label{ch_j}
\|\psi -C_h^n \psi \|_{L^2(e)} & \leq C^*h_e^{\frac{1}{2}}\|\nabla \psi\|_{L^2(\omega_e)}\qquad \forall e \in \mathcal{E}_h^n\,,
\end{align}
where the constant $C^* >0$ only depends on the minimum angle of the mesh $\mathcal{T}_h^n$, see for instance \cite[Def.~3.8]{BartelsNumNonlinear}.

The fully discrete numerical approximation of the stochastic Cahn-Hilliard equation \eqref{StochCH} is given as follows:
set $u_h^0= P_h^0u_0^\varepsilon$, fix the noise truncation parameter $0<r<\infty$ and for $n=1,\ldots,N$ determine the numerical approximations $u_h^n,w_h^n\in \Vh^n$ as the solution of
\begin{subequations}\label{StochCH_discr}
\begin{alignat}{2}
\frac{1}{\tau_n}\left(u_h^n- u_h^{n-1},\varphi_h\right) &+ \left(\nabla w_h^n,\nabla \varphi_h\right) = \epsgamma  \left(\frac{\sigmahnm\Delta_n \WW^r}{\tau_n},\varphi_h\right) \\ 
\left(w_h^n,\varphi_h\right) &= \varepsilon \left(\nabla u_h^n,\nabla \varphi_h\right) +\varepsilon^{-1}\left(f(u_h^n),\varphi_h\right) &&\quad \forall \varphi_h \in \Vh^n,
\end{alignat}
\end{subequations}
where $\sigmahnm \Delta_n \WW^r = \Big( \sigma(t_{n-1})\sum_{l=1}^r \nu_l e_l\Delta_n \Wl \Big)$
with discrete Brownian increments
$$
\Delta_n \Wl= \Wl(t_n)-\Wl(t_{n-1})  \qquad l=1,\ldots, r\,.
$$ 

We define the piecewise linear time interpolant $u_{h,\tau}$ of the numerical solution $\{u_h^n\}_{n=0}^N$ as
\begin{equation}\label{interpol}
u_{h,\tau}(t)=  \frac{t-t_{n-1}}{\tau_n}u_h^{n}+\left(1-\frac{t-t_{n-1}}{\tau_n}\right)u_h^{n-1},
\end{equation}
for $t\in [t_{n-1},t_{n}]$, $n=1,\ldots,N$; analogically we define $w_{h,\tau}$ as the interpolant of  $\{u_h^n\}_{n=0}^N$.

\subsection{The discrete splitting}\label{Subsec_sCH_discrSplitting}
We introduce a discrete analogue of the splitting \eqref{CHstoch_lin}, \eqref{CHstoch_nonlin}: we split the discrete solution as $u_h^n = \tilde{u}_h^n+ \hat{u}_h^n \in \Vh^n$.
The solutions $\tilde{u}_h^n\in \Vht^n$, $\hat{u}_h^n\in \Vhc^n$ satisfy (\ref{CHstoch_lin_discr}), (\ref{CHstoch_nonlin_discr}), respectively, which are the respective discrete counterparts of 
(\ref{CHstoch_lin}) and (\ref{CHstoch_nonlin}).

For $\tilde{u}_h^0=0$ the solutions $\tilde{u}_h^n\in \Vht^n$, $n =1,\ldots,{N}$ satisfy
\begin{subequations}\label{CHstoch_lin_discr}
\begin{alignat}{2}
\left(\frac{\tilde{u}_h^n- \tilde{u}_h^{n-1}}{{\tau_n}},\varphi_h\right) + \left(\nabla \tilde{w}_h^n,\nabla \varphi_h\right) & = \epsgamma  \left(\frac{\sigmahnm\Delta_n \WW^r}{{\tau_n}},\varphi_h\right)  &&\quad \forall \varphi_h \in \Vht^n,\\
\left(\tilde{w}_h^n,\varphi_h\right) &= \varepsilon \left(\nabla \tilde{u}_h^n,\nabla \varphi_h\right) &&\quad \forall \varphi_h \in \Vht^n.
\end{alignat}
\end{subequations}
For $\hat{u}_h^0=u_h^0$ the solutions $\hat{u}_h^n \in \Vhc^n$, $n=1,\ldots,N$ satisfy
\begin{subequations}\label{CHstoch_nonlin_discr} 
\begin{alignat}{2}
\frac{1}{\tau_n}\left(\hat{u}_h^n-\hat{u}_h^{n-1},\varphi_h\right) + \left(\nabla \hat{w}_h^n,\nabla \varphi_h\right) & = 0  &&\quad \forall \varphi_h \in \Vhc^n,\\
\left(\hat{w}_h^n,\varphi_h\right) &= \varepsilon \left(\nabla \hat{u}_h^n,\nabla \varphi_h\right) +\varepsilon^{-1}\left(f(u_h^n),\varphi_h\right) &&\quad \forall \varphi_h \in \Vhc^n.
\end{alignat}
\end{subequations}

The piecewise linear interpolants $\tilde{u}_{h,\tau}$, $\tilde{w}_{h,\tau}$ and $\hat{u}_{h,\tau}$, $\hat{u}_{h,\tau}$
of the respective solutions of (\ref{CHstoch_lin_discr}) and (\ref{CHstoch_nonlin_discr}) are defined analogically to (\ref{interpol}).
Note that $u_{h,\tau}= \tilde{u}_{h,\tau}+\hat{u}_{h,\tau}$.

\begin{rem}\label{rem_spaces}
The solutions of (\ref{CHstoch_lin_discr}), (\ref{CHstoch_nonlin_discr}) may be sought
in possibly different finite element spaces in order to increase the efficiency of the adaptive algorithm.
In scenarios of practical interest the error is typically dominated by the approximation of the nonlinear part.
Typically, the solution of the linear part (\ref{CHstoch_lin_discr}) can  be approximated on a coarser
mesh than the nonlinear equation (\ref{CHstoch_nonlin_discr}) (and (\ref{StochCH_discr})), cf. Figures~\ref{Fig_CHstoch_uh}~and~\ref{Fig_CHstoch_tildeuh_mesh} below.

We also remark that the existence and uniqueness and measurability of the numerical solutions in (\ref{StochCH_discr}), (\ref{CHstoch_lin_discr}), (\ref{CHstoch_nonlin_discr}) follows
by standard arguments, cf. \cite{fp05}, \cite{Banas19}.
\end{rem}

\section{Estimates for the linear stochastic equation}\label{Subsec_sCH_EstLin}
In this section we estimate the approximation error $\tilde{e}= \tilde{u}_{h,\tau}-\tilde{u}$ for the linear stochastic equation \eqref{CHstoch_lin}.

The weak formulation of \eqref{CHstoch_lin} reads as  
\begin{subequations}\label{StochCHu_tildeweak}
\begin{align}
(\tilde{u}(t),\varphi)+\int_0^t(\nabla \tilde{w}(s),\nabla \varphi)\d s &= \epsgamma \int_0^t(\sigma(s)\d \WW(s),\varphi),\\
(\tilde{w}(t),\varphi)&=\varepsilon(\nabla \tilde{u}(t),\nabla \varphi) \qquad\qquad \forall \varphi\in \Hone,
\end{align}
\end{subequations}
for $t\in(0,T)$, $\mathbb{P}$-a.s.

Analogically to \cite{MajeeProhl} we introduce the transformation
\begin{align}\label{ytrans}
y(t,x)&=\tilde{u}(t,x)-\epsgamma \int_0^t\sigma(s,x)\d \WW(s,x),
\end{align}
and define
\begin{align*}
y_w(t,x)&=-\varepsilon \Delta y(t,x)=-\varepsilon \Delta \tilde{u}(t,x)+\varepsilon \Delta\left(\epsgamma \int_0^t\sigma(s,x)\d \WW(s,x)\right)\\
&\equiv\, \tilde{w}(t,x)+ \varepsilon\Delta\left(\epsgamma \int_0^t\sigma(s,x)\d \WW(s,x)\right).
\end{align*}
Note that from the assumptions above it follows that $\partial_{\vec{n}}y=\partial_{\vec{n}}y_w=0$. 

Then $(y,y_w)$ \corrl{$\mathbb{P}$-a.s.} solve the random PDE 
\begin{subequations}\label{rpde}
\begin{align}
(y(t),\varphi)+\int_0^t(\nabla y_w(s),\nabla \varphi)\d s &= \epsgamma \int_0^t(\nabla g(s),\nabla \varphi)\d s\quad \forall \varphi\in\Hone,\\
(y_w(t),\varphi)&=\varepsilon(\nabla y(t),\nabla \varphi),
\end{align}
\end{subequations}
for all $t\in(0,T)$, with $y(0)=0$, where the process $g$ is given by
$$
g(t)=\varepsilon\Delta\int_0^t\sigma(s)\d \WW(s)=\varepsilon\sum_{l=1}^\infty\int_0^t\nu_l\Delta[\sigma(s)e_l]\d \Wl(s). 
$$

\corrl{From \eqref{rpde} it follows by standard arguments, cf. \cite[Section 2]{DaPratoDebussche}, 
that the time derivative of $y$ exists and satisfies $\partial_t y\in L^2(0,T; \mathbb{H}^{-1})$. Hence, \eqref{rpde}}
is equivalent to
\todo{Spaces for $y,y_w$?}
\begin{subequations}\label{eq_randPDEy}
\begin{align}
\dpairing{\partial_t y(t)}{\varphi}{\Hk[1]}+(\nabla y_w(t),\nabla \varphi) &= \epsgamma (\nabla g(t),\nabla \varphi),\\
(y_w(t),\varphi)&=\varepsilon(\nabla y(t),\nabla \varphi),\\
y(0)&= 0.\nonumber
\end{align}
\end{subequations}

We consider the following numerical scheme for the approximation of $y$ and $y_w$: set $y_h^0 \equiv 0\in \Vht^0$ and for $n = 1,\ldots,N$ find $y_h^n,y_{w,h}^n\in \Vht^n$ such that
\begin{subequations}\label{CH_discr_y}
\begin{align}
\left(\frac{y_h^n-y_h^{n-1}}{{\tau_n}},\varphi_h\right)+(\nabla y_{w,h}^n,\nabla \varphi_h) &= \epsgamma (\nabla g^{r,n},\nabla \varphi_h),\\
(y_{w,h}^n,\varphi_h)&=\varepsilon(\nabla y_h^n,\nabla \varphi_h),
\end{align}
\end{subequations}
for all $\varphi_h\in \Vht^n$, where
$$
g^{r,n}=\varepsilon\sum_{j=1}^n\Delta \int_{t_{j-1}}^{t_j} \sigmahjm\d \WW^r(s)
=\varepsilon\sum_{j=1}^n \sum_{l=1}^r\int_{t_{j-1}}^{t_j}\nu_l \Delta [\sigmahjm e_l]\d \Wl(s)\,,
$$
is the truncated version of 
$$
g^{n}=\varepsilon\sum_{j=1}^n\Delta \int_{t_{j-1}}^{t_j} \sigmahjm\d \WW(s)
=\varepsilon\sum_{j=1}^n \sum_{l=1}^\infty\int_{t_{j-1}}^{t_j}\nu_l \Delta[\sigmahjm e_l]\d \Wl(s).
$$

In the next lemma we formulate the discrete counterpart of the transformation (\ref{ytrans})
under the condition that the finite element spaces on all time levels are nested.
\begin{lemma}\label{Lem_stochCHdiscrtrafo}
Assume that \corrl{$\Vht^{n-1}\subset \Vht^{n}$}, $n=0,\dots,N$.
Then the following relation holds $\mathbb{P}$-a.s. between the solutions of \eqref{CH_discr_y} and \eqref{CHstoch_lin_discr}:
\begin{subequations}\label{stochCHdiscrtrafo}
\begin{align}\label{dta}
\tilde{u}_h^n &= y_h^n + \epsgamma P_h^n\sum_{j=1}^n \int_{t_{j-1}}^{t_j} \sigmahjm\d \WW^r(s),\\
\tilde{w}_h^n &= y_{w,h}^n - \epsgamma P_h^n g^{r,n}\,.
\end{align}
\end{subequations}
\end{lemma}
\begin{proof}
We show the statement for $n=1$. The statement for $n=2,\dots, N$ then follows by induction due to the fact that $P_h^n P_h^j=P_h^n$ for $j<n$, since $\Vh^j\subset \Vh^n$.

We consider equation \eqref{CHstoch_lin_discr} for $n=1$, recall $\tilde{u}_h^{0}=0$ and obtain after adding and subtracting the corresponding terms that
\begin{align*}
\left(\frac{(\tilde{u}_h^1- \sigma^1\Delta_1 \WW^r)}{{\tau_1}},\varphi_h\right) + \left(\nabla (\tilde{w}_h^1+\sigma^0 \Delta_1 \WW^r),\nabla \varphi_h\right)
               & =  \left(\nabla \sigma^0 \Delta_1 \WW^r,\nabla \varphi_h\right),\\
\left(\tilde{w}_h^1+ \eps \Delta (\sigma^0\Delta_1\WW^r),\varphi_h\right) &= \varepsilon \left(\nabla (\tilde{u}_h^1-\sigma^0 \Delta_1 \WW^r),\nabla \varphi_h\right)\,.
\end{align*}
where we used that integration by parts implies $\left(\Delta (\sigma^0\Delta_1\WW^r),\varphi_h\right) = -\left(\nabla \sigma^0\Delta_1\WW^r,\nabla \varphi_h\right)$.

On noting that $y_h^0=0$ and the definition of the orthogonal projection (\ref{l2proj})
we deduce that 
$\tilde{y}_h^1 := \tilde{u}_h^1- P_h^1\sigma^1\Delta_1 \WW^r$, $\tilde{y}_{w,h}^1 := \tilde{w}_h^1 + \eps P_h^1\Delta (\sigma^0\Delta_1\WW^r)$
solve \eqref{CH_discr_y} for $n=1$ and by uniqueness of the solutions of \eqref{CH_discr_y}, \eqref{CHstoch_lin_discr} this implies the statement for $n=1$.

The rest of the proof follows by induction.

\end{proof}

\corrl{
\begin{rem}\label{rem_coarsening}
For simplicity, we assume in Lemma~\ref{Lem_stochCHdiscrtrafo}, that the finite element spaces satisfy {$\Vht^{n-1}\subset \Vht^{n}$},
i.e., we neglect the error due to the mesh coarsening.
In the general case the discrete transformation (\ref{dta}) takes the form
$$
\tilde{u}_h^n = y_h^n + \epsgamma \sum_{j=1}^n \int_{t_{j-1}}^{t_j} P_h^n\circ P_h^{n-1}\circ \cdots \circ P_h^j \sigma(t_{j-1})\d \WW^r(s)\,,
$$
and {\red the coarsening errors can be analyzed analogically to \cite{bw21}}.
\end{rem}
}

The piecewise linear time interpolants of the solutions of (\ref{CH_discr_y}) (constructed analogically to (\ref{interpol})) are denoted as $y_{h,\tau}$, $y_{w,h,\tau}$.
Note, that by definition 
$$
\partial_t y_{h,\tau}(t)=\frac{y_h^n-y_h^{n-1}}{\tau_n}\qquad \text{for} \qquad t\in(t_{n-1},t_n).
$$
Furthermore, we define the following piecewise constant time interpolants on $(0,T)$
$$
\bar{g}^r(t)=g^{r,n},\qquad \bar{g}(t)=g^n \quad \text{for } t\in (t_{n-1},t_n].
$$

The above interpolants satisfy
\begin{subequations}\label{Eq_yhtau}
\begin{align}
(\partial_t y_{h,\tau}(t),\varphi)+ (\nabla y_{w,h,\tau}(t), \nabla \varphi)&=\epsgamma(\nabla \bar{g}^r(t),\nabla \varphi)+ \dpairing{\mathcal{R}_y(t)}{\varphi}{\Hk[1]},\\
-(y_{w,h,\tau}(t),\varphi)&=-\varepsilon (\nabla y_{h,\tau}(t),\nabla \varphi)+ \dpairing{\mathcal{S}_y(t)}{\varphi}{\Hk[1]}\qquad \forall \varphi\in \Hone,
\end{align}
\end{subequations}
where we define
\begin{align*}
\dpairing{\mathcal{R}_y(t)}{\varphi}{\Hk[1]}&=\left({\partial_t y_{h,\tau}(t)},\varphi\right)+(\nabla y_{w,h,\tau}(t),\nabla \varphi)-\epsgamma(\nabla \bar{g}^r(t),\nabla \varphi),\\
\dpairing{\mathcal{S}_y(t)}{\varphi}{\Hk[1]}&=-(y_{w,h,\tau}(t),\varphi)+\varepsilon(\nabla y_{h,\tau}(t),\nabla\varphi),
\end{align*}
for $t\in (0,T)$.

\begin{rem}\label{Rem_estRySy}
The above residuals can be estimated by computable quantities as follows.  
On noting \eqref{CH_discr_y} we deduce for any $\varphi\in \Hone$, $\varphi_h\in \Vht^n$ and for $t\in (t_{n-1},t_n]$ that
\begin{align*}
\dpairing{\mathcal{R}_y(t)}{\varphi}{\Hk[1]}&=\left(\frac{y_h^n-y_h^{n-1}}{\tau_n},\varphi-\varphi_h\right)+(\nabla y_{w,h}^n,\nabla [\varphi-\varphi_h])\\
&\quad -\epsgamma(\nabla g^{r,n},\nabla [\varphi-\varphi_h]) + (\nabla [y_{w,h,\tau}(t)-y_{w,h}^n],\nabla \varphi),\\
\dpairing{\mathcal{S}_y(t)}{\varphi}{\Hk[1]}&=(y_{w,h}^n-y_{w,h,\tau}(t),\varphi)+(y_{w,h}^n,\varphi_h-\varphi)\\
&\quad +\varepsilon(\nabla [y_{h,\tau}(t)-y_h^n],\nabla\varphi) +\varepsilon(\nabla y_h^n,\nabla[\varphi-\varphi_h]).
\end{align*}
Setting $\varphi_h=C_h^n\varphi\in \Vht^n$ we obtain after an element-wise integration by parts using (\ref{ch_j}), (\ref{ch_r}), cf.,~e.g.,~\cite[Prop.~6.3]{BartelsNumNonlinear}, that:
\begin{align*}
\dpairing{\mathcal{R}_y(t)}{\varphi}{\Hk[1]}&\le 
(C^*\eta_{\mathrm{SPACE},1}^n+\eta_{\mathrm{TIME},1}^n)\|\nabla \varphi\|=:\mu_{-1}(t)\|\nabla \varphi\|\,,
\end{align*}
and
\begin{align*}
\dpairing{\mathcal{S}_y(t)}{\varphi}{\Hk[1]}&\le \eta_{\mathrm{TIME},2}^n\|\varphi\|+(\eta_{\mathrm{TIME},3}^n+\eta_{\mathrm{SPACE},2}^n+C^*\eta_{\mathrm{SPACE},3}^n)\|\nabla \varphi\|\\
&=:\mu_0(t)\|\varphi\|+\mu_1(t)\|\nabla \varphi\|,
\end{align*}
with the error indicators
\begin{align*}
\eta_{\mathrm{SPACE},1}^n&=
\left(\sum_{T\in\mathcal{T}_h^n } h_T^2 \|\tau_n^{-1}(y_h^n-y_h^{n-1}) + \Delta\epsgamma g^{r,n}\|_{L^2(T)}^2\right)^{1/2}\\
&\qquad+\left(\sum_{e\in \mathcal{E}_h^n} h_e\|[ \nabla {\red y_{w,h}^n}\cdot\vec{n}_e]_e\|_{L^2(e)}^2\right)^{1/2},\\
\eta_{\mathrm{SPACE},2}^n&=\left(\sum_{T\in\mathcal{T}_h^n } h_T^2 \|y_{w,h}^n\|_{L^2(T)}^2\right)^{1/2},\\
\eta_{\mathrm{SPACE},3}^n&=
\left(\varepsilon\sum_{e\in \mathcal{E}_h^n} h_e\|[\nabla y_h^n\cdot\vec{n}_e]_e\|_{L^2(e)}^2\right)^{1/2},\\
\end{align*}
where $[ \nabla u \cdot\vec{n}_e]_e := \nabla u|_{T_1} \cdot \vec{n}_1 + \nabla u|_{T_2} \cdot \vec{n}_2$ for $e= \overline{T}_1\cap \overline{T}_2$ with the vectors
$\vec{n}_1$, $\vec{n}_2$ being the respective outer unit normals to the elements ${T}_1$, ${T}_2\in \mathcal{T}_h^n$ at $e\in \mathcal{E}_h^n$.
Furthermore, the time indicators take the form
\begin{align*}
\eta_{\mathrm{TIME},1}^n&=\|\nabla [y_{w,h}^{n-1}-y_{w,h}^n]\|,\\
\eta_{\mathrm{TIME},2}^n & =\|y_{w,h}^{n-1}-y_{w,h}^n\|,\\
\eta_{\mathrm{TIME},3}^n & =\varepsilon\|\nabla [y_{h}^{n-1}-y_{h}^n]\|.
\end{align*}
\end{rem}

\corrl{In the next lemma we derive an a posteriori error estimate for the numerical approximation \eqref{CH_discr_y} of the linear RPDE \eqref{eq_randPDEy}
which involves the computable error indicators from Remark~\ref{Rem_estRySy}, and, in addition includes the error due to the noise approximation
\begin{align*}
\eta_{\text{NOISE},1}^n = 
& \tau_n\sum_{j=1}^n \tau_j  \sum_{l=r+1}^\infty\nu_l^2\left\| \nabla(\sigmajm e_l)\right\|^2\\
& + \tau_n \sum_{j=1}^n \sum_{l=1}^\infty \nu_l^2\int_{t_{j-1}}^{t_j}\left\|\nabla\big(\{\sigma(s)-\sigmajm\}e_l\big)\right\|^2\d s\\
&+\int_{t_{n-1}}^{t_{n}} \sum_{l=1}^\infty \nu_l^2\int_{t}^{t_{n}}\left\|\nabla(\sigma(s)e_l)\right\|^2\d s\d t\,.
\end{align*}

}

\begin{lemma}\label{Lemma_CHesty}
There exists a constant $C>0$ such that the following error estimate holds
\begin{align*}
\sup_{t\in[0,T]}& \E{\|y_{h,\tau}(t)-y(t)\|_{\Hmonecirc}^2} +\varepsilon \int_{0}^{T}\E{\|\nabla[ y_{h,\tau}(s)-y(s)]\|^2}\d s\\
&\le C\int_{0}^{T} \E{T\mu_{-1}^2(s) + \sqrt{\frac{T}{\eps}}\mu_0^2(s)+ \varepsilon^{-1} \mu_1^2(s)} \d s  
+ C\epsgamma[2\gamma+1]\sum_{n=1}^N \E{\eta_{\text{NOISE},1}^n} \,.
\end{align*}
\nomenclature[B]{$\eta_{\mathrm{NOISE}}$}{Error estimator due to the approximation/truncation of the noise term}
\end{lemma}
\begin{proof}
We subtract \eqref{eq_randPDEy} and \eqref{Eq_yhtau} set $\varphi=\mDeltaN^{-1}[y_{h,\tau}(t)-y(t)]$ in the first resulting equation and $\varphi=y_{h,\tau}(t)-y(t)$ in the second 
resulting equation
and get
\begin{align*}
\frac{1}{2}\frac{\d}{\d t}\|y_{h,\tau}(t)-y(t)\|_{\Hmonecirc}^2 & +( y_{w,h,\tau}(t)-y_w(t), y_{h,\tau}(t)-y(t))
\\&
=\epsgamma(\bar{g}^r(t)-g(t),y_{h,\tau}(t)-y(t))\\
&\qquad + \dpairing{\mathcal{R}_y(t)}{\mDeltaN^{-1}[y_{h,\tau}(t)-y(t)]}{\Hk[1]},\\
-(y_{w,h,\tau}(t)-y_w,y_{h,\tau}(t)-y(t))& =-\varepsilon \|\nabla[ y_{h,\tau}(t)-y(t)]\|^2
+ \dpairing{\mathcal{S}_y(t)}{y_{h,\tau}(t)-y(t)}{\Hk[1]}.
\end{align*}
 We sum up the above equations and take expectation and integrate over $(0,t)$ (recall since $y_{h,\tau}(0)=y(0)= 0$) to obtain
\begin{align*}
&\frac{1}{2} \E{\|y_{h,\tau}(t)-y(t)\|_{\Hmonecirc}^2} +\varepsilon \int_{0}^{t}\E{\|\nabla[ y_{h,\tau}(s)-y(s)]\|^2}\d s\\
&=\epsgamma\int_{0}^{t} \E{(\bar{g}^r(s)-g(s),y_{h,\tau}(s)-y(s))}\d s\\
&\qquad + \int_{0}^{t}\E{\dpairing{\mathcal{R}_y(s)} {\mDeltaN^{-1}[y_{h,\tau}(s)-y(s)]}{\Hk[1]}} \d s\\
&\qquad+ \int_{0}^{t}\E{\dpairing{\mathcal{S}_y(s)}{y_{h,\tau}(s)-y(s)}{\Hk[1]}} \d s.
\end{align*}
Using the respective bounds for ($\mathcal{R}_y, \mathcal{S}_y$) from Remark \ref{Rem_estRySy}  we estimate
\begin{align}\label{sCH_est_y1}
&\frac{1}{2}\E{\|y_{h,\tau}(t)-y(t)\|_{\Hmonecirc}^2} +\varepsilon \int_{0}^{t}\E{\|\nabla[ y_{h,\tau}(s)-y(s)]\|^2}\d s\nonumber\\
&\le \epsgamma\int_{0}^{t} \E{(\bar{g}^r(s)-g(s),y_{h,\tau}(s)-y(s))}\d s
\\
&\qquad + \int_{0}^{t}\E{ \mu_{-1}(s)\|y_{h,\tau}(s)-y(s)\|_{\Hmone}} \d s\nonumber
\\
\nonumber
&\qquad + \int_{0}^{t}\E{\mu_0(s)\|y_{h,\tau}(s)-y(s)\| + \mu_1(s)\|\nabla [y_{h,\tau}(s)-y(s)]\|} \d s.
\\ \nonumber
&= I_1 + I_2 + I_3 + I_4\,.
\end{align}
We estimate the second and fourth term using the Young's inequality as
$$
I_2 \leq 2T\int_{0}^{t}\E{ \mu_{-1}^2(s)}\d s + \frac{1}{8}\sup_{s\in [0,t]}\E{\|y_{h,\tau}(s)-y(s)\|_{\Hmone}^2}\,,
$$
and
$$
I_4 \leq 2\eps^{-1}\int_{0}^{t}\E{ \mu_1^2(s)]}\d s + \frac{\eps}{8}\int_{0}^{t}\E{\|\nabla [y_{h,\tau}(s)-y(s)]\|^2} \d s\,.
$$
The third term can be estimated using the interpolation inequality $\|u\|_{\Ltwo}^2\leq\|u\|_{\Hmone}\|\nabla u\|_{\Ltwo}$ and Young's inquality as
$$
I_3 \leq C\sqrt{\frac{T}{\eps}}\int_{0}^{t}\E{\mu_0^2(s)}\d s  + \frac{\eps}{8} \int_{0}^{t}\E{\|\nabla[y_{h,\tau}(s)-y(s)]\|^2}\d s
 + \frac{1}{8}\sup_{s\in [0,t]}\E{\|y_{h,\tau}(s)-y(s)\|_{\Hmone}^2 }\,.
$$

Recalling the definition of $g$, $\bar{g}$ and $\bar{g}^r$ we get after integrating by parts that
\begin{align}
I_1
 \leq & \sum_{n=1}^N \int_{t_{n-1}}^{t_n} \E{|(\bar{g}^r(s)-\bar{g}(s),y_{h,\tau}(s)-y(s))|+|(\bar{g}(s)-g(s),y_{h,\tau}(s)-y(s))|}\d s\nonumber\\
= & \varepsilon \sum_{n=1}^N \int_{t_{n-1}}^{t_n}\mathbb{E}\Bigg[\Bigg|\Bigg(\sum_{j=1}^n \sum_{l=r+1}^\infty\int_{t_{j-1}}^{t_j}\nu_l \nabla[\sigmahjm e_l]\d \Wl(r),\nabla [y_{h,\tau}(s)-y(s)]\Bigg)\Bigg|\Bigg]\d s\nonumber\\
&+ \varepsilon\sum_{n=1}^N \int_{t_{n-1}}^{t_n} \mathbb{E}\Bigg[\Bigg|\Bigg(\Bigg\{\sum_{j=1}^n \sum_{l=1}^\infty\int_{t_{j-1}}^{t_j}\nu_l \nabla[\{\sigma(r)-\sigmahjm\}e_l]\d \Wl(r)\nonumber \\
&\qquad\qquad\qquad\qquad \quad  -\sum_{l=1}^\infty\int_s^{t_{n}}\nu_l\nabla[\sigma(r)e_l]\d \Wl(r)\Bigg\},\nabla[y_{h,\tau}(s)-y(s)]\Bigg)\Bigg|\Bigg]\d s.\nonumber
\end{align}
After estimating the right-hand side above using Cauchy-Schwarz and Young's inequalities and It\^{o}'s isometry we conclude that
\begin{align*}
I_1 
\le C\varepsilon\epsgamma \sum_{n=1}^N \E{\eta_{\text{NOISE},1}^n} + \frac{\varepsilon}{8}\epsgamma[-\gamma] \int_{0}^{t} \E{\|\nabla[y_{h,\tau}(s)-y(s)]\|^2}\d s. 
\end{align*}
We insert the above estimates for $I_1$, $\dots$, $I_4$ into \eqref{sCH_est_y1} and obtain after absorbing the corresponding terms into the left hand side that
\begin{align*}
  \frac{1}{4} \sup_{s\in[0,t]}& \E{\|y_{h,\tau}(s)-y(s)\|_{\Hmonecirc}^2} + \frac{\eps}{4} \int_{0}^{t}\E{\|\nabla[ y_{h,\tau}(s)-y(s)]\|^2}\d s\\
&\le C\epsgamma[2\gamma+1]\sum_{n=1}^N \E{\eta_{\text{NOISE},1}^n}           
+ C \int_{0}^{T} \E{ T\mu_{-1}^2(s) + \sqrt{\frac{T}{\eps}}\mu_0^2(s) + \varepsilon^{-1}\mu_1^2(s)}\d s\,,
\end{align*}
which concludes the proof.
\end{proof}

In addition to the $L^\infty([0,T];L^2(\Omega;\Hmonecirc))$ estimate from the Lemma \ref{Lemma_CHesty}, we also derive an estimate 
for $y_{h,\tau}-y$ in the stronger $L^2(\Omega;L^\infty([0,T];\Hmonecirc))$-norm.
\begin{corollar}\label{Cor_CHesty}
There exists a constant $C>0$ such that the following estimate holds
\begin{align*}
& \E{ \sup_{t\in[0,T]}\|y_{h,\tau}(t)-y(t)\|_{\Hmonecirc}^2} +\varepsilon \E{\int_{0}^{T}\|\nabla[ y_{h,\tau}(s)-y(s)]\|^2\d s}\\
&\le C\epsgamma[2\gamma+1]\sum_{n=1}^N \E{\eta_{\text{NOISE},1}^n} 
+ C\int_{0}^{T} \E{T\mu_{-1}^2(s) + \sqrt{\frac{T}{\eps}}\mu_0^2(s)+ \varepsilon^{-1} \mu_1^2(s)} \d s.
\end{align*}
\end{corollar}
\begin{proof}
We proceed analogically to the proof of Lemma~\ref{Lemma_CHesty}:
we subtract the equations \eqref{eq_randPDEy} and \eqref{Eq_yhtau}, integrate in time (cf.~\eqref{sCH_est_y1}), take the supremum then the expectation and arrive at
\begin{align*}
&\frac{1}{2} \E{\sup_{t\in(0,T)}\|y_{h,\tau}(t)-y(t)\|_{\Hmonecirc}^2} +\varepsilon \E{\int_{0}^{T}\|\nabla[ y_{h,\tau}(s)-y(s)]\|^2\d s}\\
&\leq\epsgamma \E{ \int_{0}^{T} \Big|(\bar{g}^r(s)-g(s),y_{h,\tau}(s)-y(s))\Big|\d s}\\
&\qquad + \E{\int_{0}^{T}\Big|\dpairing{\mathcal{R}_y(s)} {\mDeltaN^{-1}[y_{h,\tau}(s)-y(s)]}{\Hk[1]}\Big| \d s}\\
&\qquad+ \E{\int_{0}^{T}\Big|\dpairing{\mathcal{S}_y(s)}{y_{h,\tau}(s)-y(s)}{\Hk[1]}\Big| \d s}.
\end{align*}
The remainder of the proof follows exactly as in Lemma~\ref{Lemma_CHesty}.
\end{proof}

We define the following additional ''noise'' error indicators which arise due to the discrete transformation in Lemma~\ref{Lem_stochCHdiscrtrafo}
\begin{align*}
\eta_{\text{NOISE},2}^n = & \tau_n \sum_{l=r+1}^\infty\nu_l^2\left\| \sigmahnm e_l\right\|_{\Hmonecirc}^2
+ \sum_{l=1}^\infty \nu_l^2\int_{t_{n-1}}^{t_n}\left\|\{\sigma(s)-\sigmahnm\}e_l\right\|_{\Hmonecirc}^2 \d s\\
 & + \tau_n\sum_{l=1}^r\nu_l^2\left\|P_h^n (\sigmahnm e_l) - \sigmahnm e_l \right\|_{\Hmonecirc}^2
\\
& + \eps \tau_n  {\sum_{j=1}^n\tau_j\sum_{l=1}^r\nu_l^2\left\|\nabla[P_h^n (\sigmahjm e_l) - \sigmahjm e_l ]\right\|^2}\,,
\\
\eta_{\text{NOISE},3}^n = & \sum_{l=1}^\infty \nu_l^2\int_{{t_{n-1}}}^{t_{n}}\left\|\sigma(s)e_l\right\|_{\Hmonecirc}^2\d s
\,.
\end{align*}
From Lemmas \ref{Lemma_CHesty} and \ref{Lem_stochCHdiscrtrafo}, we deduce the following estimate for the error of the approximation (\ref{CHstoch_lin_discr}) of the linear SPDE (\ref{CHstoch_lin}).
\begin{lemma}\label{Thm_CH_estetilde2}
The following a posteriori estimate holds for the error $\tilde{u}_{h,\tau}-\tilde{u}$:
\begin{align*}
\sup_{t\in[0,T]} &\E{\left\|\tilde{u}_{h,\tau}(t) -\tilde{u}(t)\right\|_{\Hmonecirc}^2} + \varepsilon \int_0^T \E{\left\|\nabla\left[\tilde{u}_{h,\tau}(s) -\tilde{u}(s)\right]\right\|^2}\d s
\\
&\le C\Bigg\{ 
  \epsgamma[2\gamma+1] \sum_{n=1}^N \E{\eta_{\text{NOISE},1}^n} + \epsgamma[2\gamma]\sum_{n=1}^N \E{\eta_{\text{NOISE},2}^n} + \max_{n=1,\dots,N} \E{\eta_{\text{NOISE},3}^n}
\\ 
    &\qquad + \int_{0}^{T} \E{T \mu_{-1}^2(s) +\sqrt{\frac{T}{\varepsilon}}\mu_0^2(s)+ \varepsilon^{-1} \mu_1^2(s)} \d s\\
&\qquad + \max_{n=1,\ldots,N} \Big( \E{\|\tilde{u}_h^{n-1}-\tilde{u}_h^n\|_{\Hmonecirc}^2}+ {\E{\|y_h^{n-1}-y_h^n\|_{\Hmonecirc}^2}}\Big)\\
&\qquad +  \varepsilon \sum_{n=1}^N\tau_n \Big( \E{\|\nabla[\tilde{u}_h^{n-1}-\tilde{u}_h^n]\|^2} + {\E{\|\nabla[y_h^{n-1}-y_h^n]\|^2}}\Big)
\Bigg\}\,.
\end{align*}
\end{lemma}
\begin{proof}
We estimate the error $\tilde{u}_{h,\tau} -\tilde{u}$.
By the uniqueness of the solutions we deduce from (\ref{StochCHu_tildeweak}), (\ref{rpde})  that $\tilde{u}(t)=y(t)+\epsgamma \int_0^t\sigma(s)\d \WW(s)$.
{By the triangle inequality we get 
\begin{align*}
 \left\|\tilde{u}(t)- \tilde{u}_{h,\tau}(t)\right\|_{\Hmonecirc} & = \| y(t)+\epsgamma \int_0^t\sigma(s)\d \WW(s) - y_{h,\tau}(t) +   y_{h,\tau}(t) - \tilde{u}_{h,\tau}(t)\|_{\Hmonecirc}
\\
& \leq \| y(t) - y_{h,\tau}(t)\|_{\Hmonecirc} + \| y_{h,\tau}(t) + \epsgamma \int_0^t\sigma(s)\d \WW(s) - \tilde{u}_{h,\tau}(t) \|_{\Hmonecirc}\,,
\end{align*}
and similarly we estimate $\int_0^t\left\|\nabla\left[\tilde{u}_{h,\tau}(s) -\tilde{u}(s)\right]\right\|\d s$.
}

Hence, we bound
\begin{align}\label{ProofThm_CH_estestilde}
&\sup_{t\in[0,T]}\E{\left\|\tilde{u}_{h,\tau}(t) -\tilde{u}(t)\right\|_{\Hmonecirc}^2} + \varepsilon \int_0^T \E{\left\|\nabla\left[\tilde{u}_{h,\tau}(s) -\tilde{u}(s)\right]\right\|^2}\d s\nonumber
\\
&\qquad \le 2 \sup_{t\in[0,T]}\E{\left\| y_{h,\tau}(t) + \epsgamma \int_0^t\sigma(s)\d \WW(s) - \tilde{u}_{h,\tau}(t) \right\|_{\Hmonecirc}^2}
\\
&\qquad \qquad +  2 \sup_{t\in[0,T]}\E{\left\|y_{h,\tau}(t) -y(t)\right\|_{\Hmonecirc}^2}\nonumber\\
&\qquad \qquad  + 2\varepsilon \int_0^T \E{\left\|\nabla\left[y_{h,\tau}(t) + \epsgamma \int_0^t\sigma(s)\d \WW(s) - \tilde{u}_{h,\tau}(t)\right]\right\|^2}{\d t}\nonumber
\\ \nonumber
&\qquad \qquad + 2\varepsilon \int_0^T \E{\left\|\nabla\left[y_{h,\tau}(t) -y(t)\right]\right\|^2}\d t
\\ \nonumber
& \qquad = I_1+\ldots +I_4.
\end{align}
The terms $I_2$, $I_4$ can be estimated directly by Lemma \ref{Lemma_CHesty}. 

We estimate the first stochastic term using the triangle inequality as
\begin{align*}
I_1 
&\le 14\max_{n=1,\ldots,N} \sup_{t\in[t_{n-1},t_n]}\Bigg\{\E{ \|\tilde{u}_{h,\tau}(t) -\tilde{u}_h^n\|_{\Hmonecirc}^2} + \E{\|y_{h,\tau}(t)- y_h^n\|_{\Hmonecirc}^2}\\
&\qquad\qquad+\epsgamma\E{\left\|P_h^n\sum_{j=1}^n\int_{t_{j-1}}^{t_j}\sigmahjm\d \WW^r(s) -\sum_{j=1}^n\int_{t_{j-1}}^{t_j}\sigmahjm\d \WW^r(s) \right\|_{\Hmonecirc}^2} \\
&\qquad\qquad+\epsgamma\E{\left\|\sum_{j=1}^n\int_{t_{j-1}}^{t_j}\sigmahjm\d \WW^r(s) -\sum_{j=1}^n\int_{t_{j-1}}^{t_j}\sigmahjm\d \WW(s) \right\|_{\Hmonecirc}^2} \\
&\qquad\qquad+ \epsgamma\E{\left\|\sum_{j=1}^n\int_{t_{j-1}}^{t_j}\sigmahjm\d \WW(s) -\int_0^{t_n}\sigma(s)\d \WW(s) \right\|_{\Hmonecirc}^2} \\
&\qquad\qquad+ \epsgamma\E{\Bigg\|\int_0^{t_n}\sigma(s)\d \WW(s) - \int_0^t\sigma(s)\d \WW(s)\Bigg\|_{\Hmonecirc}^2}\\
&\qquad + \E{\Big\|\tilde{u}_h^n-y_h^n-\epsgamma P_h^n\sum_{j=1}^n\int_{t_{j-1}}^{t_j}\sigmahjm\d \WW^r(s) \Big\|_{\Hmonecirc}^2}\Bigg\}\\
&= I_{1,1}+\ldots+I_{1,7},
\end{align*}
where $\corrl{I_{1,7} = 0}$ by Lemma~\ref{Lem_stochCHdiscrtrafo} and we estimate by the linearity of $\tilde{u}_{h,\tau}$, $y_{h,\tau}$
$$
I_{1,1}\leq 28\max_{n=1,\ldots,N} \E{ \left\|\tilde{u}_h^{n-1} -\tilde{u}_h^n\right\|_{\Hmonecirc}^2},
\quad I_{1,2}\leq 28 \max_{n=1,\ldots,N}\E{ \left\|y_h^{n-1}-y_h^n\right\|_{\Hmonecirc}^2}.
$$
Next, we get by It\^{o}'s isometry 
\begin{align*}
I_{1,3} & \leq 14
 \E{\sum_{j=1}^N\tau_j\sum_{l=1}^r\nu_l^2\left\|P_h^n (\sigmahjm e_l) - \sigmahjm e_l \right\|_{\Hmonecirc}^2},
\\
I_{1,4} 
&\leq 14\epsgamma[2\gamma] \sum_{j=1}^N\E{ \sum_{l=r+1}^\infty\nu_l^2 \int_{t_{j-1}}^{t_j} \left\|\sigmahjm e_l\right\|_{\Hmonecirc}^2 \d s}, \\
I_{1,5} 
&\leq 14\epsgamma[2\gamma] \sum_{j=1}^N\E{ \sum_{l=1}^\infty\nu_l^2\int_{t_{j-1}}^{t_j} \left\|\left\{\sigma(s)- \sigmahjm\right\}e_l\right\|_{\Hmonecirc}^2 \d s},\\
{\red I_{1,6}} 
&\leq14\epsgamma[2\gamma]\max_{n=1,\ldots,N} \sup_{t\in[t_{n-1},t_n]}\E{\sum_{l=1}^\infty \nu_l^2\int_{t}^{t_n}\left\|\sigma(s)e_l\right\|_{\Hmonecirc}^2\d s}\,.
\end{align*}
The term $I_3$ can be estimated analogically to $I_1$ as
$$
I_3 \leq C \eps \sum_{n=1}^{N}\Bigg(\eta_{\mathrm{NOISE},1}^n + {\red \tau_n  \E{\sum_{j=1}^n\tau_j\sum_{l=1}^r\nu_l^2\left\|\nabla[P_h^n (\sigmahjm e_l) - \sigmahjm e_l ]\right\|^2}} \Bigg)\,.
$$
The statement then follows from Lemma~\ref{Lemma_CHesty}.
\end{proof}

In the next lemma we derive an a posteriori estimate for the numerical approximation \eqref{CHstoch_lin_discr} of the linear SPDE 
in the stronger $L^2(\Omega;L^\infty([0,T];\Hmonecirc))$-norm which is required for the control of the approximation error pathwise on the probability subset (\ref{def_Omega_tildeeps}), below.
\corrl{The estimate includes a (global) error term that reflect the error of linear
interpolation of the numerical solution; the term is of order $\tau^{2\lambda}$, $2\lambda < 1$ where $\tau:=\displaystyle \max_{n=1,\dots,N}\tau_n$.}
\begin{lemma}\label{Lemma_CHest_etilde}
The following error bound holds:
\begin{align*}
&\E{\sup_{t\in[0,T]}\left\|\tilde{u}_{h,\tau}(t) -\tilde{u}(t)\right\|_{\Hmonecirc}^2} + \varepsilon \int_0^T \E{\left\|\nabla\left[\tilde{u}_{h,\tau}(t) -\tilde{u}(t)\right]\right\|^2}\d t
\\
&\le C\Bigg\{ 
  \epsgamma[2\gamma+1] \sum_{n=1}^N \E{\eta_{\text{NOISE},1}^n} + \epsgamma[2\gamma]\sum_{n=1}^N \E{\eta_{\text{NOISE},2}^n} + \max_{n=1,\dots,N} \E{\eta_{\text{NOISE},3}^n}
\\ 
    &\qquad + \int_{0}^{T} \E{T \mu_{-1}^2(s) +\sqrt{\frac{T}{\varepsilon}}\mu_0^2(s)+ \varepsilon^{-1} \mu_1^2(s)} \d s\\
&\qquad +  \E{\max_{n=1,\ldots,N}\|\tilde{u}_h^{n-1}-\tilde{u}_h^n\|_{\Hmonecirc}^2}+ {\E{\max_{n=1,\ldots,N}\|y_h^{n-1}-y_h^n\|_{\Hmonecirc}^2}}\\
&\qquad +  \varepsilon \sum_{n=1}^N\tau_n \Big( \E{\|\nabla[\tilde{u}_h^{n-1}-\tilde{u}_h^n]\|^2} + {\E{\|\nabla[y_h^{n-1}-y_h^n]\|^2}}\Big)
\\
&\qquad \corrl{+ C_p \tau^{2\lambda}\E{\sum_{l=1}^\infty \nu_l^2\int_0^T\|\sigma(s)e_l\|_{\Hmone}^a\,\d s}^\frac{2}{a}}
\Bigg\}\,,
\end{align*}
\corrl{ for any $\lambda=q-\frac{1}{p}$, $a,\, p\in(2,\infty)$, {\red $a\geq p$}, $q>\frac{1}{p}$ which satisfy {\red $\frac{1}{p}+q<\frac{1}{2}-\frac{1}{a}$}}.
\end{lemma}
\begin{proof}
Analogically to \eqref{ProofThm_CH_estestilde}, by first taking the supremum in time and then then expectation
we obtain that
{
\begin{align*}
&\E{\sup_{t\in[0,T]}\left\|\tilde{u}_{h,\tau}(t) -\tilde{u}(t)\right\|_{\Hmonecirc}^2} + \varepsilon \int_0^T \E{\left\|\nabla\left[\tilde{u}_{h,\tau}(s) -\tilde{u}(s)\right]\right\|^2}\d s\\
&\le 2 \E{\sup_{t\in[0,T]}\left\|\tilde{u}_{h,\tau}(t) -y_{h,\tau}(t)-\epsgamma \int_0^t\sigma(s)\d \WW(s)\right\|_{\Hmonecirc}^2}\\
&\qquad +  2 \E{\sup_{t\in[0,T]}\left\|y_{h,\tau}(t) -y(t)\right\|_{\Hmonecirc}^2}\\
&\qquad  + 2\varepsilon\int_0^T \E{\left\|\nabla\left[\tilde{u}_{h,\tau}(t) -y_{h,\tau}(t)-\epsgamma \int_0^t\sigma(s)\d \WW(s)\right]\right\|^2}\d t\\
&\qquad + 2\varepsilon \int_0^T \E{\left\|\nabla\left[y_{h,\tau}(t) -y(t)\right]\right\|^2}\d t
\\
 & = I_1+\ldots +I_4.
\end{align*}
}
The terms $I_2$, $I_4$ can be directly estimated by Corollary~\ref{Cor_CHesty}, the term $I_3$ is estimated as in Lemma~\ref{Thm_CH_estetilde2}.

We estimate the first stochastic term using the triangle inequality as
\begin{align*}
I_1 & = 
2 \E{{\red \max_{n=1,\ldots,N} }\sup_{t\in[t_{n-1},t_n]}\left\|\tilde{u}_{h,\tau}(t) -y_{h,\tau}(t)-\epsgamma \int_0^t\sigma(s)\d \WW(s)\right\|_{\Hmonecirc}^2}\\
&\leq 14\mathbb{E}\Bigg[ \max_{n=1,\ldots,N} \sup_{t\in[t_{n-1},t_n]}\Big\|\tilde{u}_{h,\tau}(t) -\tilde{u}_h^n\Big\|_{\Hmonecirc}^2\Bigg] 
    + \mathbb{E}\Bigg[\max_{n=1,\ldots,N}\sup_{t\in[t_{n-1},t_n]}\Big\|y_h^n-y_{h,\tau}(t)\Big\|_{\Hmonecirc}^2\Bigg]\\
&\qquad\qquad+\epsgamma\mathbb{E}\left[\max_{n=1,\ldots,N}\Bigg\|P_h^n\sum_{j=1}^n\int_{t_{j-1}}^{t_j}\sigmahjm\d \WW^r(s) -
\sum_{j=1}^n\int_{t_{j-1}}^{t_j}\sigmahjm\d \WW^r(s) \Bigg\|_{\Hmonecirc}^2 \right] \\
&\qquad\qquad+\epsgamma\mathbb{E}\left[\max_{n=1,\ldots,N}\Bigg\|\sum_{j=1}^n\int_{t_{j-1}}^{t_j}\sigmahjm\d \WW^r(s) -\sum_{j=1}^n\int_{t_{j-1}}^{t_j}\sigmahjm\d \WW(s)\Bigg\|_{\Hmonecirc}^2 \right] \\
&\qquad\qquad+ \epsgamma\mathbb{E}\left[\max_{n=1,\ldots,N}\Bigg\|\sum_{j=1}^n\int_{t_{j-1}}^{t_j}\sigmahjm\d \WW(s) - \int_{0}^{t_n}\sigma(s)\d \WW(s) \Bigg\|_{\Hmonecirc}^2\right] \\
&\qquad\qquad+ \epsgamma\mathbb{E}\left[\max_{n=1,\ldots,N} \sup_{t\in[t_{n-1},t_n]} \Bigg\|\int_{0}^{t_n}\sigma(s)\d \WW(s) - \int_0^t\sigma(s)\d \WW(s) \Bigg\|_{\Hmonecirc}^2\right]\\
&\qquad\qquad +  \E{\max_{n=1,\ldots,N}\Big\|\tilde{u}_h^n-y_h^n-\epsgamma P_h^n\sum_{j=1}^n\int_{t_{j-1}}^{t_j}\sigmahjm\d \WW^r(s) \Big\|_{\Hmonecirc}^2}\\
&= I_{1,1}+I_{1,2}+\ldots+I_{1,7},
\end{align*}
where $\corrl{I_{1,7} = 0}$ by Lemma~\ref{Lem_stochCHdiscrtrafo} and we estimate
$$
I_{1,1}\leq 14 \E{ \max_{n=1,\ldots,N}\left\|\tilde{u}_h^{n-1} -\tilde{u}_h^n\right\|_{\Hmonecirc}^2},\quad I_{1,2}\leq14\E{\max_{n=1,\ldots,N} \left\|y_h^{n-1}-y_h^n\right\|_{\Hmonecirc}^2}.
$$
Next, we get by It\^{o}'s isometry
\begin{align*}
I_{1,3} 
&\leq {14 \epsgamma[2\gamma] \max_{n=1,\dots,N}\sum_{j=1}^n\tau_j\E{ \sum_{l=1}^r\nu_l^2  \left\|\sigmahjm e_l - P_h^n\sigmahjm e_l\right\|_{\Hmonecirc}^2 \d s}} ,\\
{I_{1,4}} 
&\leq 14\epsgamma[2\gamma] \sum_{j=1}^N\tau_j\E{ \sum_{l=r+1}^\infty\nu_l^2  \left\|\sigmahjm e_l\right\|_{\Hmonecirc}^2 \d s}, \\
I_{1,5} 
&\leq 14\epsgamma[2\gamma] \sum_{j=1}^N\E{ \sum_{l=1}^\infty\nu_l^2\int_{t_{j-1}}^{t_j} \left\|\left\{\sigma(s)- \sigmahjm\right\}e_l\right\|_{\Hmonecirc}^2 \d s}\,.
\end{align*}

\corrl{To estimate the remaining term $I_{1,6}$ we denote
$$
J(t):=\int_0^t\sigma(s)\d \WW(s)\, .
$$
Then we may write
$$
I_{1,6}=14\, \E{\max_{n=1,\dots,N}\sup_{t\in[t_{n-1},t_n]}\|J(t)-J(t_n)\|^2_{\Hmone}}.
$$
Furthermore, using the notation
$ \displaystyle \|h\|_{C^{0,\lambda}([0,T])}=\sup_{s,t\in (0,T)}\frac{\|h(s)-h(t)\|_{\Hmone}}{|s-t|^\lambda}$
we deduce that
$$
I_{1,6}\le C\tau^{2\lambda} \E{\|J\|^2_{C^{0,\lambda}([0,T])}}\,,
$$
where $\tau:=\displaystyle \max_{n=1,\dots,N}\tau_n$.

Finally, for $a,p\in(2,\infty)$, {\red $a\geq p$}, $q>\frac{1}{p}$ which satisfy {\red $\frac{1}{p}+q<\frac{1}{2}-\frac{1}{a}$} and $\lambda=q-\frac{1}{p}$ 
and fixed constants $\kappa_1, \kappa_2, \kappa_3$ 
we estimate
\begin{align*}
I^{p/2}_{1,6}&\le C\tau^{p\lambda}\E{\|J\|^p_{C^{0,\lambda}([0,T])}}
\le C \kappa_1^p\tau^{p\lambda} \E{\int_0^T\int_0^T\frac{\|J(t)-J(s)\|_{\Hmone}^p}{|t-s|^{1+qp}}\,\d s\,\d t}
\\
&=2C \kappa_1^p\tau^{p\lambda}\E{\int_0^T\int_0^t\frac{\|J(t)-J(s)\|_{\Hmone}^p}{|t-s|^{1+qp}}\,\d s\, \d t}
 =2C\kappa_1^p\tau^{p\lambda}\int_0^T\int_0^t\frac{\E{\|J(t)-J(s)\|_{\Hmone}^p}}{|t-s|^{1+qp}}\,\d s\, \d t
\\
&\le C\kappa_2^p\tau^{p\lambda}\E{\int_0^T\int_0^t\frac{1}{(t-s)^{1+qp}}\left(\int_s^t\sum_{l=1}^\infty \nu_l^2\left\|\sigma(r)e_l\right\|_{\Hmonecirc}^2   \d r\right)^\frac{p}{2}\,\d s\, \d t}
\\
&\le C\kappa_2^p\tau^{p\lambda}\E{\int_0^T\int_0^t\frac{1}{(t-s)^{1+qp}}\left((t-s)^{\frac{a-2}{a}} \int_s^t\sum_{l=1}^\infty \nu_l^2\left\|\sigma(r)e_l\right\|_{\Hmonecirc}^a   \d r\right)^\frac{p}{2}\,\d s\, \d t}
\\
&\le C\kappa_3^p\tau^{p\lambda}\E{\left(\sum_{l=1}^\infty \nu_l^2\int_0^T\|\sigma(r)e_l\|_{\Hmone}^a\,\d r\right)^\frac{p}{a}}
\leq {\red C\kappa_3^p\tau^{p\lambda}\E{\left(\sum_{l=1}^\infty \nu_l^2\int_0^T\|\sigma(r)e_l\|_{\Hmone}^a\,\d r\right)}^\frac{p}{a} }
\,,
\end{align*}
where we used \cite[Corollary 26]{simon90} to deduce the second inequality and the Burkholder inequality to get the third inequality,
and the fourth inequality follows after an application of H\"olders inequality. 
}


The statement of the lemma then follows after collecting the estimates for $I_1,\ldots,I_4$.
\end{proof}

\section{Estimates for the nonlinear random PDE}\label{Subsec_sCH_EstNonLin}
In this section we derive a posteriori estimates for the numerical approximation \eqref{CHstoch_nonlin_discr} of the RPDE \eqref{CHstoch_nonlin}.

Analogically to Section~\ref{Subsec_sCH_EstLin} using the linear time-interpolants $\hat{u}_{h,\tau}$, $\hat{w}_{h,\tau}$ we define the following residuals 
\begin{subequations}\label{StochCH_Def_RcheckScheck}
\begin{align}\nonumber
\dpairing{\hat{\mathcal{R}}(t)}{\varphi}{\Hk[1]}&= \dpairing{\partial_t \hat{u}_{h,\tau}(t)}{\varphi}{\Hk[1]}+(\nabla \hat{w}_{h,\tau}(t),\nabla \varphi ), \\
\nonumber\dpairing{\hat{\mathcal{S}}(t)}{\varphi}{\Hk[1]}&=-(\hat{w}_{h,\tau}(t),\varphi) + \varepsilon (\nabla \hat{u}_{h,\tau}(t), \nabla \varphi) + \varepsilon^{-1}(f(u_{h,\tau}(t)),\varphi).
\end{align}
\end{subequations}
Analogically to Remark~\ref{Rem_estRySy} one can estimate the residuals as follows:
\begin{subequations}\label{est_RcheckScheck}
\begin{align}
\dpairing{\hat{\mathcal{R}}(t)}{\varphi}{\Hk[1]}&\le \hat{\mu}_{-1}(t) \|\nabla \varphi\|,\\
\dpairing{\hat{\mathcal{S}}(t)}{\varphi}{\Hk[1]}&\le\hat{\mu}_{0}(t) \|\varphi\|+\hat{\mu}_{1}(t) \|\nabla \varphi\|\,.  
\end{align}
\end{subequations}
\corrl{The indicators $\hat{\mu}_{i}$, $i=-1,0,1$ are defined as their counterparts $\mu_{i}$ in Remark~\ref{Rem_estRySy}
with $y_h^n$, $y_h^{n-1}$ replaced by $\hat{u}_h^n$, $\hat{u}_h^{n-1}$, respectively,
with the exception that the noise term $g^{r,n}$ is omitted in the space residual $\eta_{\mathrm{SPACE},1}$
and the second time indicator also includes the contribution from the nonlinear term, i.e.,
$$
\eta_{\mathrm{TIME},2}^n = \|y_{w,h}^{n-1}-y_{w,h}^n\| + \varepsilon^{-1}\|f(u_{h}^n) - f(u_{h}^{n-1})\|\,.
$$
}

For an arbitrary $\tilde{\varepsilon}>0$ we define the following set
\begin{equation}\label{def_Omega_tildeeps}
\Omega_{\tilde{\varepsilon}} =\left\{\omega\in\Omega: \sup_{t\in[0,T]}\|\tilde{e}(t)\|_{\Hmonecirc}^2+\varepsilon\int_0^T\|\nabla \tilde{e}(s)\|^2\d s\le \tilde{\varepsilon}\right\}.
\end{equation}
For a fixed $\tilde{\varepsilon}$ the size of the subset $\Omega _{\tilde{\varepsilon}} \subset\Omega$, 
can be controlled by the accuracy of the numerical approximation of the linear SPDE.
In particular the Markov inequality implies that $\mathbb{P}[\Omega_{\tilde{\varepsilon}}]>0$ for any $\tilde{\varepsilon}>0$ 
for sufficiently small $\tau\equiv\tau(\tilde{\eps})$, $h\equiv h(\tilde{\eps})$ 
(and $\mathbb{P}[\Omega_{\tilde{\varepsilon}}]\rightarrow 1$ for $\tau,\ h\rightarrow 0$).
{We note that the condition (\ref{ass_hat}) below requires that $\tilde{\eps}$ is sufficiently small.}
The estimate (\ref{est_uhat}) suggest to choose $\tilde{\eps} = C(\tau^\kappa + h^\gamma)$
for sufficiently small exponents $0 <\kappa, \gamma < 1$, see also Remark~\ref{rem_teps} below.

In addition, for $0 < \delta< 1/2$ we consider the set
\begin{align*} 
\Omega_{\delta}=\Big\{\omega\in\Omega:\, \sup_{t\in(0,T)}\varepsilon^3\big(\|u(t)\|_{\Hone}^6+\|u(t)\|^2\big)\le \tilde{\varepsilon}^{-\delta} \Big\}.
\end{align*}
{Due to the higher-moment energy estimate \cite[Lemma 2.1 ii)]{Banas19} 
the Markov inequality implies that $\Omega_{\delta}\rightarrow \Omega$ for $\tilde{\eps}\rightarrow 0$.}

Finally, the analysis in spatial dimension $d=3$ (cf. Lemma~\ref{Lemma_L3_est}) requires to consider the set $\Omega_\infty \subset \Omega$, s.t.,
$$
\Omega_\infty=\Big\{\omega\in\Omega:\,\,  \sup_{t\in(0,T)}\|\tilde{e}(t)\|_{\Lp[\infty]}+ \sup_{t\in(0,T)} \|\hat{e}(t)\|_{\Lp[\infty]} \le C_\infty\Big\}\,,
$$
for a fixed $C_{\infty}>0$; for $d=2$ we set $\Omega_\infty\equiv \Omega$.
We note that the solution of the deterministic Cahn-Hilliard equation is bounded in the $\Lp[\infty]$-norm, see e.g.~\cite[proof of Theorem 2.3]{abc94}. 

\begin{rem}\label{rem_linfty}
So far the validity of the $\Lp[\infty]$-bound for the stochastic equation 
has not been rigorously verified in general setting. The only available result in this direction is the work \cite{Banas19}
where \cite[inequality (5.8)]{Banas19} indicates that $\mathbb{P}[\Omega_\infty]$ is close to $1$, for sufficiently small $\varepsilon$ and sufficiently small
noise intensity. 

In addition, we assume in the proof of the error estimate in Lemma~\ref{Lemma_CHest_echeck} below
the boundedness of the numerical approximation $\displaystyle \sup_{t\in(0,T)}\|u_{h,\tau}(t)\|_{\Lp[\infty]}\le C_{h,\infty}$.
Even though we can not verify this assumption rigorously, cf. \cite[Lemma~5.1]{Banas19},
the assumption is not particularly restrictive since:
a)  the bound can be verified a posteriori and is in fact always satisfied in numerical simulations
with ''reasonable'' noise, cf., numerical simulations in \cite{Banas19} and Section~\ref{Sec_sCH_NumResults};
b) (potentially) unbounded solutions are not practically computable therefore only the error estimate for
bounded solutions is of practical interest.
\end{rem}

We define the \textit{stochastic} principal eigenvalue\index{Principal eigenvalue} (cf.~\cite{abc94}, \cite{FengProhl}, \cite{BartelsMueller2011}) as
\begin{equation}\label{Eigenvalue_sCH}
-\Lambda_{CH}(t)=\inf_{\substack{ \eta \in \Hone\setminus \{ 0\},\\ \int_\D \eta \d x=0}}\frac{\varepsilon\|\nabla \eta \|^2+\varepsilon^{-1}(f'(u_{h,\tau}(t))\eta,\eta)}{\|\nabla\mDeltaN^{-1}\eta\|^2}.
\end{equation}
We point out that the above definition involves a linearization
about the numerical solution $u_{h,\tau}$ of the stochastic Cahn-Hilliard equation \eqref{StochCH_discr} and is therefore computable.

We recall the following interpolation estimate, cf. \cite{BartelsMueller2011}.
\begin{lemma}\label{Lemma_L3_est}
There exists a constant $C_I>0$ such that for all $v\in \Hone$ if $d=2$ and for all $v\in\Hone\cap \Lp[\infty]$ if $d=3$ we have
\begin{equation}
\|v\|_{\Lp[3]}^3\le C_I\|v\|_{\Lp[\infty]}^{1-a}\|v\|_{\mathbb{H}^{-1}}^a\|\nabla v\|^2,
\end{equation}
where $a=1$ if $d=2$ and $a=4/5$ if $d=3$.
\end{lemma}

We will make use of the following generalization of the Gronwall Lemma, see also \cite{BartelsMueller2011}.
\begin{lemma}\label{Generalized_Gronwall}
Let $T>0$ be fixed. Suppose that the non-negative functions $y_1\in \mathcal{C}([0,T])$, $y_2,y_3\in L^1(0,T)$, $\alpha\in L^\infty(0,T)$ and the real number $A\ge 0$ satisfy
\begin{equation}
y_1(t)+\int_0^t y_2(s)\d s \le A+\int_0^t \alpha(s)y_1(s)\d s+\int_0^t y_3(s)\d s,
\end{equation}
for all $t\in[0,T]$. Assume that for $B\ge 0$, $\beta>0$ and every $t\in [0,T]$ we have 
\begin{equation}
\int_0^t y_3(s)\d s\le B \sup_{s\in [0,t]} y_1^\beta(s)\int_0^t (y_1(s)+y_2(s))\d s.
\end{equation}
Set $E=\exp\left(\int_0^t \alpha(s)\d s \right)$ and assume that $8AE\le(8B(1+T)E)^{-1/\beta}$. We then have 
\begin{equation}
\sup_{t\in [0,T]} y_1(t) + \int_0^T y_2(s)\d  s\le 8A \exp\left(\int_0^T \alpha(s)\d s \right).
\end{equation}
\end{lemma}

Below, we show a pathwise estimate
for the error $\hat{e}=\hat{u}_{h,\tau}-\hat{u}$ of the approximation of the non-linear RPDE
which holds the subspace $\Omega_{\tilde{\varepsilon}}\cap \Omega_{\delta}$.
\begin{lemma}\label{Lemma_CHest_echeck}
Let the assumptions of Lemma~\ref{Lemma_CHest_etilde} be satisfied and assume in addition
that {$\displaystyle \sup_{t\in(0,T)}\|u_{h,\tau}(t)\|_{\Lp[\infty]}\le C_{h,\infty}$ $\mathbb{P}$-a.s. 
on $\Omega_\infty\cap \Omega_{\tilde{\varepsilon}}\cap \Omega_{\delta}$}.
Furthermore, set
\begin{align*}
\alpha(t)&= \left(9 + 4(1-\varepsilon^3)\Lambda_{CH}(t)\right)_+,\quad B= \varepsilon^{-5}({ C_I} C_{h,\infty} C_\infty^{1-a})\,,\\
E&=\exp\left(\int_0^T \alpha(s) \d s\right)\,,
\end{align*}
{and assume that the following inequality is satisfied}
\begin{align}\label{ass_hat}
\int_0^T & \left( \hat{\mu}_{-1}^2(s)+ \varepsilon^{-2}\hat{\mu}_0^2(s)+\varepsilon^{-4}\hat{\mu}_1^2(s) \right)\d s + \|\hat{e}(0)\|_{\Hmonecirc}^2\nonumber
\\
&+\sqrt{\tilde{\varepsilon}}\biggr[4\varepsilon \sqrt{\tilde{\varepsilon}}+4(1-\varepsilon^3)\sqrt{\tilde{\varepsilon}}\int_0^T(\Lambda_{CH}(s))_+\d s +2\left((1-\varepsilon^3)+8\varepsilon^{-3}(1-\varepsilon^3)^2\right)\sqrt{\tilde{\varepsilon}}
\\
&\qquad\qquad +  C C_I C_{h,\infty} C_\infty^{1-a}\varepsilon^{-2}\tilde{\varepsilon}^{1/2+a}\biggr] + C\varepsilon^{-4}\tilde{\eps}^{1/2-\delta} \nonumber
\\ \nonumber
\le &(8E)^{-(1+2/a)}B^{-2/a}(1+T)^{-2/a},
\end{align}
where $a =1$ if $d=2$ and $a=4/5$ if $d=3$. 

Then the following a posteriori estimate holds
\begin{align}\label{est_uhat}
 \sup_{t\in[0,T]} & \|\hat{e}(t)\|_{\Hmonecirc}^2+{\varepsilon^4}\int_0^T \|\nabla \hat{e}(s)\|^2 \d s 
\leq \exp\left(\int_0^T\left(26+4(1-\varepsilon^3)\Lambda_{CH}(s)\right)_+  \d s\right)
 \nonumber
\\ \nonumber
 & \times {8} \bigg(\int_0^T \left( \hat{\mu}_{-1}^2(s)+ \varepsilon^{-2}\hat{\mu}_0^2(s)+\varepsilon^{-4}\hat{\mu}_1^2(s) \right)\d  s + \|\hat{e}(0)\|_{\Hmonecirc}^2
\\
&\qquad  +2\sqrt{\tilde{\varepsilon}}\biggr[2\varepsilon \sqrt{\tilde{\varepsilon}}+2(1-\varepsilon^3)\sqrt{\tilde{\varepsilon}}\int_0^T(\Lambda_{CH}(s))_+\d s 
+\left((1-\varepsilon^3)+8\varepsilon^{-3}(1-\varepsilon^3)^2\right) \sqrt{\tilde{\varepsilon}}
\\ \nonumber
&\qquad  + C C_I C_{h,\infty} C_\infty^{1-a}\varepsilon^{-2}\tilde{\varepsilon}^{1/2+a}\biggr] + C\varepsilon^{-4}\tilde{\eps}^{1/2-\delta} \bigg)\,,
\end{align}
$\mathbb{P}$-a.s. on $\Omega_\infty\cap \Omega_{\tilde{\varepsilon}}\cap \Omega_{\delta}$.
\end{lemma}
\begin{proof}
We denote $\hat{e}_w=\hat{w}_{h,\tau}-\hat{w}$, then by subtracting \eqref{CHstoch_nonlin} and \eqref{StochCH_Def_RcheckScheck} with $\varphi= \mDeltaN^{-1}\hat{e}$ and $\varphi=\hat{e}$ respectively, we get 
\begin{align*}
\dpairing{\partial_t \hat{e}(t)}{\mDeltaN^{-1}\hat{e}(t)}{\Hk[1]}+&(\nabla \hat{e}_w(t),\nabla \mDeltaN^{-1}\hat{e}(t))\\
 & = \dpairing{\hat{\mathcal{R}}(t)}{\mDeltaN^{-1}\hat{e}(t)}{\Hk[1]},\\
-(\hat{e}_w(t),\hat{e}(t)) + \varepsilon (\nabla \hat{e}(t), \nabla \hat{e}(t)) &= - \varepsilon^{-1}(f(u_{h,\tau}(t))-f(u(t)),\hat{e}(t))\\ &\qquad +\dpairing{\hat{\mathcal{S}}(t)}{\hat{e}(t)}{\Hk[1]}.
\end{align*}
We sum up the two equations and obtain
\begin{align}\label{est_hat1}
\frac{1}{2} \frac{\mathrm{d}}{\mathrm{d}t}\|\hat{e}(t)\|_{\Hmonecirc}^2&+\varepsilon \|\nabla \hat{e}(t)\|^2 =
- \varepsilon^{-1}(f(u_{h,\tau}(t))-f(u(t)),\hat{e}(t))
\\ \nonumber
&\; + \dpairing{\hat{\mathcal{R}}(t)}{\mDeltaN^{-1}\hat{e}(t)}{\Hk[1]}
 +\dpairing{\hat{\mathcal{S}}(t)}{\hat{e}(t)}{\Hk[1]}.
\end{align}
On noting that $f=u^3-u$, we write $\hat{e}=e-\tilde{e}$ and use the monotonicity of $u^3$ to derive 
\begin{align}\label{xx1}
\nonumber - (f(u_{h,\tau}(t))-f(u(t)),\hat{e}(t)) 
&= - (u_{h,\tau}(t)-u(t),e(t)) -(u_{h,\tau}(t)^3-u(t)^3,e(t))\\
&\qquad + (f(u_{h,\tau}(t))-f(u(t)),\tilde{e}(t))\\
\nonumber &\le \|e(t)\|^2 + \|f(u_{h,\tau}(t))-f(u(t))\|\|\tilde{e}(t)\|\,.
\end{align}
We have by \cite[eq. (2.6)]{fp05} that
\begin{align*}
&-\left(f(u_{h,\tau}(t))-f(u(t)),e(t)\right)\\
&\qquad \le -\left(f'(u_{h,\tau}(t))e(t),e(t)\right)+ 3\left\|u_{h,\tau}(t)\right\|_{\Lp[\infty]}\left\|e(t)\right\|_{\Lp[3]}^3\,.
\end{align*}
Hence, from above inequality and the eigenvalue estimate \eqref{Eigenvalue_sCH} we deduce
\begin{align}\label{xx2}
\nonumber - \varepsilon^{-1}& (f(u_{h,\tau}(t))-f(u(t)),\hat{e}(t))  
\\
\nonumber &  =  -\varepsilon^{-1}(f(u_{h,\tau}(t))-f(u(t)),e(t))+ \varepsilon^{-1}(f(u_{h,\tau}(t))-f(u(t)),\tilde{e}(t)) 
\\
&\le -\varepsilon^{-1}(f'(u_{h,\tau}(t))e(t),e(t))+ 6\varepsilon^{-1}\|u_{h,\tau}(t)\|_{\Lp[\infty]}\|e(t)\|_{\Lp[3]}^3
\\ \nonumber
&\qquad + \varepsilon^{-1}\|f(u_{h,\tau}(t))-f(u(t))\|\|\tilde{e}(t)\|
\\ \nonumber
&\le \Lambda_{CH}(t)\|e(t)\|_{\Hmonecirc}^2 + \varepsilon \|\nabla e(t)\|^2+ 6\varepsilon^{-1}\|u_{h,\tau}(t)\|_{\Lp[\infty]}\|e(t)\|_{\Lp[3]}^3
\\ \nonumber
&\qquad + \varepsilon^{-1}\|f(u_{h,\tau}(t))-f(u(t))\|\|\tilde{e}(t)\|.
\end{align}
Using (\ref{xx1}) we deduce from (\ref{est_hat1}) that
\begin{align}
&\frac{1}{2} \frac{\mathrm{d}}{\mathrm{d}t}\|\hat{e}(t)\|_{\Hmonecirc}^2+\varepsilon \|\nabla \hat{e}(t)\|^2\nonumber\\
& \le \varepsilon^{-1}\|e(t)\|^2 + \varepsilon^{-1}\|f(u_{h,\tau}(t))-f(u(t))\|\|\tilde{e}(t)\|\nonumber \\
 &\qquad+ \dpairing{\hat{\mathcal{R}}(t)}{\mDeltaN^{-1}\hat{e}(t)}{\Hk[1]}
 +\dpairing{\hat{\mathcal{S}}(t)}{\hat{e}(t)}{\Hk[1]}\nonumber\\
 &\le 2\varepsilon^{-1}(\|\hat{e}(t)\|^2 + \|\tilde{e}(t)\|^2)+ \varepsilon^{-1}\|f(u_{h,\tau}(t))-f(u(t))\|\|\tilde{e}(t)\|\label{BM11_eq8_stoch} \\
  &\qquad+ \dpairing{\hat{\mathcal{R}}(t)}{\mDeltaN^{-1}\hat{e}(t)}{\Hk[1]}
 +\dpairing{\hat{\mathcal{S}}(t)}{\hat{e}(t)}{\Hk[1]}\nonumber\,.
\end{align}
On the other had, using (\ref{xx2}) in (\ref{est_hat1}) yields
\begin{align}
  &\frac{1}{2} \frac{\mathrm{d}}{\mathrm{d}t}\|\hat{e}(t)\|_{\Hmonecirc}^2+\varepsilon \|\nabla \hat{e}(t)\|^2\nonumber \\
  &\le\Lambda_{CH}(t)\|e(t)\|_{\Hone}^2  +\varepsilon \|\nabla e(t)\|^2 + 6\varepsilon^{-1}\|u_{h,\tau}(t)\|_{\Lp[\infty]}\|e(t)\|_{\Lp[3]}^3\nonumber\\
&\qquad + \varepsilon^{-1}\|f(u_{h,\tau}(t))-f(u(t))\|\|\tilde{e}(t)\|\nonumber\\
&\qquad + \dpairing{\hat{\mathcal{R}}(t)}{\mDeltaN^{-1}\hat{e}(t)}{\Hk[1]}
 +\dpairing{\hat{\mathcal{S}}(t)}{\hat{e}(t)}{\Hk[1]}\nonumber\\
&\le \;2\Lambda_{CH}(t)\|\hat{e}(t)\|_{\Hmonecirc}^2 +2\Lambda_{CH}(t)\|\tilde{e}(t)\|_{\Hmonecirc}^2 \label{BM11_eq9_stoch}\\
&\qquad +\varepsilon (\|\nabla \hat{e}(t)\|^2 + \|\nabla \tilde{e}(t)\|^2 +2 \|\nabla \hat{e}(t)\| \|\nabla \tilde{e}(t)\|)\nonumber\\
&\qquad  + 24\varepsilon^{-1}\|u_{h,\tau}(t)\|_{\Lp[\infty]}(\|\hat{e}(t)\|_{\Lp[3]}^3+ \|\tilde{e}(t)\|_{\Lp[3]}^3)\nonumber\\
&\qquad + \varepsilon^{-1}\|f(u_{h,\tau}(t))-f(u(t))\|\|\tilde{e}(t)\| \nonumber\\
&\qquad+  \dpairing{\hat{\mathcal{R}}(t)}{\mDeltaN^{-1}\hat{e}(t)}{\Hk[1]}
 +\dpairing{\hat{\mathcal{S}}(t)}{\hat{e}(t)}{\Hk[1]}.\nonumber
\end{align}
We multiply \eqref{BM11_eq8_stoch} and \eqref{BM11_eq9_stoch} by $\varepsilon^3$ and $(1-\varepsilon^3)$, respectively, and 
sum the results up
\begin{align}\label{BM11_eq9a_stoch}
\frac{1}{2} \frac{\mathrm{d}}{\mathrm{d}t}\| \hat{e}(t)\|_{\Hmonecirc}^2+\varepsilon^4 \|\nabla \hat{e}(t)\|^2
 \le & 2\varepsilon^{2}(\|\hat{e}(t)\|^2 + \|\tilde{e}(t)\|^2)+ 2(1-\varepsilon^3)\Lambda_{CH}(t)\|\hat{e}(t)\|_{\Hmonecirc}^2 \nonumber\\
&\;  +2(1-\varepsilon^3)\Lambda_{CH}(t)\|\tilde{e}(t)\|_{\Hmonecirc}^2 \nonumber\\
&\; +\varepsilon(1-\varepsilon^3) \|\nabla \tilde{e}(t)\|^2 +2\varepsilon(1-\varepsilon^3) \|\nabla \hat{e}(t)\| \|\nabla \tilde{e}(t)\|\nonumber\\
&\;  + 24\varepsilon^{-1}(1-\varepsilon^3)\|u_{h,\tau}(t)\|_{\Lp[\infty]}(\|\hat{e}(t)\|_{\Lp[3]}^3+\|\tilde{e}(t)\|_{\Lp[3]}^3)\\
&\; + \varepsilon^{-1}\|f(u_{h,\tau}(t))-f(u(t))\|\|\tilde{e}(t)\|\nonumber\\
&\; + \dpairing{\hat{\mathcal{R}}(t)}{\mDeltaN^{-1}\hat{e}(t)}{\Hk[1]}
 +\dpairing{\hat{\mathcal{S}}(t)}{\hat{e}(t)}{\Hk[1]}.\nonumber
\end{align}
The residual estimate \eqref{est_RcheckScheck} implies
 \begin{align*}
&2\dpairing{\hat{\mathcal{R}}(t)}{\mDeltaN^{-1}\hat{e}(t)}{\Hk[1]}
 +2\dpairing{\hat{\mathcal{S}}(t)}{\hat{e}(t)}{\Hk[1]}\\
 &\le \hat{\mu}_{-1}(t)^2+ \varepsilon^{-2}\hat{\mu}_{0}(t)^2+\varepsilon^{-4}\hat{\mu}_{1}(t)^2\\
 &+ \|\hat{e}(t)\|_{\Hmonecirc}^2 + \varepsilon^2\|\hat{e}(t)\|^2+ \varepsilon^4\|\nabla \hat{e}(t)\|^2.
 \end{align*}
Furthermore, we estimate
$$
4\eps^2 \|\hat{e}\|_{\Ltwo}^2\leq 4\eps^2 \|\hat{e}\|_{\Hmone}\|\nabla \hat{e}\|_{\Ltwo} \leq 
\frac{\eps^4}{2} \|\nabla \hat{e}\|_{\Ltwo}^2 + 8\|\hat{e}\|_{\Hmone}^2
\,,
$$
and
$$
4\varepsilon(1-\varepsilon^3) \|\nabla \hat{e}\| \|\nabla \tilde{e}\|\leq 
\frac{\eps^4}{2} \|\nabla \hat{e}\|^2  + 8\varepsilon^{-2}(1-\varepsilon^3)^2  \|\nabla \tilde{e}\|^2\,.
$$
Hence, we use the previous three inequalities to estimate the corresponding terms on the right-hand side of  \eqref{BM11_eq9a_stoch} and obtain after integrating the result over $(0,t)$ 
that
\begin{align}\label{est_x1}
& \|\hat{e}(t)\|_{\Hmonecirc}^2+ {\varepsilon^4} \int_0^t\|\nabla \hat{e}(s)\|^2\d s\nonumber\\
& \le \int_0^t\left(9+4(1-\varepsilon^3)\Lambda_{CH}(s)\right)\|\hat{e}(s)\|_{\Hmonecirc}^2\d s \nonumber\\
&\quad  + 48\varepsilon^{-1}\int_0^t\|u_{h,\tau}(s)\|_{\Lp[\infty]}\|\hat{e}(s)\|_{\Lp[3]}^3\d s 
 + 48\varepsilon^{-1}\int_0^t \|u_{h,\tau}(s)\|_{\Lp[\infty]}\|\tilde{e}(s)\|_{\Lp[3]}^3\d s \nonumber\\
&\quad +\int_0^t\left( \hat{\mu}_{-1}(s)^2+ \varepsilon^{-2}\hat{\mu}_{0}(s)^2+\varepsilon^{-4}\hat{\mu}_{1}(s)^2\right)\d s + \|\hat{e}(0)\|_{\Hmonecirc}^2 \nonumber\\
&\quad + 4\varepsilon^{2}\int_0^t\|\tilde{e}(s)\|^2\d s +4(1-\varepsilon^3)\int_0^t\Lambda_{CH}(s)\|\tilde{e}(s)\|_{\Hmonecirc}^2\d s\\
&\quad +\left(2\varepsilon(1-\varepsilon^3) + 8\varepsilon^{-2}(1-\varepsilon^3)^2\right)\int_0^t  \|\nabla \tilde{e}(s)\|^2\d s\nonumber\\
&\quad + 2\varepsilon^{-1}\int_0^t\|f(u_{h,\tau}(s))-f(u(s))\|\|\tilde{e}(s)\|\d s.\nonumber
\end{align}
Next, we bound the nonlinear term on the right-hand side of the above inequality using the Cauchy-Schwarz inequality as
\begin{align}\label{est_f}
\nonumber &\int_0^t\|f(u_{h,\tau}(s))-f(u(s))\|\|\tilde{e}(s)\|\d s\\
&\qquad \le \varepsilon^{-3}\left(\int_0^t\varepsilon^3\|f(u_{h,\tau}(s))-f(u(s))\|^2\d s\right)^{1/2}\left(\varepsilon\int_0^t\|\tilde{e}(s)\|^2\d s\right)^{1/2}
\\
&\qquad \le \varepsilon^{-3}\left(2\int_0^t\varepsilon^3\big(\|f(u_{h,\tau}(s))\|^2 + \|f(u(s))\big)\|^2\d s\right)^{1/2}
\left(\varepsilon\int_0^t\|\tilde{e}(s)\|^2\d s\right)^{1/2}\,,\nonumber
\end{align}
Recalling the definition of $\Omega_{\delta}$ using the continuous embedding $\mathbb{H}^1\subset\mathbb{L}^6$ we estimate
\begin{align*}
\varepsilon^3\|f(u(t))\|^2 & =\varepsilon^3\|u^3(t)-u(t)\|^2\le \varepsilon^3\big(\|u(t)\|_{\Lp[6]}^6+\|u(t)\|_{\Lp[2]}^2\big)
\le C\varepsilon^3\big(\|u(t)\|_{\Hone}^6+\|u(t)\|_{\Lp[2]}^2\big)
\\ &\le C\tilde{\eps}^{-\delta}\,.
\end{align*}
Further, assuming without loss of generality that $\tilde{\eps}^{-\delta}\gg 1$ we estimate
$$
\|f(u_{h,\tau})\|^2 \leq C (C_{h,\infty}^6 + C_{h,\infty}^2) \leq C\tilde{\eps}^{-\delta}\,.
$$
On noting the definition of $\Omega_{\tilde{\varepsilon}}$ 
we estimate (\ref{est_f}) using the last two inequalities and the Poincar\'e inequality as
\begin{align}\label{est2}
\int_0^t\|f(u_{h,\tau}(s))-f(u(s))\|\|\tilde{e}(s)\|\d s \leq C \eps^{-3}\tilde{\eps}^{-\delta/2}\tilde{\eps}^{1/2}\,.
\end{align}
Using Lemma \ref{Lemma_L3_est}, the definition of $\Omega_{\tilde{\eps}}$ (along with the fact that
$\|\tilde{e}\|_{\Lp[\infty]}\le C_\infty $ on $\Omega_\infty$ for $d=3$) we deduce
\begin{align}\label{est3}
&\int_0^t \|\tilde{e}(s)\|_{\Lp[3]}^3\d s \le C_I C_\infty^{1-a} \sup_{s\in[0,t]}\|\tilde{e}(s)\|_{\Hmonecirc}^a \int_0^t \|\nabla \tilde{e}(s)\|^2\d s
\leq \eps^{-1}C_I C_\infty^{1-a} \tilde{\eps}^{1+\corrl{a/2}}\,,
\end{align}
with $a=1$ in $d=2$ and $a=4/5$ in $d=3$, respectively.

We substitute (\ref{est2}), (\ref{est3}) into (\ref{est_x1}) and arrive at
\begin{align}
& \|\hat{e}(t)\|_{\Hmonecirc}^2+\frac{\varepsilon^4}{2} \int_0^t\|\nabla \hat{e}(s)\|^2\d s\nonumber\\
& \quad \le \int_0^t\left(9+4(1-\varepsilon^3)\Lambda_{CH}(s)\right)\|\hat{e}(s)\|_{\Hmonecirc}^2\d s \nonumber\\
&\qquad  + {48 C_{h,\infty}\varepsilon^{-1}\int_0^t\|\hat{e}(s)\|_{\Lp[3]}^3\d s } 
+\int_0^t\left( \hat{\mu}_{-1}(s)^2+ \varepsilon^{-2}\hat{\mu}_{0}(s)^2+\varepsilon^{-4}\hat{\mu}_{1}(s)^2\right)\d s + \|\hat{e}(0)\|_{\Hmonecirc}^2 \nonumber\\
&\qquad +\sqrt{\tilde{\varepsilon}}\biggr[4\varepsilon \sqrt{\tilde{\varepsilon}}+4(1-\varepsilon^3)\sqrt{\tilde{\varepsilon}}\int_0^t(\Lambda_{CH}(s))_+\d s + 
\left(2(1-\varepsilon^3)+8\varepsilon^{-3}(1-\varepsilon^3)^2\right) \sqrt{\tilde{\varepsilon}} \nonumber\\
&\qquad\qquad\quad + C C_I C_{h,\infty} C_\infty^{1-a}\varepsilon^{-2}\tilde{\varepsilon}^{1/2+\corrl{a/2}}\biggr] + C\varepsilon^{-4}\tilde{\eps}^{1/2-\delta}. \nonumber
\end{align}
Finally, on noting that (on $\Omega_\infty$ if $d=3$)
$$
\int_{0}^t\|\hat{e}\|_{\Lp[3]}^3\leq C_I C_\infty^{1-a}\sup_{t\in(0,t)}\|\hat{e}\|_{\mathbb{H}^{-1}}^a\int_0^t\|\nabla \hat{e}\|^2\d s\,,
$$
the proof can be concluded by the generalized Gronwall Lemma \ref{Generalized_Gronwall} with
\begin{align*}
y_1(t)&=\|\hat{e}(t)\|_{\Hmonecirc}^2,\quad y_2(t)=\frac{\varepsilon^4}{2}\|\nabla \hat{e}(t)\|^2,\quad 
y_3(t)=48\varepsilon^{-1}C_{h,\infty}\|\hat{e}(t)\|_{\Lp[3]}^3\,,
\end{align*}
with $\beta=a/2$ and $\alpha$, $A$, $B$ defined as above, cf., \cite[proof of Proposition 4.4]{BartelsMueller2011}.
\end{proof}

\section{Error estimate for the numerical approximation of the stochastic Cahn-Hilliard equation}\label{Subsec_sCH_Est}

We state the following bound for the error $\hat{e}$, which follows by the triangle inequality from the energy estimate for $\hat{u}$ in \cite[Sec.~2.1]{DaPratoDebussche};
an analogous estimate for the discrete solution $\hat{u}_{h,\tau}$ can be derived
by combining \cite[Sec.~2.1]{DaPratoDebussche} with \cite[Lemma 3.2, iv)]{Banas19}, \cite{fp05}.
\begin{lemma}\label{Prop_boundedness_checke}
There exists a constant $\hat{C}_0>0$, such that 
$$
\E{\left(\sup_{t\in[0,T]}\left\|\hat{e}(t)\right\|_{\Hmonecirc}^2 + \varepsilon \int_0^T \left\|\nabla \hat{e}(s)\right\|^2\d s\right)^{2}}\le \hat{C}_0\,,
$$
where the constant $\hat{C}_0$ depends on $\|u_0\|_{\mathbb{H}^1}$, $\D$, $T$.
\end{lemma}

As a combination of the a posteriori estimates for the errors $\tilde{e}$ and $\hat{e}$ we conclude the following error estimate for the approximation error of the stochastic Cahn-Hilliard equation.
To simplify the notation in the theorem below we use the compact notation for the right hand sides of the a posteriori estimates.
We write the a posteriori estimate from Lemma~\ref{Lemma_CHest_etilde} as
\begin{align*}
 & \E{\sup_{t\in[0,T]}\|\tilde{e}(t)\|_{\Hmonecirc}^2+ \varepsilon\int_0^T \|\nabla \tilde{e}(s)\|^2 \d s} \leq \mathbb{E}[\tilde{\mathcal{R}}]\,,
\end{align*}
and the estimate from Lemma~\ref{Lemma_CHest_echeck} is written as
\begin{align*}
 &\sup_{t\in[0,T]}\|\hat{e}(t)\|_{\Hmonecirc}^2+\frac{\varepsilon^4}{2}\int_0^T \|\nabla \hat{e}(s)\|^2 \d s \leq \hat{\mathcal{R}}\,,
\end{align*}
$\mathbb{P}$-a.s. on $\Omega_\infty\cap \Omega_{\tilde{\varepsilon}}\cap \Omega_{\delta}$.
\begin{theorem}\label{StochCH_este}
Let the assumptions of Lemmas~\ref{Lemma_CHest_etilde},~\ref{Lemma_CHest_echeck}~and~\ref{Prop_boundedness_checke} hold.
Then the following error estimate holds
\begin{align*}
\sup_{t\in[0,T]}&\E{\mathbbm{1}_{\Omega_\infty}\left\|u_{h,\tau}(t)-u(t)\right\|_{\Hmonecirc}^2} + \varepsilon \int_0^T \E{\mathbbm{1}_{\Omega_\infty}\left\|\nabla \left[u_{h,\tau}(s)-u(s)\right]\right\|^2}\d s
\\
\le & C\Big(\mathbb{E}[\tilde{\mathcal{R}}] + \mathbb{E}[\mathbbm{1}_{\Omega_\infty\cap \Omega_{\tilde{\varepsilon}}\cap \Omega_{\delta}}\hat{\mathcal{R}}] 
+ \hat{C}_0^{1/2}\sqrt{\tilde{\eps}^{-1}\mathbb{E}[\tilde{\mathcal{R}}] + \tilde{\eps}^{\delta}C_0}\Big)\,,
\end{align*}
for any $0<\delta < 1/2$.
\end{theorem}

\begin{rem}\label{rem_teps}
\corrl{
The restriction $0<\delta < 1/2$ in the above estimate is necessary to control the term $\tilde{\eps}^{1/2-\delta}$ in (\ref{est_uhat});
the optimal choice appears to be $\delta = 1/3$.
If we assume the convergence of the a posteriori estimate in Lemma~\ref{Lemma_CHest_etilde}
(i.e.,  that $\mathbb{E}[\tilde{\mathcal{R}}]\rightarrow 0$ for $h,\tau\rightarrow 0$)
we may choose $\tilde{\eps} \approx \mathbb{E}[\tilde{\mathcal{R}}]^{\alpha}$
for some $\alpha < 1$. The choice that yields the best order of convergence in the last term in the above estimates for $\delta=1/3$ is $\alpha = 3/4$.

We also stress, that the restriction of the estimate to $\Omega_\infty$ is only required in spatial dimension $d=3$ since $\Omega_\infty\equiv \Omega$ for $d=2$.
}
\end{rem}
\begin{proof} We split the error as $e = {u}_{h,\tau}- {u}  = \hat{e} + \tilde{e} $ and assume 
without loss of generality that
$\tilde{\varepsilon}, \mathrm{tol}<1$ and that  $\hat{C}_0\ge 1$ in Lemma~\ref{Prop_boundedness_checke}.

Then by the triangle and Causchy-Schwarz inequalities we estimate
\begin{align*}
&\sup_{t\in[0,T]}\E{\mathbbm{1}_{\Omega_\infty}\left\|e(t)\right\|_{\Hmonecirc}^2} + \varepsilon \int_0^T \E{\mathbbm{1}_{\Omega_\infty}\left\|\nabla e(s)\right\|^2}\d s\\
& \le 2 \sup_{t\in[0,T]}\E{\mathbbm{1}_{\Omega_\infty}\left\|\tilde{e}(t)\right\|_{\Hmonecirc}^2} + 2\varepsilon \int_0^T \E{\mathbbm{1}_{\Omega_\infty}\left\|\nabla \tilde{e}(s)\right\|^2}\d s\\
&\qquad+ 2\E{\mathbbm{1}_{\Omega_\infty}\sup_{t\in[0,T]}\left\|\hat{e}(t)\right\|_{\Hmonecirc}}^2 + 2\varepsilon \int_0^T \E{\mathbbm{1}_{\Omega_\infty}\left\|\nabla \hat{e}(s)\right\|^2}\d s\\
& = 2 \sup_{t\in[0,T]}\E{\mathbbm{1}_{\Omega_\infty}\left\|\tilde{e}(t)\right\|_{\Hmonecirc}^2} + 2\varepsilon \int_0^T \E{\mathbbm{1}_{\Omega_\infty}\left\|\nabla \tilde{e}(s)\right\|^2}\d s\\
&\qquad+ 2\E{\mathbbm{1}_{\Omega_\infty}\mathbbm{1}_{\Omega_{\tilde{\varepsilon}}\cap \Omega_{\delta}}\left(\sup_{t\in[0,T]}\left\|\hat{e}(t)\right\|_{\Hmonecirc}^2 + \varepsilon \int_0^T \left\|\nabla \hat{e}(s)\right\|^2\d s\right)}\\
&\qquad+ 2\E{\mathbbm{1}_{\Omega_\infty}\mathbbm{1}_{(\Omega_{\tilde{\varepsilon}}\cap \Omega_{\delta})^c}\left(\sup_{t\in[0,T]}\left\|\hat{e}(t)\right\|_{\Hmonecirc}^2 + \varepsilon \int_0^T \left\|\nabla \hat{e}(s)\right\|^2\d s\right)}\\
&\le 2 \sup_{t\in[0,T]}\E{\mathbbm{1}_{\Omega_\infty}\left\|\tilde{e}(t)\right\|_{\Hmonecirc}^2} + 2\varepsilon \int_0^T \E{\mathbbm{1}_{\Omega_\infty}\left\|\nabla \tilde{e}(s)\right\|^2}\d s\\
&\qquad+ 2\E{\mathbbm{1}_{\Omega_\infty\cap\Omega_{\tilde{\varepsilon}}\cap \Omega_{\delta}}\left(\sup_{t\in[0,T]}\left\|\hat{e}(t)\right\|_{\Hmonecirc}^2 + \varepsilon \int_0^T \left\|\nabla \hat{e}(s)\right\|^2\d s\right)}\\
&\qquad+ 2\mathbb{P}\left[\Omega_{\tilde{\varepsilon}}^c\cup \Omega_{\delta}^c\right]^{1/2}\E{\mathbbm{1}_{\Omega_\infty}\left(\sup_{t\in[0,T]}\left\|\hat{e}(t)\right\|_{\Hmonecirc}^2 + \varepsilon \int_0^T \left\|\nabla \hat{e}(s)\right\|^2\d s\right)^2}^{1/2}
\\
& =I_1+I_2+I_3.
\end{align*}
To estimate $I_1$ and $I_2$ we can directly use Lemma~\ref{Lemma_CHest_etilde} and Lemma~\ref{Lemma_CHest_echeck}, respectively.

We estimate 
$\mathbb{P}\left[\Omega_{\tilde{\varepsilon}}^c\cup \Omega_{\delta}^c\right]\le \mathbb{P}\left[\Omega_{\tilde{\varepsilon}}^c\right]+\mathbb{P}\left[ \Omega_{\delta}^c\right]$ 
and get by Markov's inequality
\begin{align*}
 \mathbb{P}\left[\Omega_{\tilde{\varepsilon}}^c\right] & = \mathbb{P}\left[\sup_{t\in[0,T]}\|\tilde{e}(t)\|_{\Hmonecirc}^2+\varepsilon\int_0^T\|\nabla \tilde{e}(s)\|^2\d s> \tilde{\varepsilon}\right]\\
& \le \frac{1}{\tilde{\varepsilon}} \E{\sup_{t\in[0,T]}\|\tilde{e}(t)\|_{\Hmonecirc}^2+\varepsilon\int_0^T\|\nabla \tilde{e}(s)\|^2\d s}.
\end{align*}
Analogically, we obtain
\begin{align*}
\mathbb{P}\left[ \Omega_{\delta}^c\right] & = \mathbb{P}\left[\sup_{t\in (0,T)}\varepsilon^3\big(\|u(t)\|_{\Hone}^6+\|u(t)\big)\|^2)> \tilde{\eps}^{-\delta}\right]\\
& \le \tilde{\eps}^{\delta} \E{\sup_{t\in (0,T)}\varepsilon^3\big(\|u(t)\|_{\Hone}^6+\|u(t)\|^2\big)}.
\end{align*}
We respectively use  Lemma \ref{Lemma_CHest_etilde} and the higher-moment bound
$\displaystyle \E{\sup_{t\in(0,T)}\varepsilon^3\big(\|u(t)\|_{\Hone}^6+\|u(t)\|^2)}\le C_0$ (which follows from the energy bound \cite[Lemma 2.1 ii)]{Banas19}) 
to bound the right-hands sides in the last two inequalities.
Consequently, we estimate $I_3$ using Lemma~\ref{Prop_boundedness_checke} as
\begin{align*}
I_3 
& \leq \E{\mathbbm{1}_{\Omega_\infty}\left(\sup_{t\in[0,T]}\left\|\hat{e}(t)\right\|_{\Hmonecirc}^2 + \varepsilon \int_0^T \left\|\nabla \hat{e}(s)\right\|^2\d s\right)^2}^{1/2}
\Big(\tilde{\eps}^{-1}\mathbb{E}[\tilde{\mathcal{R}}] + \tilde{\eps}^{1/4}C_0\Big)^{1/2}
\\
& \leq \hat{C}_0^{1/2}\Big(\tilde{\eps}^{-1}\mathbb{E}[\tilde{\mathcal{R}}] + \tilde{\eps}^{\delta}C_0\Big)^{1/2}\,.
\end{align*}
The statement then follows after collecting the bounds for $I_1$, $I_2$, $I_3$.
\end{proof}

\section{Numerical experiments}\label{Sec_sCH_NumResults}

In the experiments below we use a Monte-Carlo approach to solve the discrete stochastic system (\ref{StochCH_discr});
i.e., (\ref{StochCH_discr}) is solved pathwise using several indepenedent realizations of the noise term.
For a given realization of the noise, the nonlinear system for the solution $u_h^n$ is solved using the Newton method.
The size of the time steps $\tau_n$ is chosen adaptively according to the number of iteration of the Newton solver at the previous time level (for the tolerance of the Newton residual $5\times 10^{-9}$):
if the number of iteration is below $5$ the time-step is chosen as $\tau_n = 2\tau_{n-1}$,
if the number of iterations exceeds $50$ we set $\tau_n = 0.5\tau_{n-1}$ otherwise we take $\tau_{n}=\tau_{n-1}$. On average the Newton solver finished after about $10$ iterations and always stayed below
$50$ iterations, the resulting time-step size in the computations was of order $10^{-6}$.
The spatial mesh on the current time level is obtained by local refinement and coarsening of the mesh from the previous time level 
based on the local contributions of the spatial error indicator $\eta_{\mathrm{SPACE},3}^n$ 
from Remark~\ref{Rem_estRySy} until a tolerance $\mathrm{tol}=10^{-2}$ is reached, cf. \cite{lb09}, \cite{BartelsMueller2011}.
We choose to use only the indicator $\eta_{\mathrm{SPACE},3}^n$ for simplicity. In general, the remaining indicators should be involved in the refinement algorithm as well,
nevertheless, based on the results below the indicator seems to provide a reasonable criterion for mesh refinement, cf. \cite{lb09}.
\corrl{For simplicity we do not compute the solution of (\ref{CHstoch_nonlin_discr}) (or (\ref{CH_discr_y})) explicitly but evaluate the error indicators using
the solution $u_{h,\tau}$ of (\ref{StochCH_discr}) {which is the actual solution of practical interest}.
We also neglect the contribution from the noise error indicators $\eta_{\mathrm{NOISE}}$, which seams to be a reasonable simplification for sufficiently smooth noise.
The numerical experiments (for the data given below) indicate that  
the noise term given by (\ref{noise}) can be approximated on a rather coarse mesh, 
and consequently also the solution $\tilde{u}_{h,\tau}$ of (\ref{CHstoch_lin_discr}) is well resolved on a coarse mesh.
Hence, the above simplifications are justified in the present setting.
Furthermore, we do not compute the coarsening error indicators explicitly as this would be impractical.
}

We consider an initial condition on $\D=(-1,1)^2$ given by two concentric circles 
$$
 u_0(x) = -\tanh\left(\frac{\max\{-(|x|-r_1), |x|-r_2\}}{\sqrt{2}\varepsilon}\right)\,,
$$
with $r_1 = 0.2$, $r_2 = 0.55$, see Figure~\ref{fig_u0}.
\begin{figure}[htp!]
\centering
\includegraphics[scale=0.25, trim=20 0 0 0, clip]{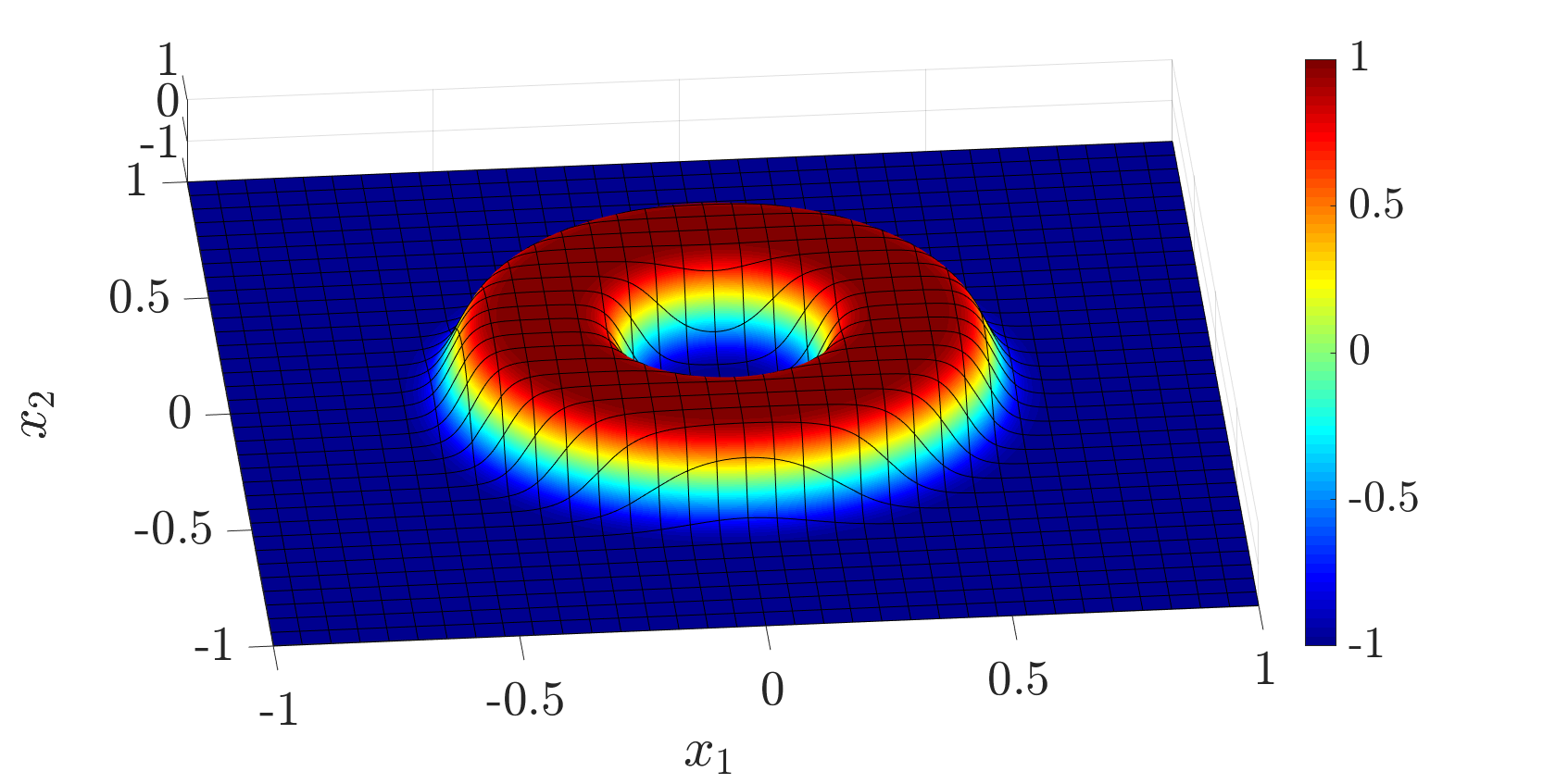}
\caption{Initial condition for $\eps=1/32$.}
\label{fig_u0}
\end{figure}

The remaining parameters in the simulation were $T= 0.012$, $\sigma \equiv 1$, {\red $\eps=1/32$}.
We employ a finite-dimensional Wiener process 
\begin{equation}\label{noise}
\Delta_n \WW^r(x) = \frac{1}{2}\sum_{l_1,l_2=1}^{4} \cos(\pi l_1(x_1-1))  \cos(\pi l_2(x_2-1))\Delta_n W_{(l_1,l_2)}\,,
\end{equation}
with Brownian increments $\Delta_n W_{(l_1,l_2)}$.
In the figures below the mesh is colored according to the values of the corresponding numerical solution, unless it is displayed in black.

In the deterministic setting the solution with the above initial condition evolves as follows:
both circles shrink until the inner circle disappear and the solution converges to a steady state which is represented by one circular interface.
The closing of the inner circle represents a topological change of the interface which is represented by the peak 
of the principal eigenvalue, cf. Figure~\ref{Fig_CHstoch_Lambda}.
We observe that, apart from small oscillations due to the influence of the noise,
the evolution of the stochastic solution is similar to the deterministic case, see Figure \ref{Fig_CHstoch_uh}.
\begin{center}\begin{figure}[htp]
\begin{minipage}{0.48\textwidth}
\begin{figure}[H]
\centering
\includegraphics[scale=0.32, trim=45 0 50 20, clip]{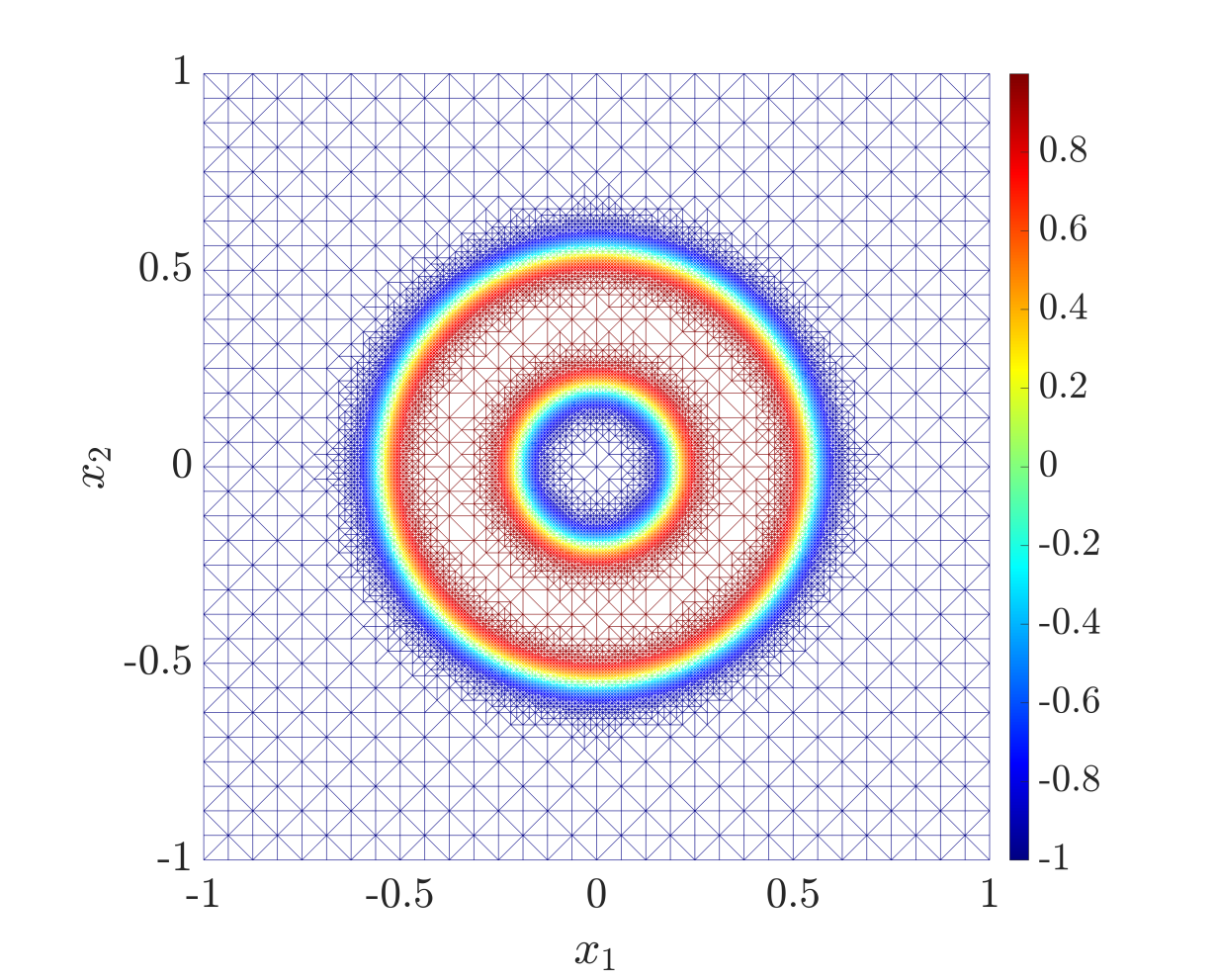}
\end{figure}
\end{minipage}
\hfill
\begin{minipage}{0.48\textwidth}
\begin{figure}[H]
\centering
\includegraphics[scale=0.32, trim=45 0 50 20, clip]{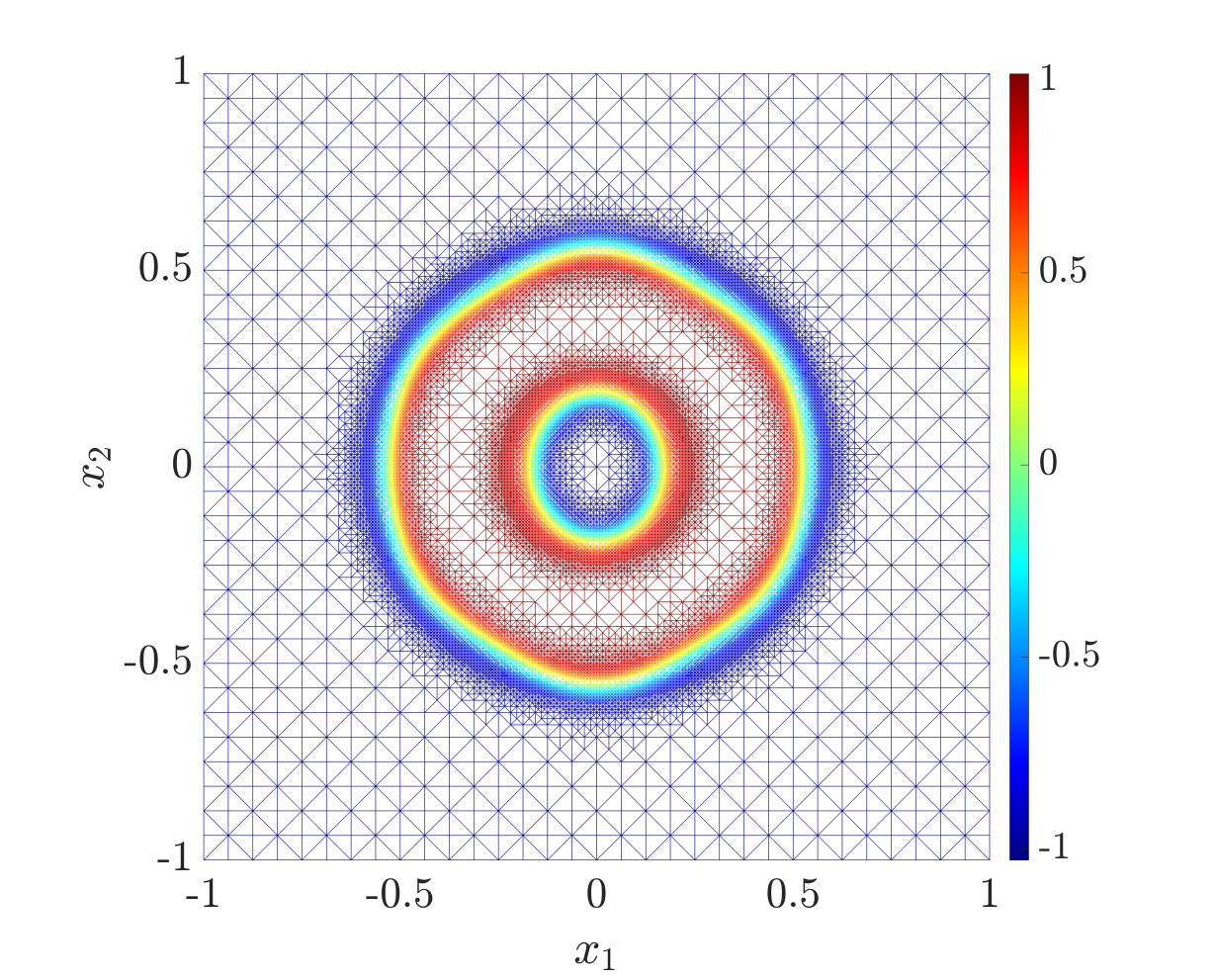}
\end{figure}
\end{minipage}\\
\begin{minipage}{0.48\textwidth}
\begin{figure}[H]
\centering
\includegraphics[scale=0.32, trim=45 0 50 20, clip]{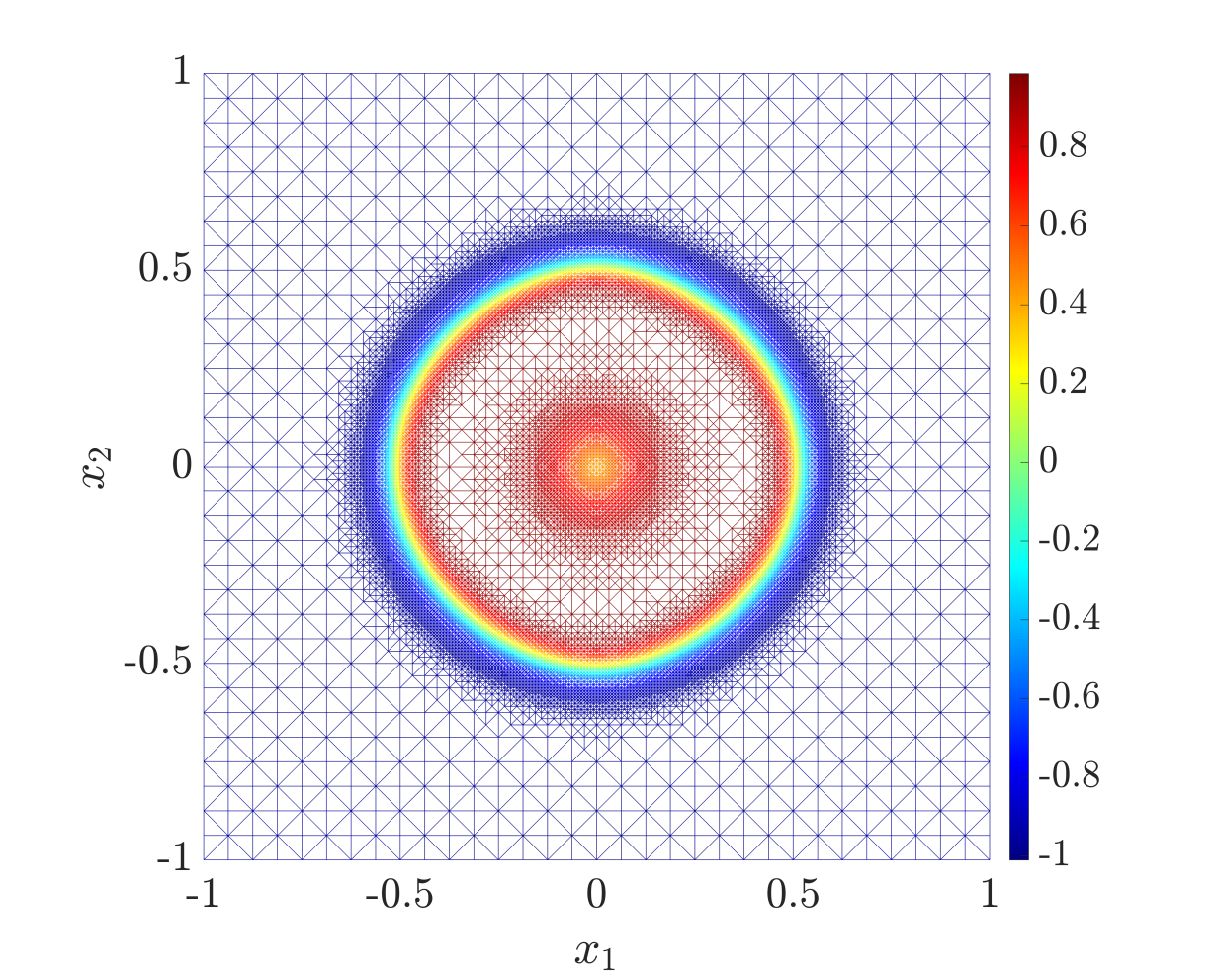}
\end{figure}
\end{minipage}
\hfill
\begin{minipage}{0.48\textwidth}
\begin{figure}[H]
\centering
\includegraphics[scale=0.32, trim=45 0 50 20, clip]{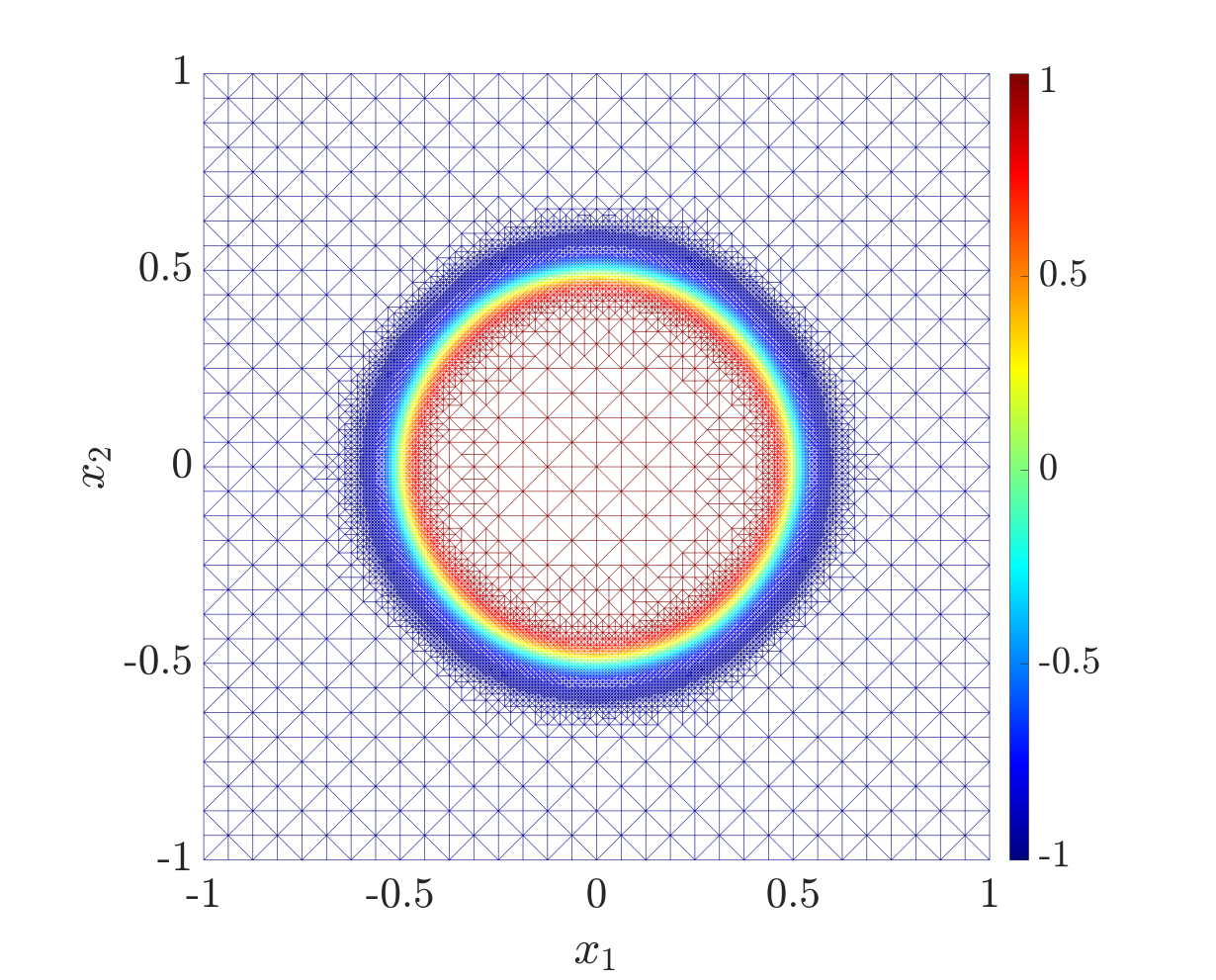}
\end{figure}
\end{minipage}
\caption{Snapshots of the mesh $\mathcal{T}^n_{h}$ (colored according to the values of the numerical 
solution $u_{h,\tau}$) at $t = 0, 0.003, 0.009, 0.012$ (from top left to bottom right).}
\label{Fig_CHstoch_uh}
\end{figure}
\end{center}
The evolution of the principal eigenvalue for two realizations of the noise in Figure~\ref{Fig_CHstoch_Lambda}
indicates that (apart from the oscillations) the overall evolution of the principal eigenvalue under the influence of the noise remains similar
to the deterministic case. 
The histogram for the occurrence of the topological change computed with $500$ realizations of the noise is displayed in Figure \ref{Fig_CHstoch_Hist}.
We observe that the occurrence of the topological change varies about its deterministic counterpart; the probability of the occurrence of topological change 
peaks close to the deterministic case. 
Furthermore, qualitatively the number of degrees of freedom evolves similarly as in the deterministic setting, see Figure~\ref{Fig_CHstoch_DOF1}.
\begin{center}
\begin{figure}[htp]
\begin{minipage}{0.5\textwidth}
\begin{figure}[H]
\centering
\includegraphics[scale=0.5, trim=0 0 0 20, clip]{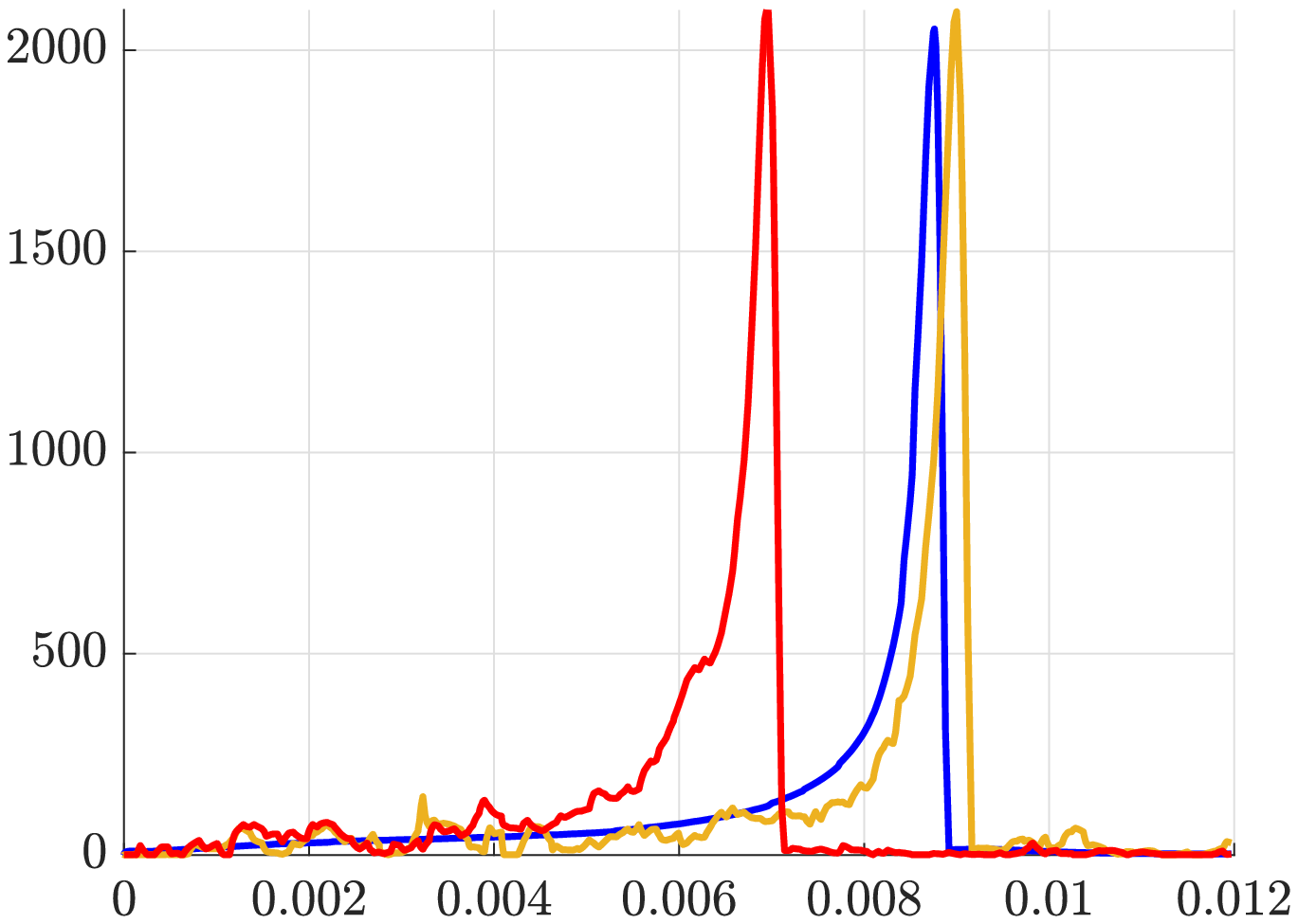}
\caption{Evolution of the principal eigenvalue: deterministic solution (blue) and two paths of the stochastic solution (red, orange).}
\label{Fig_CHstoch_Lambda}
\end{figure}
\end{minipage}\hfill
\begin{minipage}{0.5\textwidth}
\begin{figure}[H]
\centering
\includegraphics[scale=0.5, trim=0 0 0 15, clip]{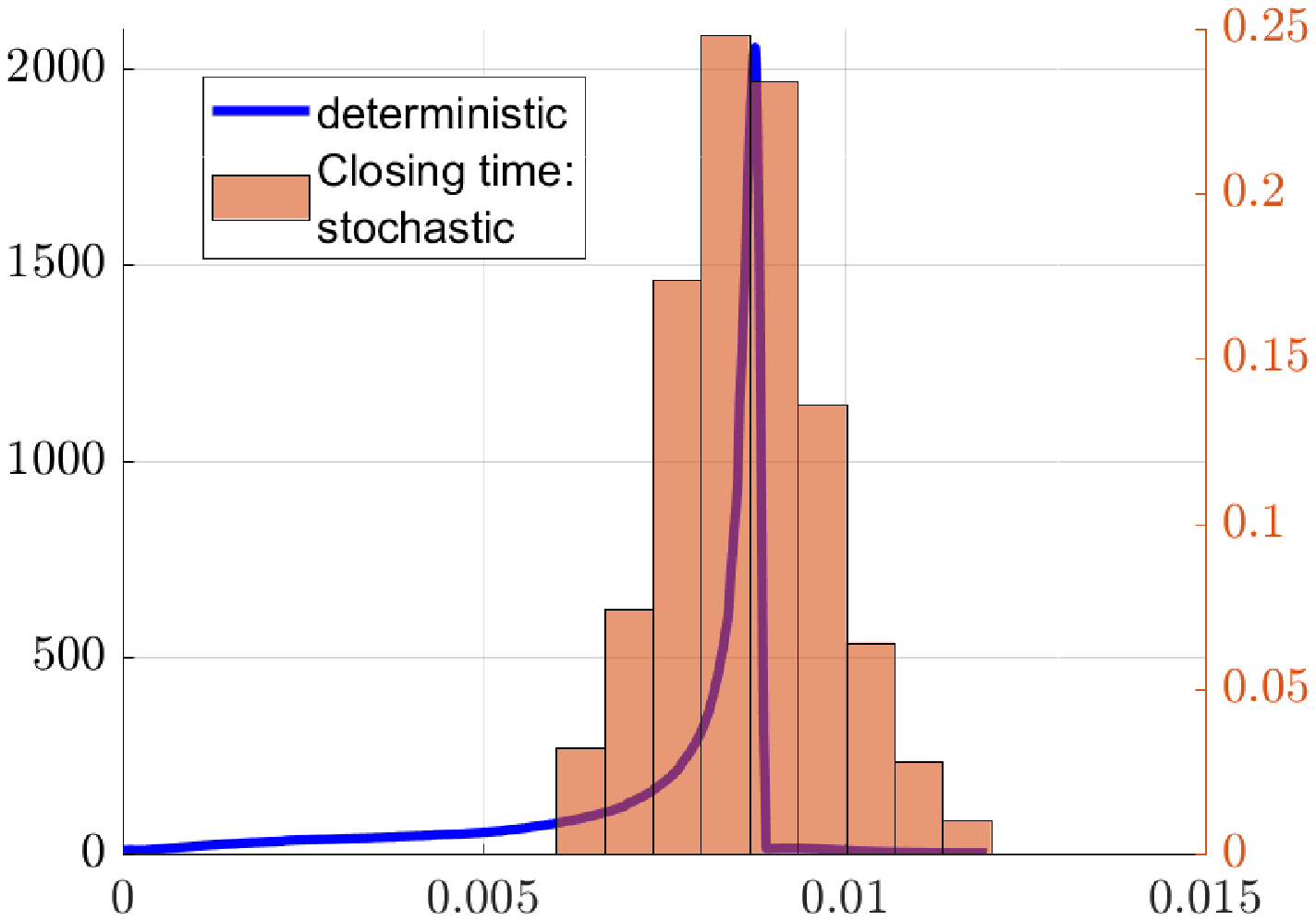}
\caption{Histogram of the occurrence of the (stochastic) closing time (scale on the right axis) 
and the evolution of the deterministic principal eigenvalue (blue line, scale on left axis).}
\label{Fig_CHstoch_Hist}
\end{figure}
\end{minipage}
\end{figure}
\end{center}

\begin{center}
\begin{figure}[htp]
\begin{minipage}{0.5\textwidth}
\begin{figure}[H]
\centering
\includegraphics[width=0.9\textwidth]{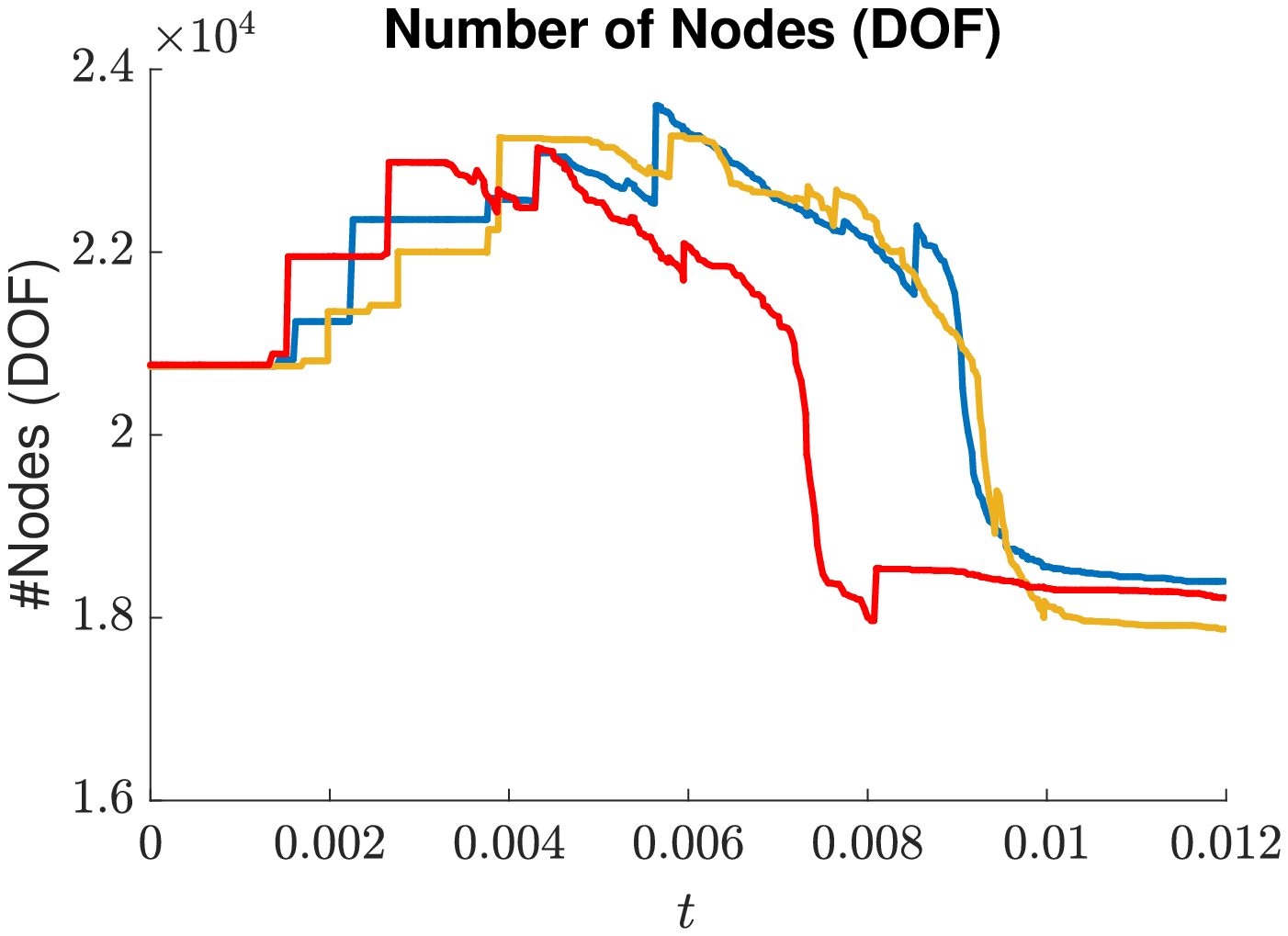}
\caption{Degrees of freedom for the approximation $u_{h,\tau}$, deterministic (blue), two paths of the stochastic solution in (red, orange).}
\label{Fig_CHstoch_DOF1}
\end{figure}
\end{minipage}\hfill
\begin{minipage}{0.5\textwidth}
\begin{figure}[H]
\centering
\includegraphics[width=0.9\textwidth]{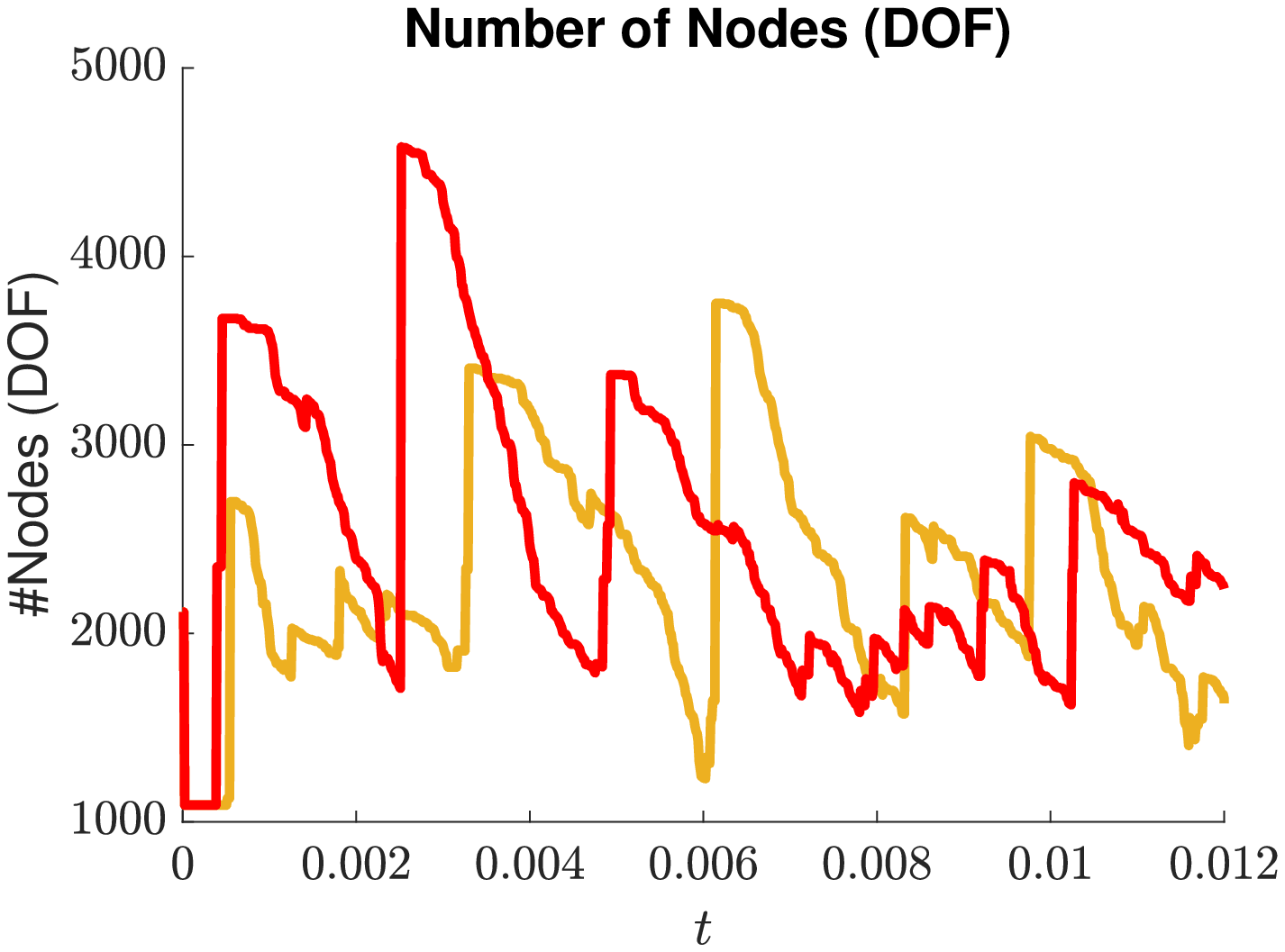}
\caption{Degrees off freedom for the approximation $\tilde{u}_{h,\tau}$.}
\label{Fig_CHstoch_DOF2}
\end{figure}
\end{minipage}
\end{figure}
\end{center}
In Figure~\ref{Fig_CHstoch_tildeuh} we show the time-evolution of the numerical approximation $\tilde{u}_{h,\tau}$ of the linear SPDE (\ref{CHstoch_lin}) for this scenario;  
the corresponding mesh is displayed in Figure~\ref{Fig_CHstoch_tildeuh_mesh} 
and in Figure~\ref{Fig_CHstoch_eta} we display the values of the corresponding spatial error indicator 
\corrl{$\eta_{\textrm{SPACE},3} \equiv \eta_{\textrm{SPACE},3}(\tilde{u}_{h,\tau})$} at different time levels.
The meshes were constructed by local mesh refinement and coarsening until the error indicator $\eta_{\textrm{SPACE},3}(\tilde{u}_{h,\tau})$
is below the tolerances $\mathrm{tol}=10^{-2}$;
the time evolution of the degrees of freedom for the approximation is displayed in Figure~\ref{Fig_CHstoch_tildeuh_mesh}.
We observe, that for the same tolerance, the numerical solution $\tilde{u}_{h,\tau}$ requires much less degrees of freedom than the numerical solution 
${u}_{h,\tau}$, see  Figures \ref{Fig_CHstoch_tildeuh_mesh} and \ref{Fig_CHstoch_DOF2}
and Figures \ref{Fig_CHstoch_uh} and \ref{Fig_CHstoch_DOF1} respectively.
\begin{center}\begin{figure}[htp]
\begin{minipage}{0.48\textwidth}
\begin{figure}[H]
\centering
\includegraphics[scale=0.48, trim=0 35 0 50, clip]{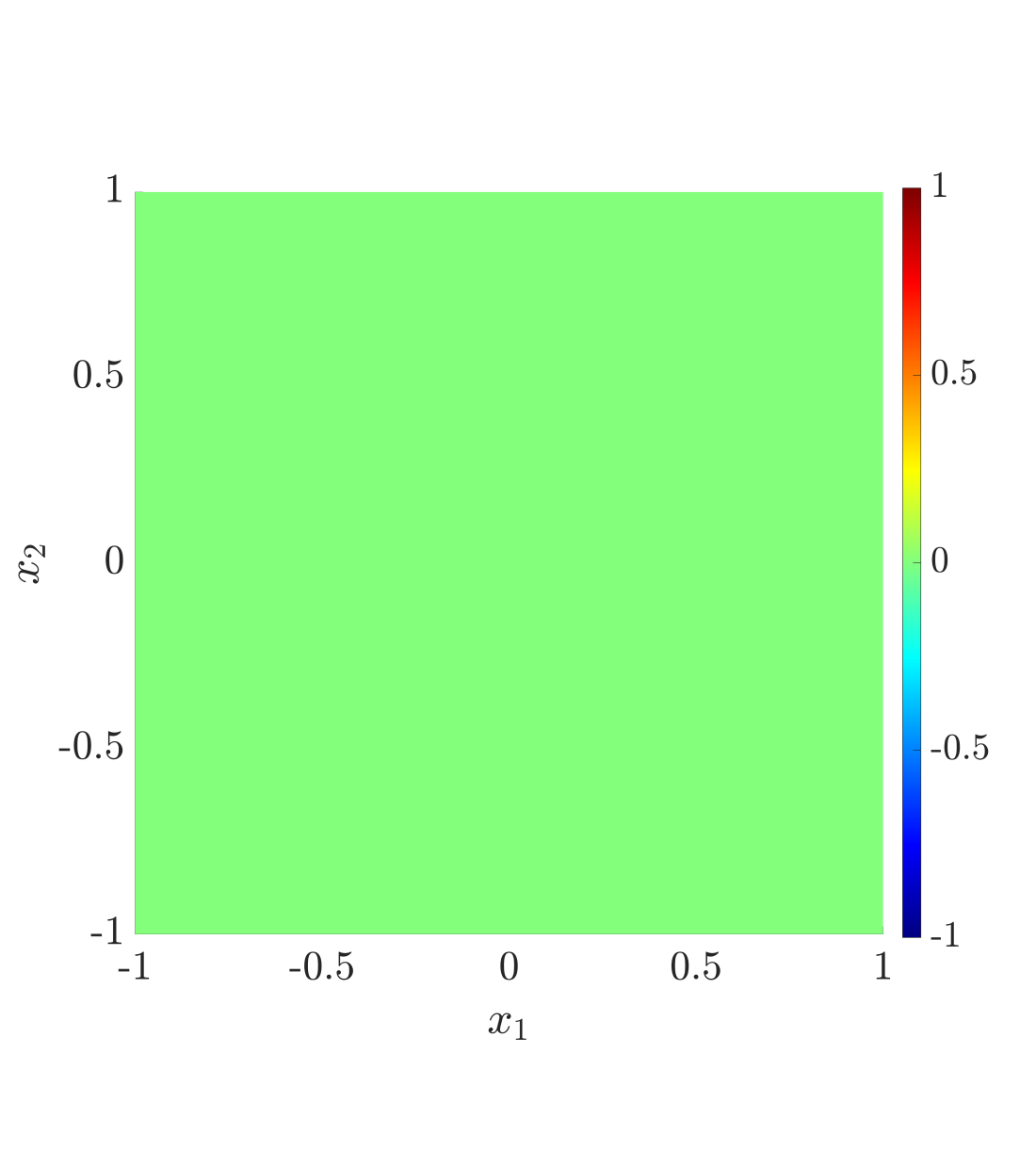}
\end{figure}
\end{minipage}
\hfill
\begin{minipage}{0.48\textwidth}
\begin{figure}[H]
\centering
\includegraphics[scale=0.48, trim=0 35 0 50, clip]{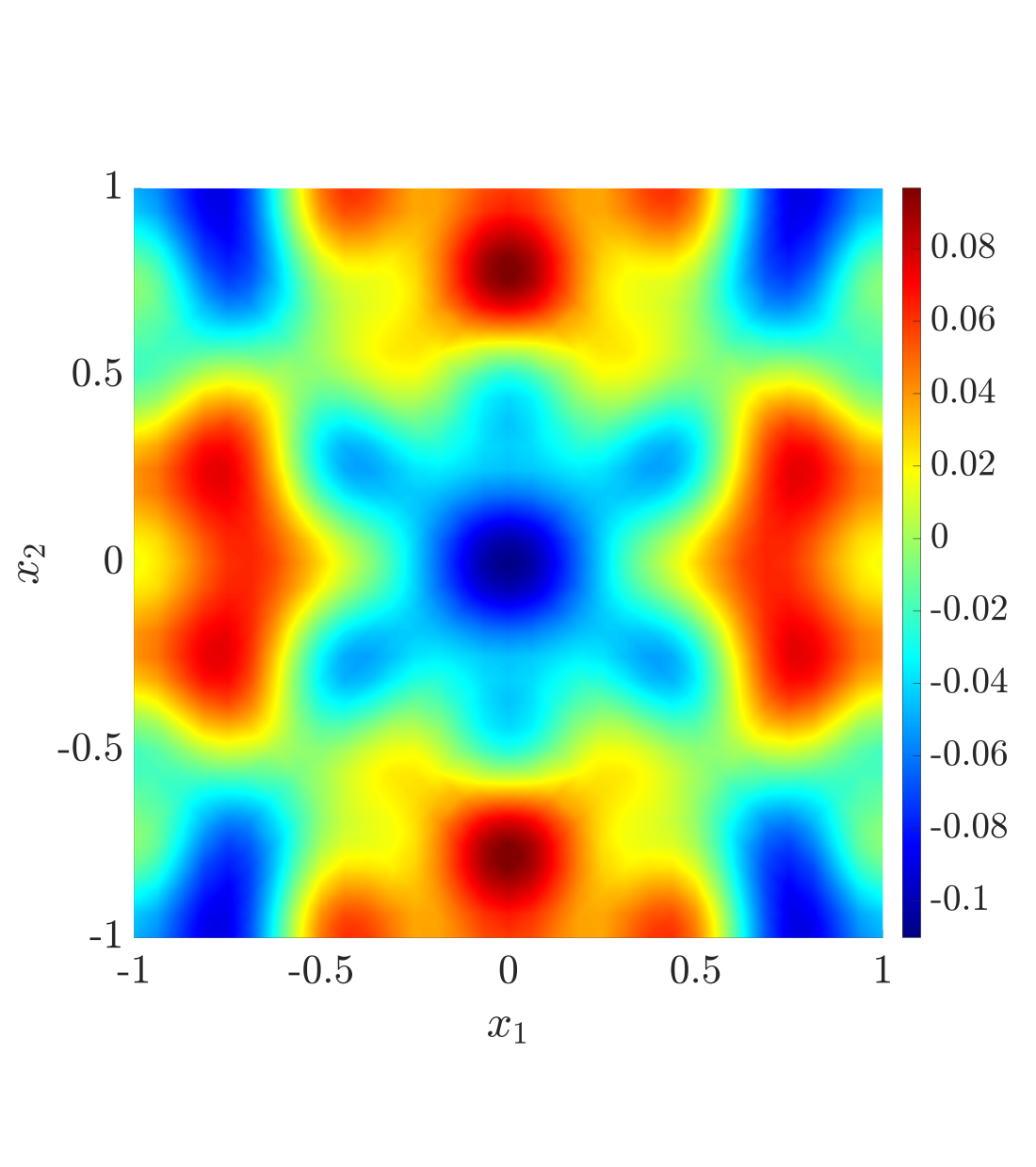}
\end{figure}
\end{minipage}
\begin{minipage}{0.48\textwidth}
\begin{figure}[H]
\centering
\includegraphics[scale=0.48, trim=0 35 0 50, clip]{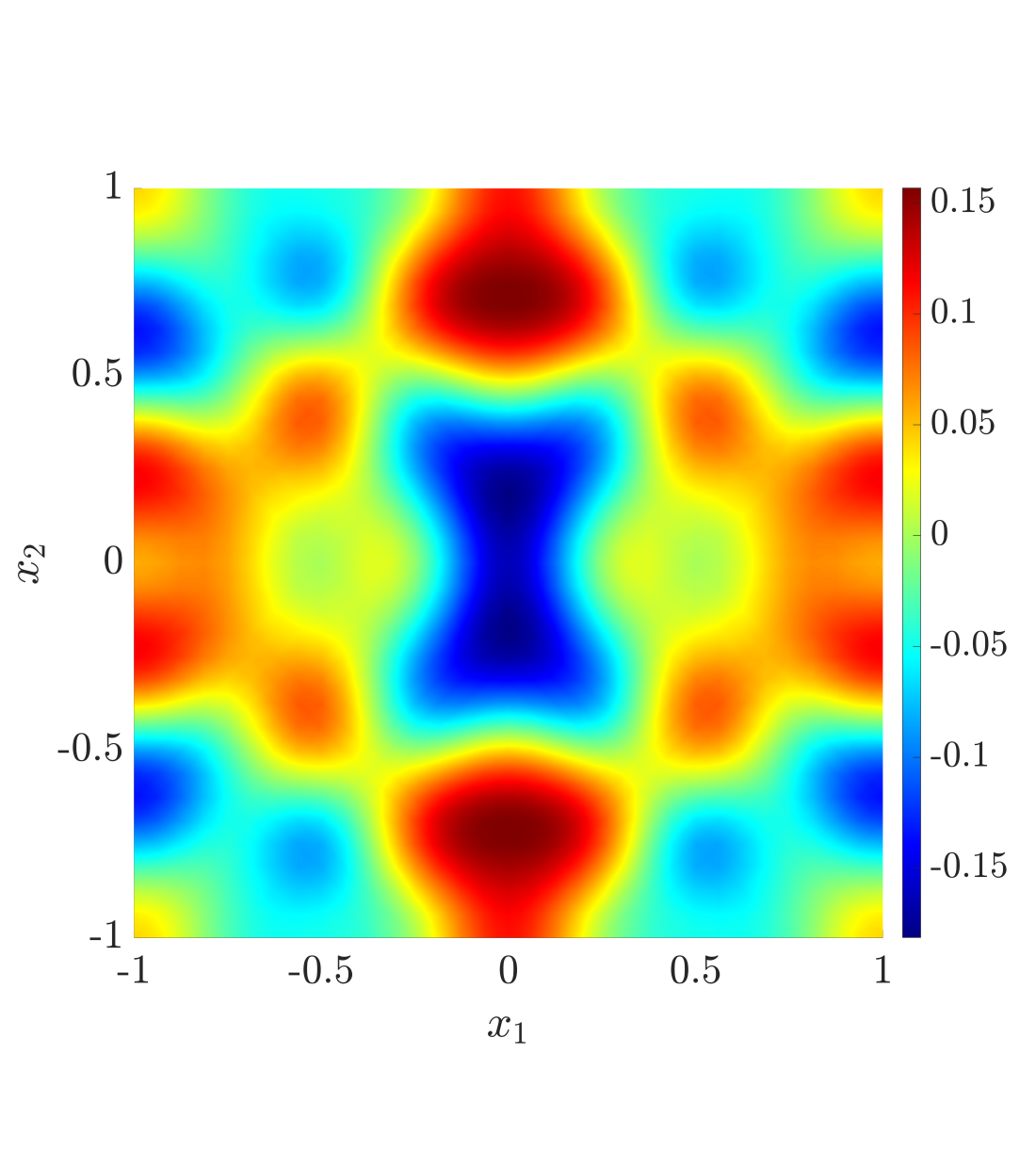}
\end{figure}
\end{minipage}
\hfill
\begin{minipage}{0.48\textwidth}
\begin{figure}[H]
\centering
\includegraphics[scale=0.48, trim=0 35 0 50, clip]{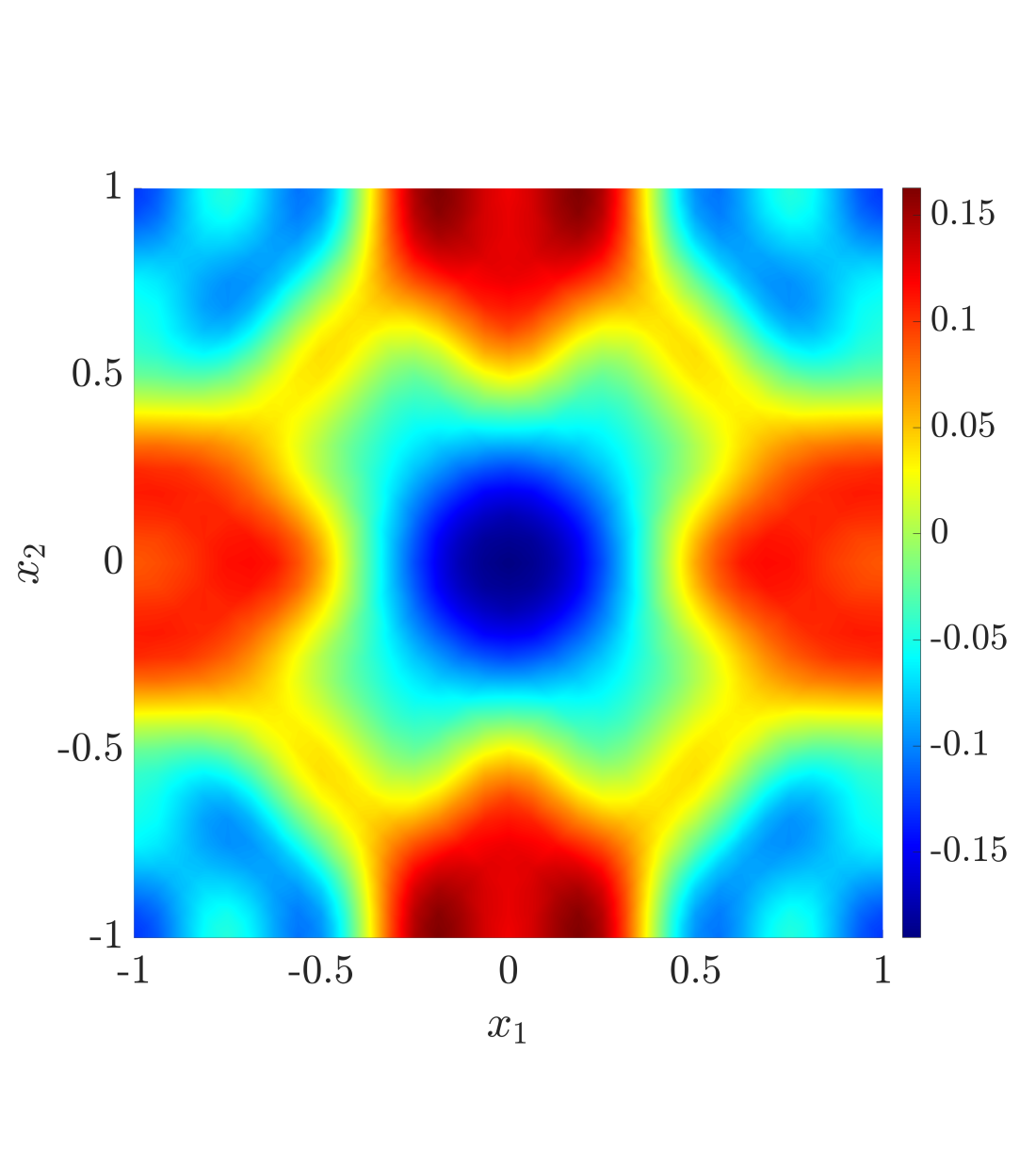}
\end{figure}
\end{minipage}
\caption{Snapshots of the numerical solution $\tilde{u}_{h,\tau}$  at $t = 0, 0.003, 0.009, 0.012$.}
\label{Fig_CHstoch_tildeuh}
\end{figure}
\end{center}
\begin{center}\begin{figure}[htp]
\begin{minipage}{0.48\textwidth}
\begin{figure}[H]
\centering
\includegraphics[scale=0.48, trim=0 35 0 50, clip]{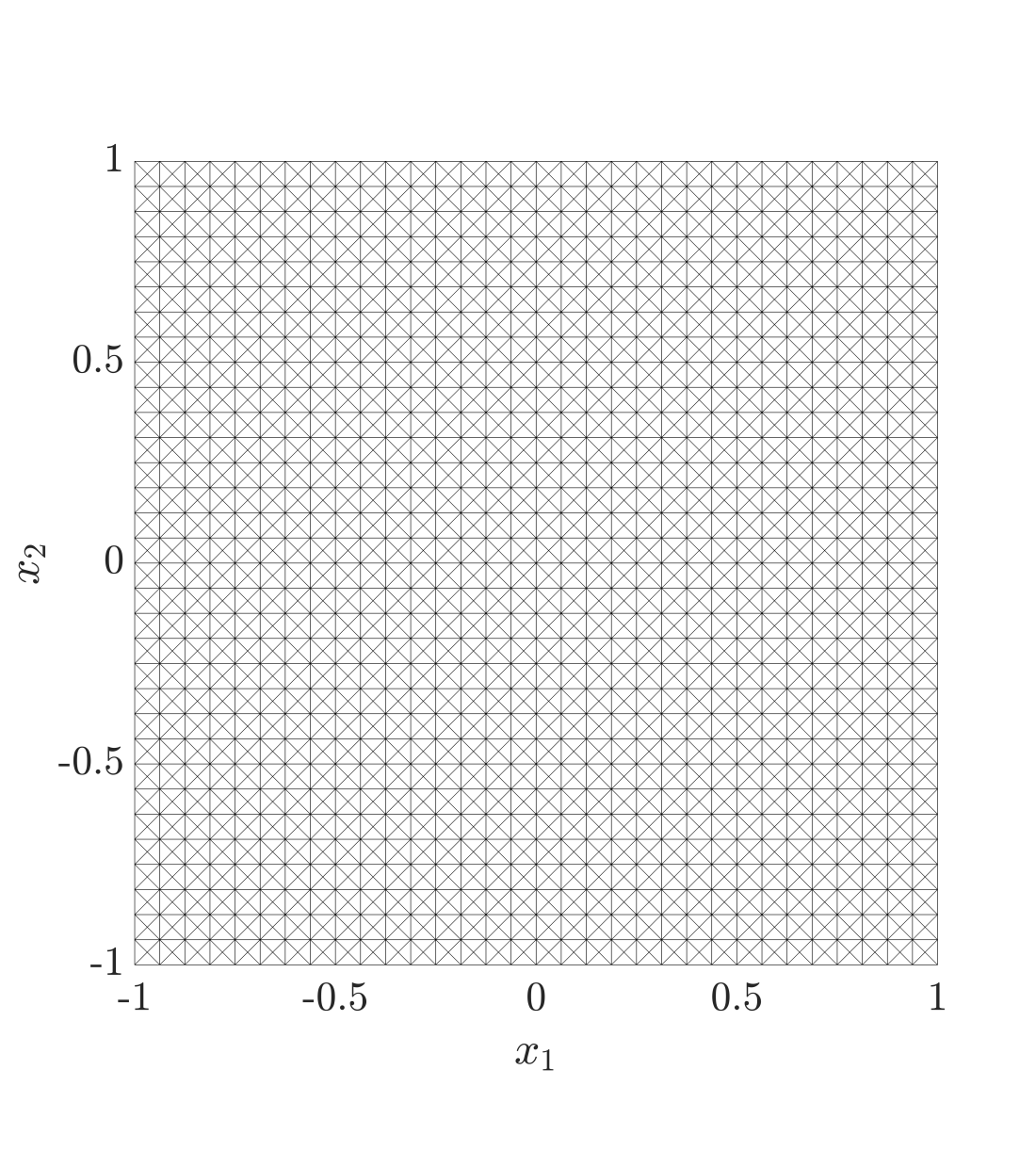}
\end{figure}
\end{minipage}
\hfill
\begin{minipage}{0.48\textwidth}
\begin{figure}[H]
\centering
\includegraphics[scale=0.48, trim=0 35 0 50, clip]{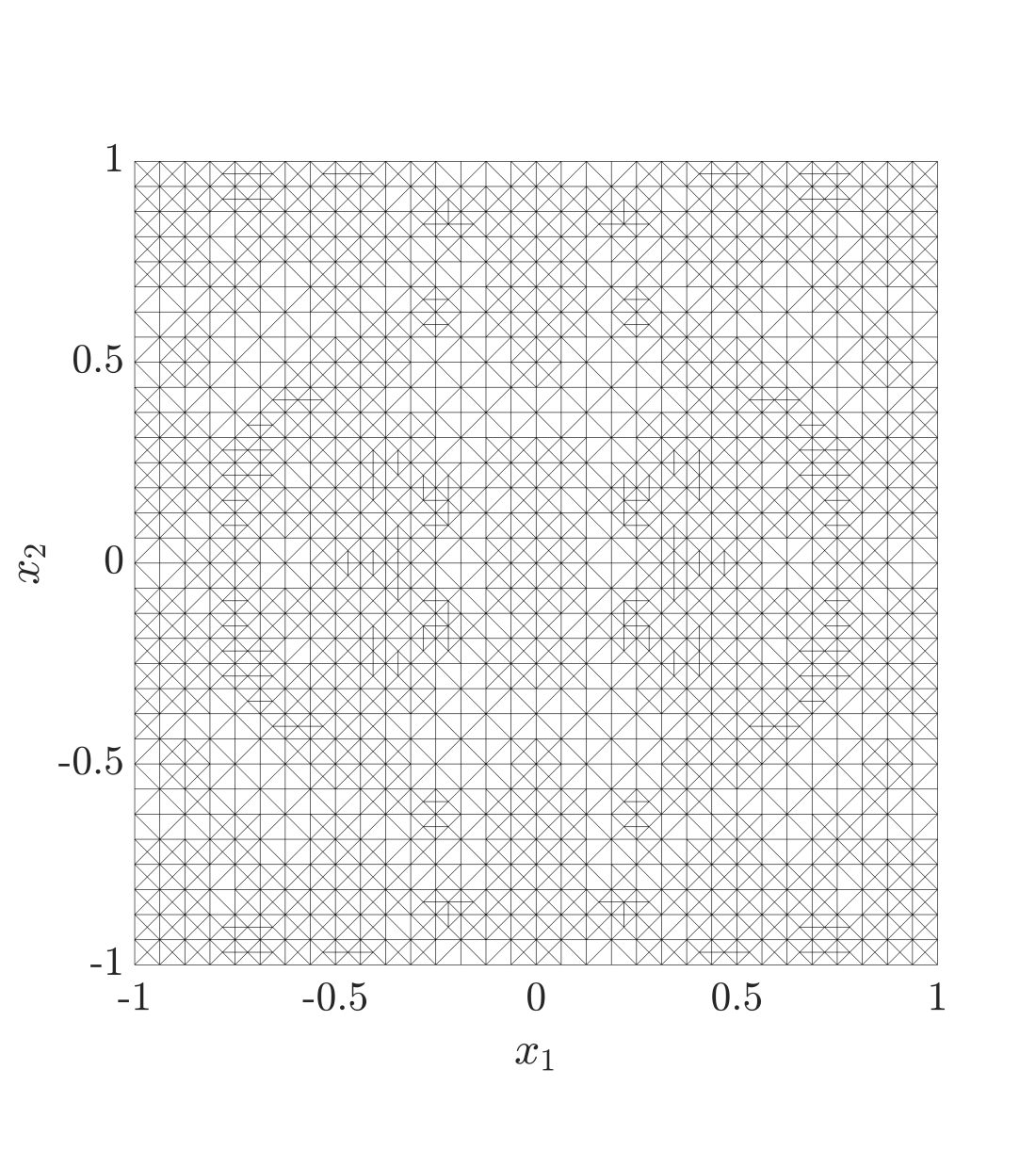}
\end{figure}
\end{minipage}\\
\begin{minipage}{0.48\textwidth}
\begin{figure}[H]
\centering
\includegraphics[scale=0.48, trim=0 35 0 50, clip]{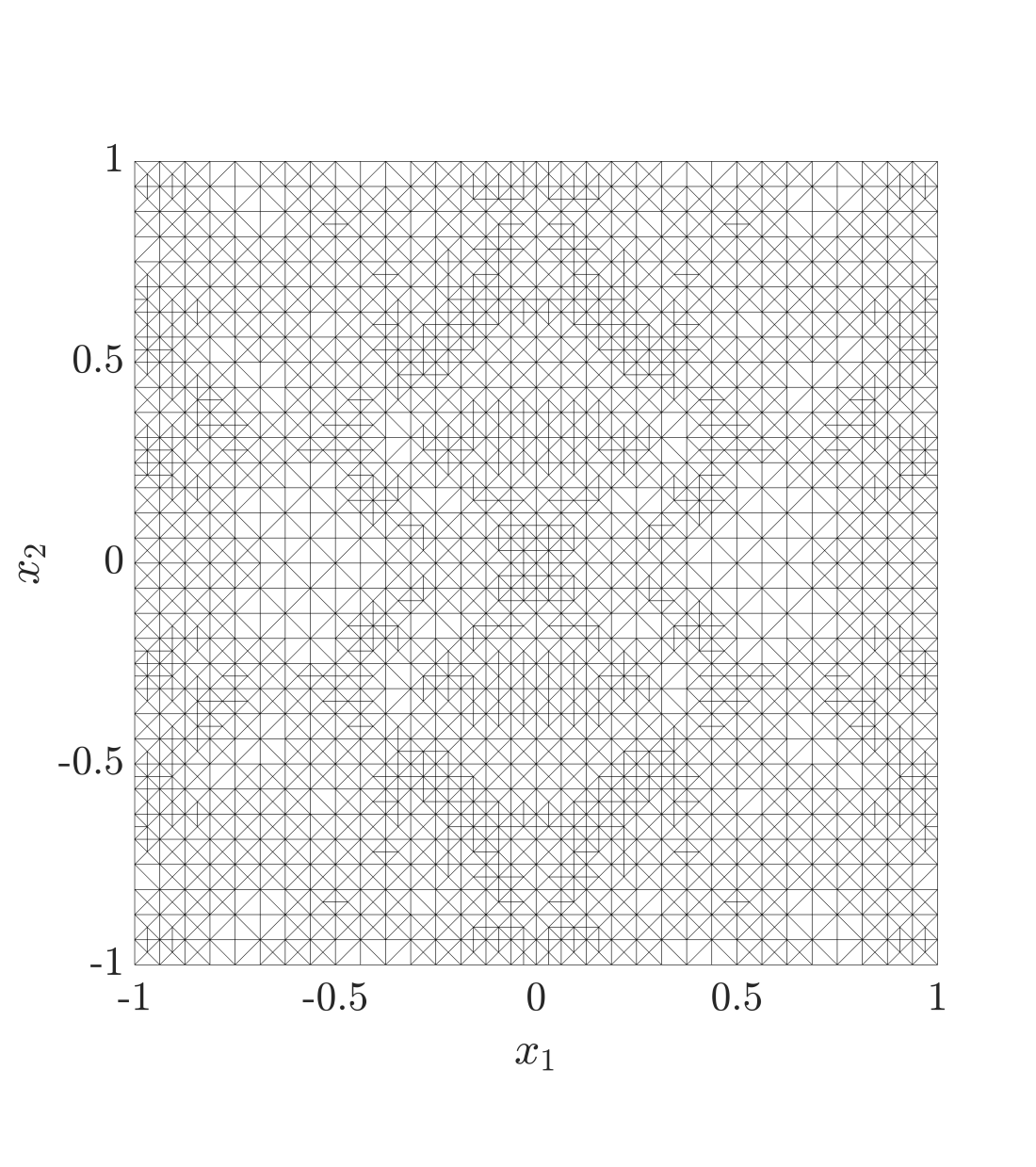}
\end{figure}
\end{minipage}
\hfill
\begin{minipage}{0.48\textwidth}
\begin{figure}[H]
\centering
\includegraphics[scale=0.48, trim=0 35 0 50, clip]{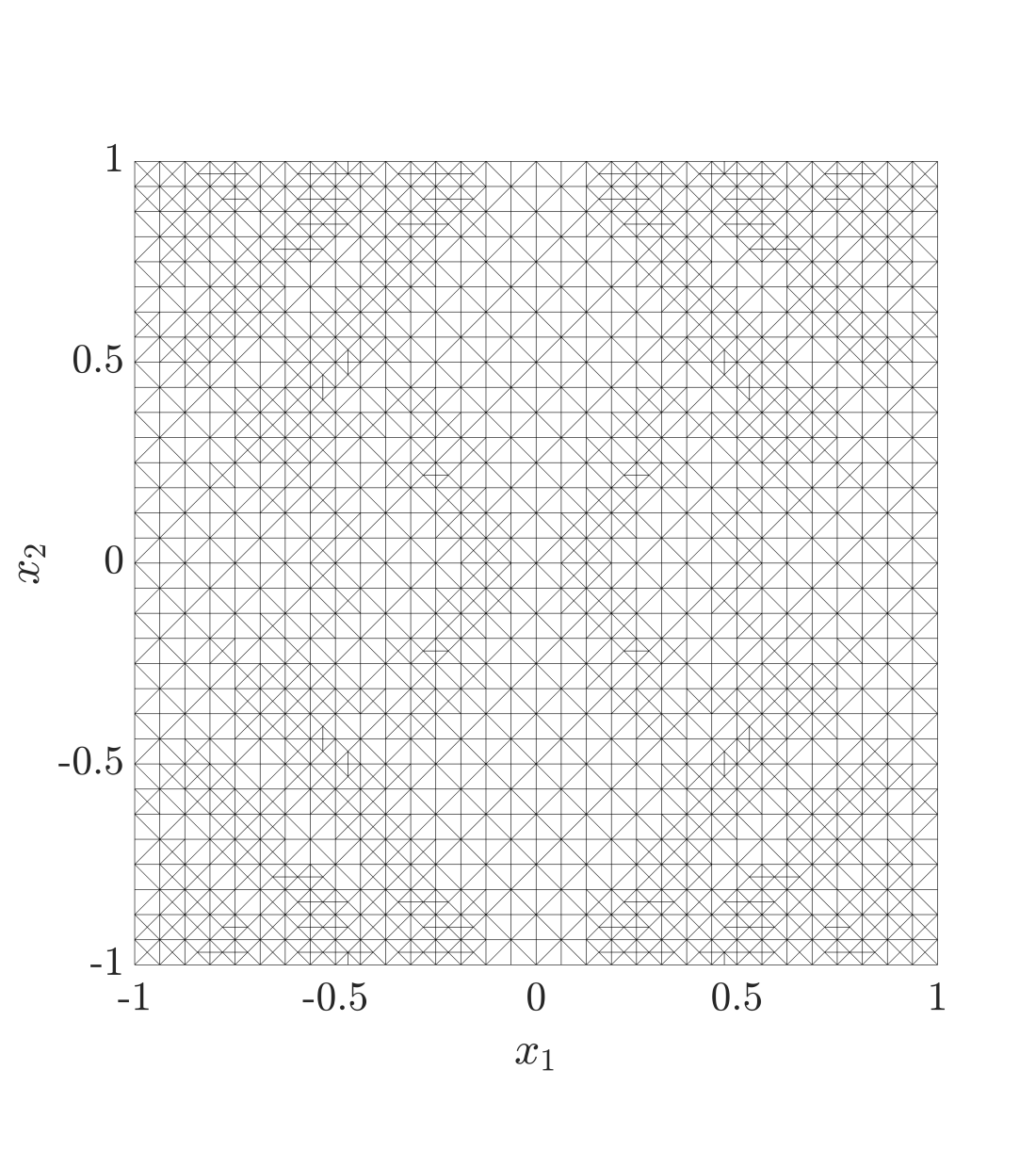}
\end{figure}
\end{minipage}
\caption{Snapshots of the mesh for the numerical solution $\tilde{u}_{h,\tau}$ at $t = 0, 0.003, 0.009, 0.012$.}
\label{Fig_CHstoch_tildeuh_mesh}
\end{figure}
\end{center}
\begin{center}\begin{figure}[htp]
\begin{minipage}{0.32\textwidth}
\begin{figure}[H]
\centering
\includegraphics[scale=0.15, trim=40 50 0 20, clip]{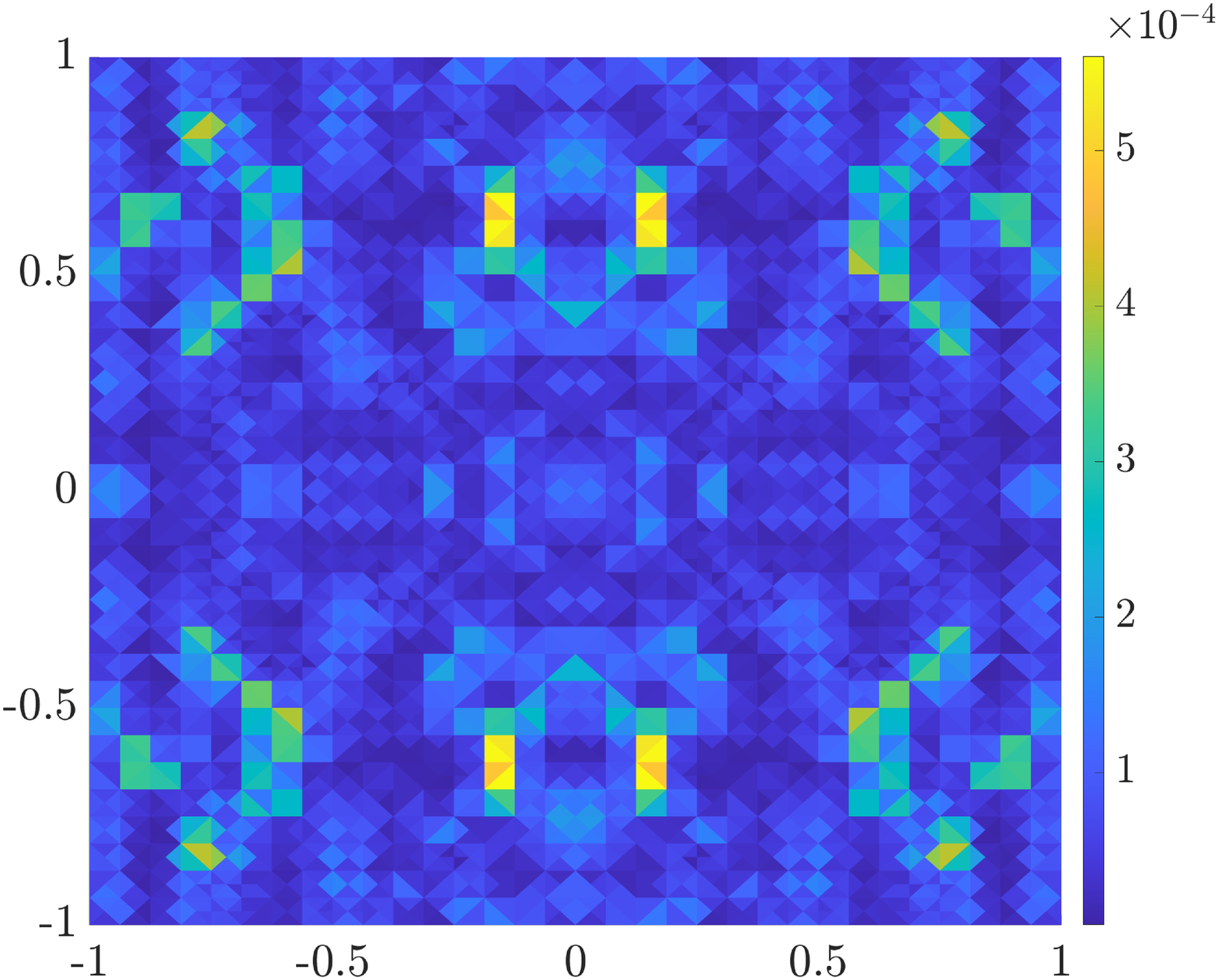}
\end{figure}
\end{minipage}
\hfill
\begin{minipage}{0.32\textwidth}
\begin{figure}[H]
\centering
\includegraphics[scale=0.15, trim=40 50 0 20, clip]{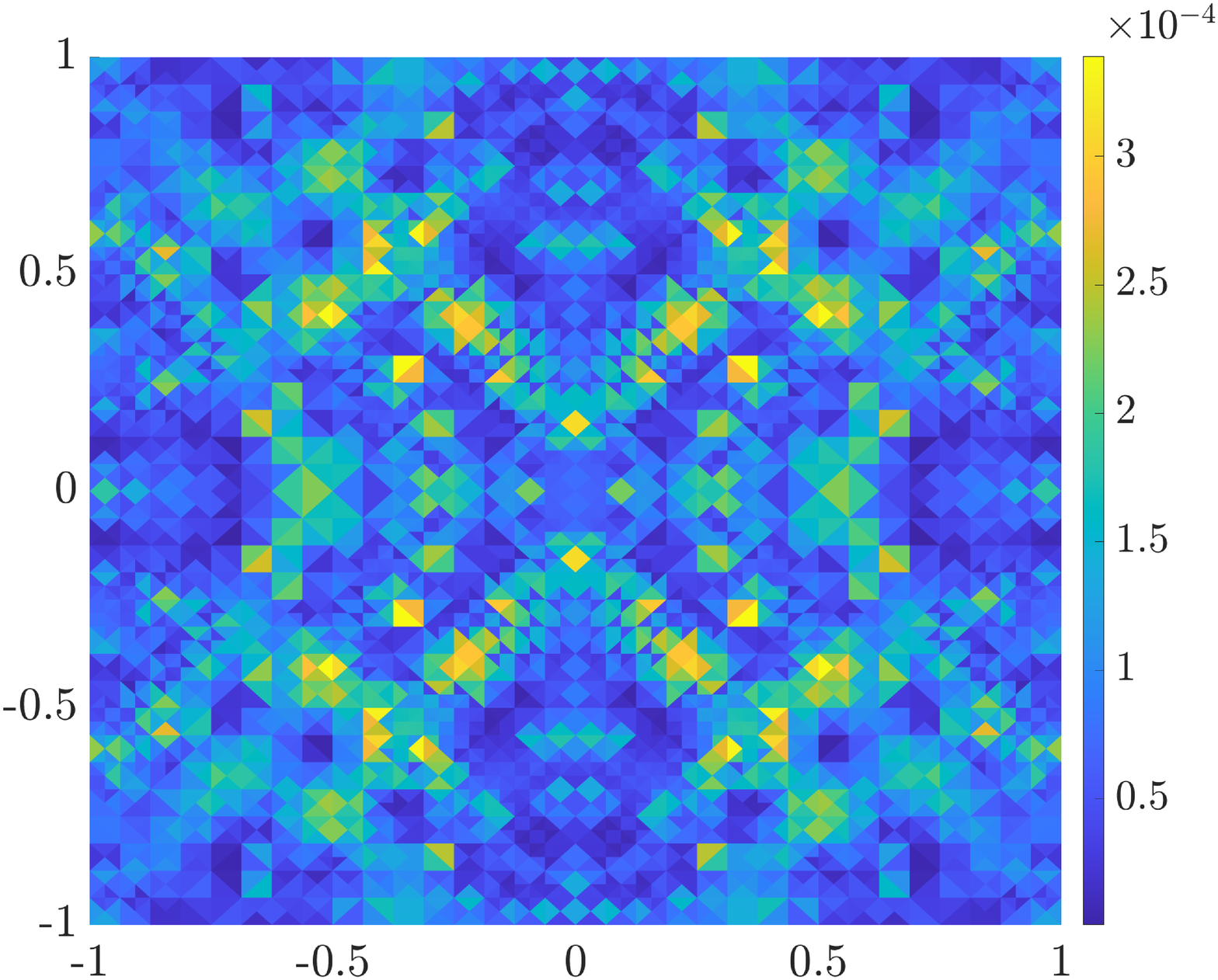}
\end{figure}
\end{minipage}
\hfill
\begin{minipage}{0.32\textwidth}
\begin{figure}[H]
\centering
\includegraphics[scale=0.15, trim=40 50 0 20, clip]{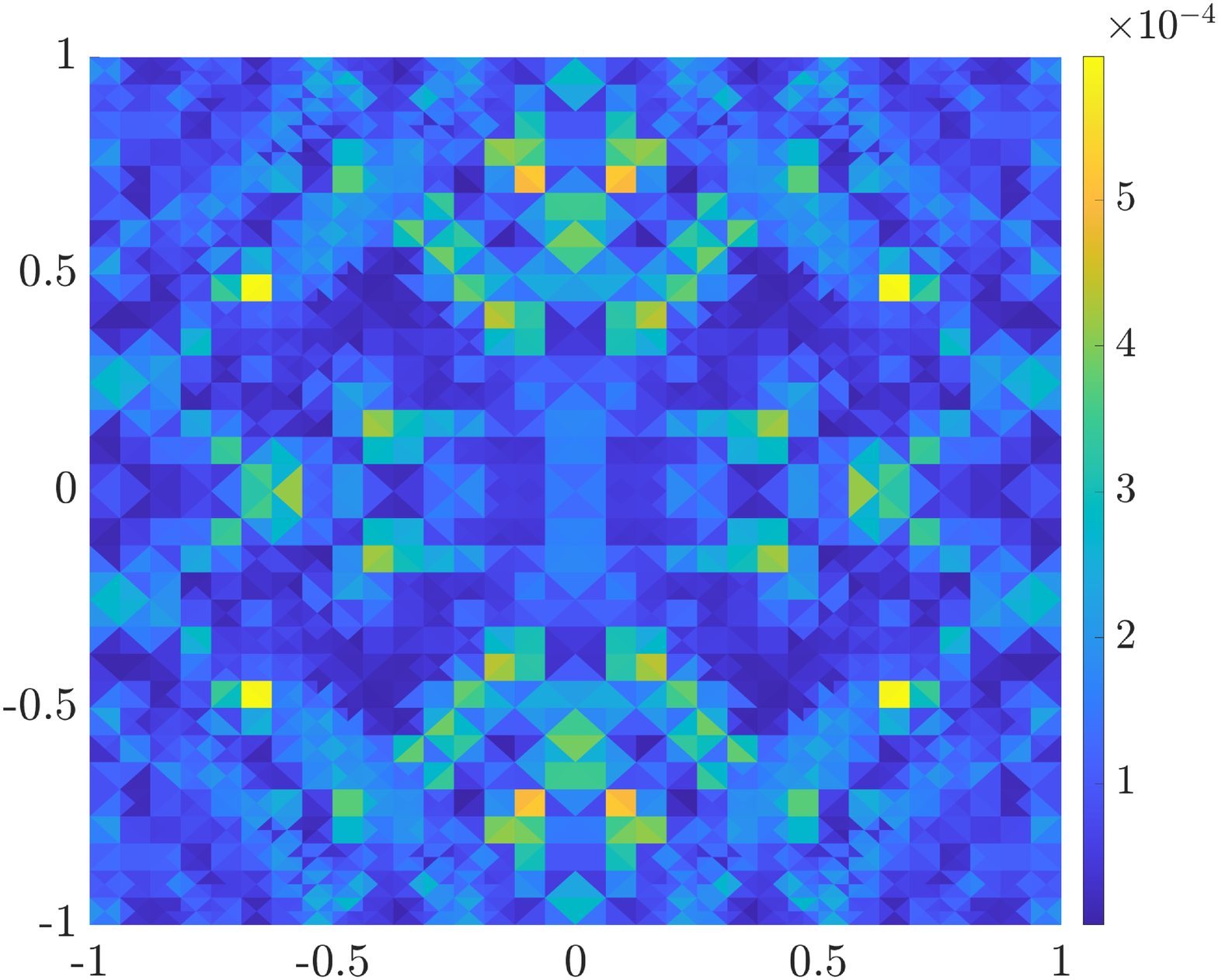}
\end{figure}
\end{minipage}
\caption{Spacial-error estimator for $\tilde{u}_{h,\tau}$ at $t = 0.003, 0.009, 0.012$.}
\label{Fig_CHstoch_eta}
\end{figure}
\end{center}


To examine the effects of stronger noise we consider the above initial condition with $r_1 = 0.3$, $r_2 = 0.45$
and compute the problem with the noise intensity $\sigma=5$.
We observe that the stochastic problem Figure \ref{Fig_CHstoch_detstoch_2_mesh} (right) exhibits a completely different evolution than its deterministic counterpart in Figure \ref{Fig_CHstoch_detstoch_2_mesh} (left); 
nevertheless the adaptive mesh refinement algorithm correctly captures the position of the interface in both cases.
In Figure \ref{Fig_CHstoch_2_Lambda} we display the evolution
of the principal eigenvalue of the numerical solution  and indicate the peaks of its value for the stochastic case by dotted vertical lines. 
By examining the numerical solution in Figure~\ref{Fig_CHstoch_stoch_2_peak_mesh} which corresponds to
the peaks of the principal eigenvalue in Figure~\ref{Fig_CHstoch_2_Lambda},
one may deduce that the peaks occur at times where the interface undergoes (or is close to) a topological change.
This is in line with the estimate in  Lemma~\ref{Lemma_CHest_echeck} which indicates that the largest contributions to the approximation error happen at the peaks of the principal eigenvalue (i.e.
when the solution undergoes a topological change). 
Furthermore, the numerical experiments support the conjecture (that goes beyond the known theory, cf. \cite{Banas19}, \cite{BanasZhu2019})
that the principal eigenvalue is a reliable indicator of topological changes of the interface even in the presence of strong noise.


\begin{center}\begin{figure}[htp]
\begin{minipage}{0.48\textwidth}
\begin{figure}[H]
\centering
\includegraphics[scale=0.35, trim=0 35 0 50, clip]{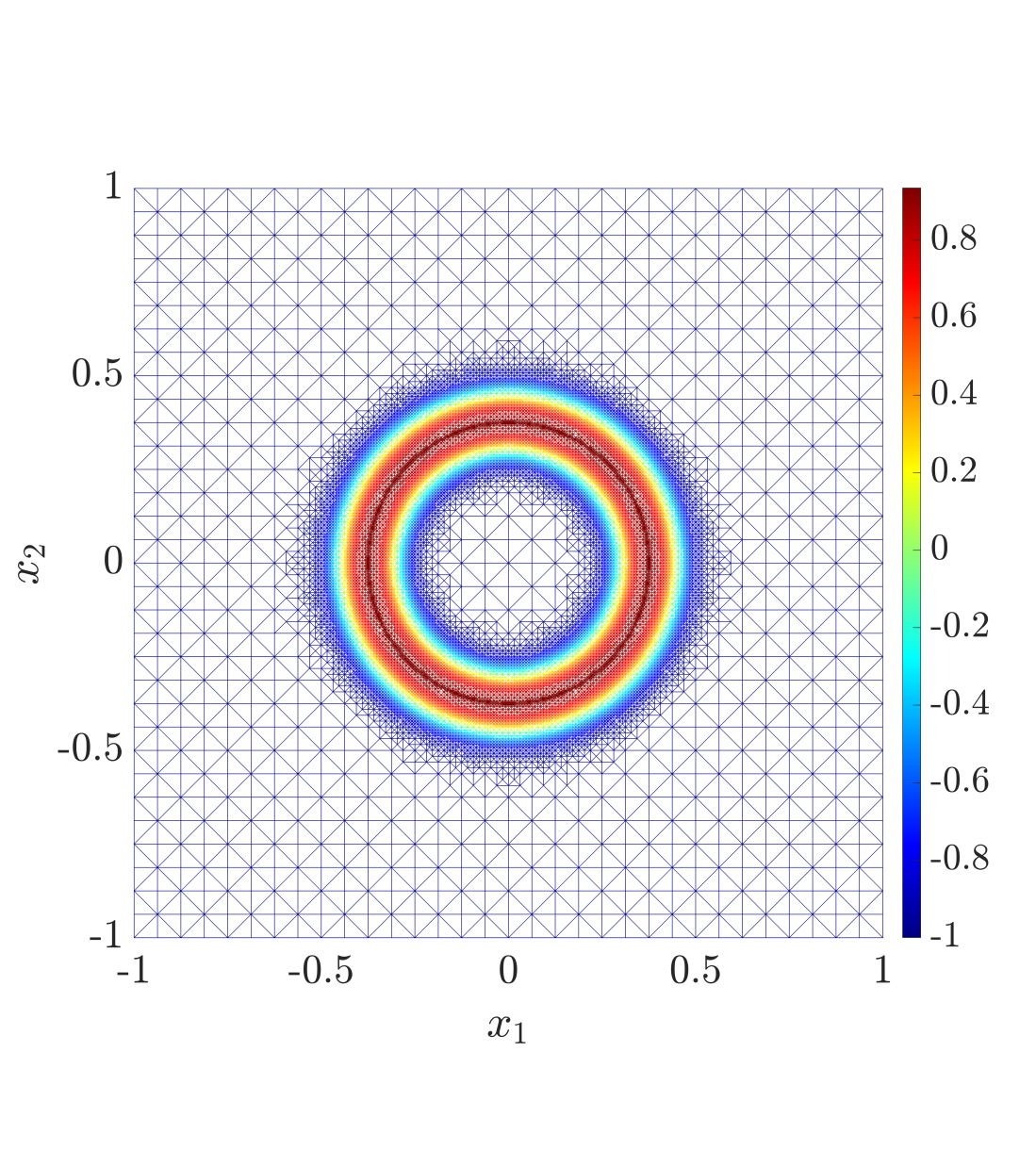}
\end{figure}
\end{minipage}
\hfill
\begin{minipage}{0.48\textwidth}
\begin{figure}[H]
\centering
\includegraphics[scale=0.35, trim=0 35 0 50, clip]{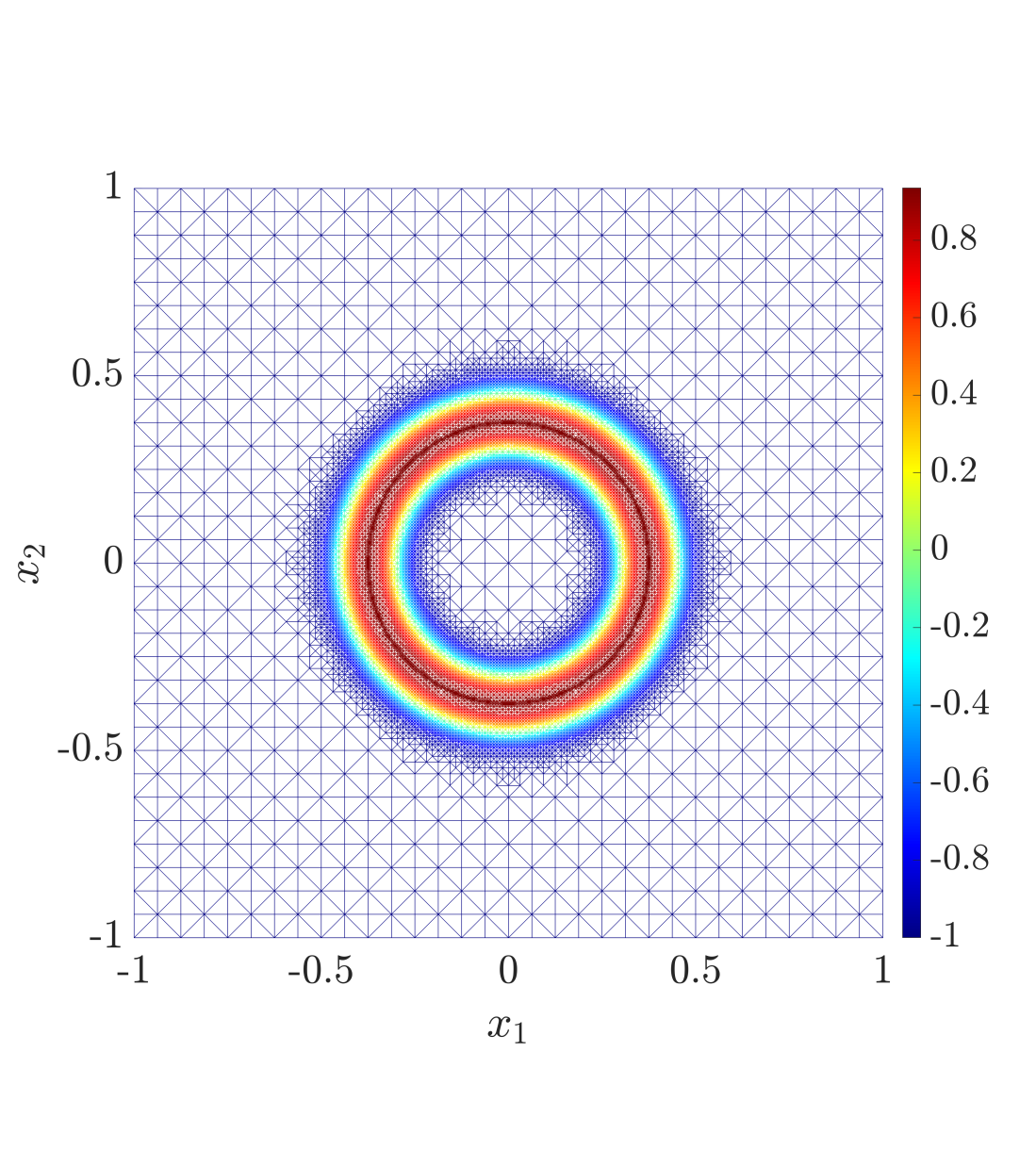}
\end{figure}
\end{minipage}\\
\begin{minipage}{0.48\textwidth}
\begin{figure}[H]
\centering
\includegraphics[scale=0.35, trim=0 35 0 50, clip]{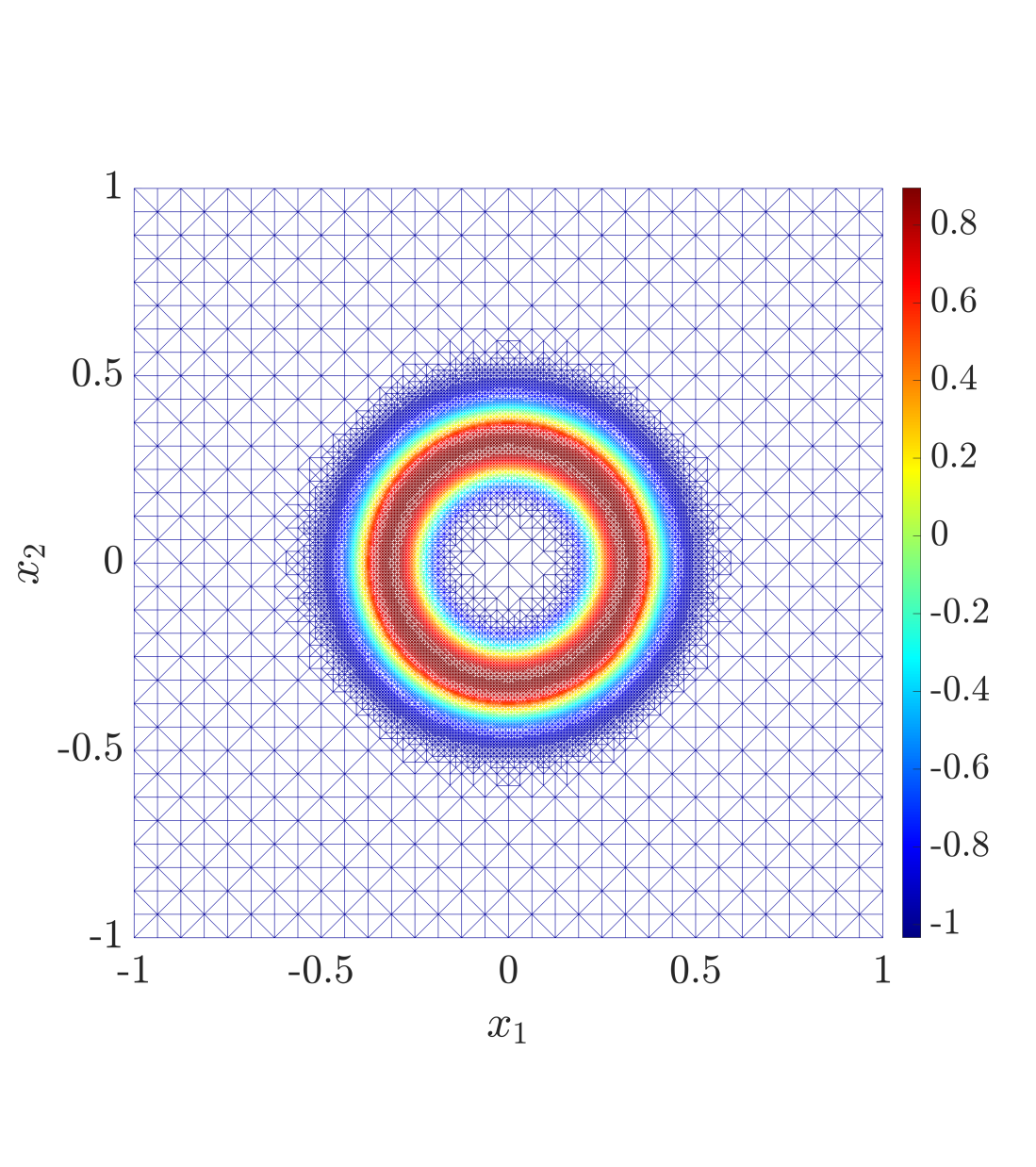}
\end{figure}
\end{minipage}
\hfill
\begin{minipage}{0.48\textwidth}
\begin{figure}[H]
\centering
\includegraphics[scale=0.35, trim=0 35 0 50, clip]{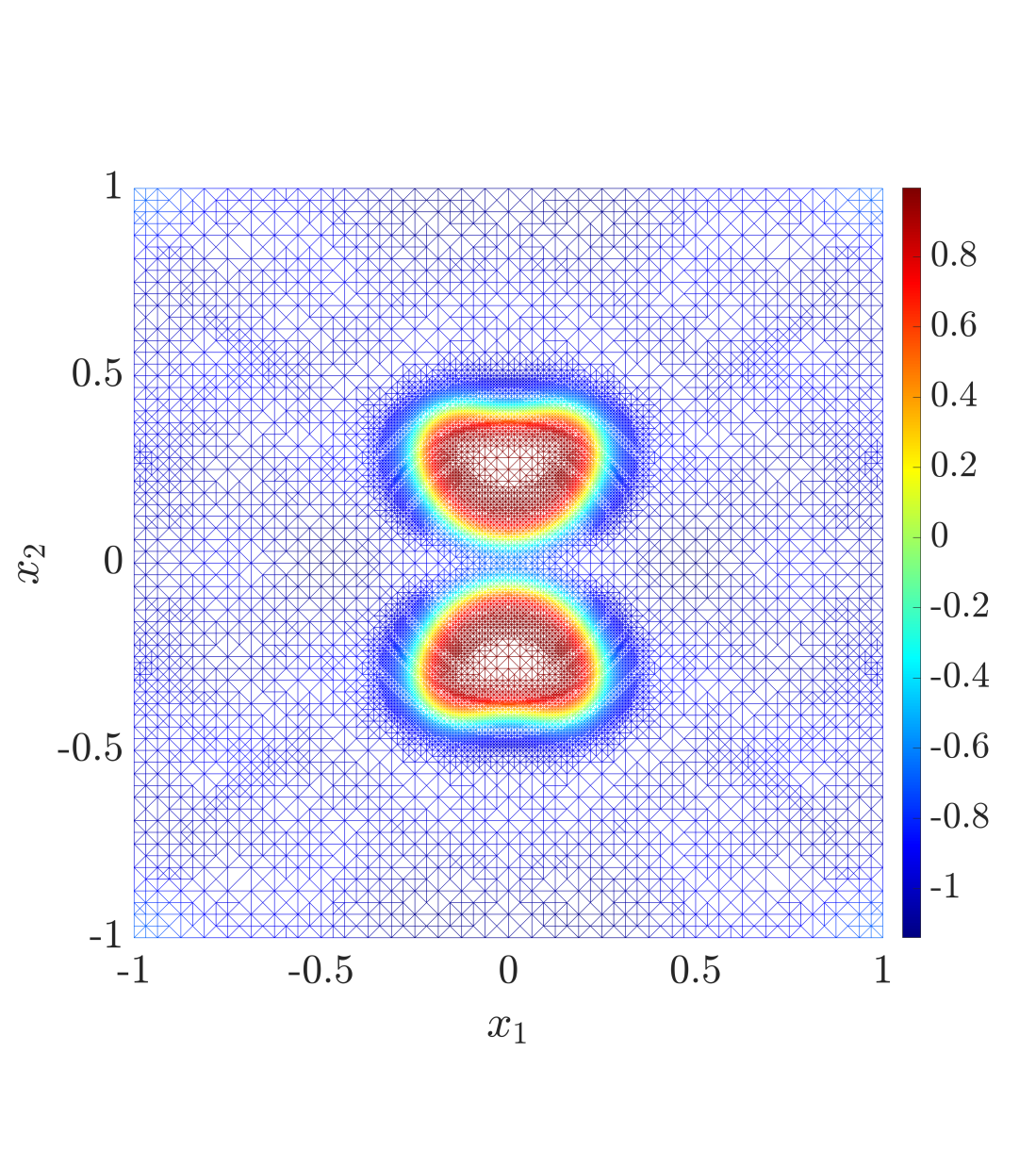}
\end{figure}
\end{minipage}\\
\begin{minipage}{0.48\textwidth}
\begin{figure}[H]
\centering
\includegraphics[scale=0.35, trim=0 35 0 50, clip]{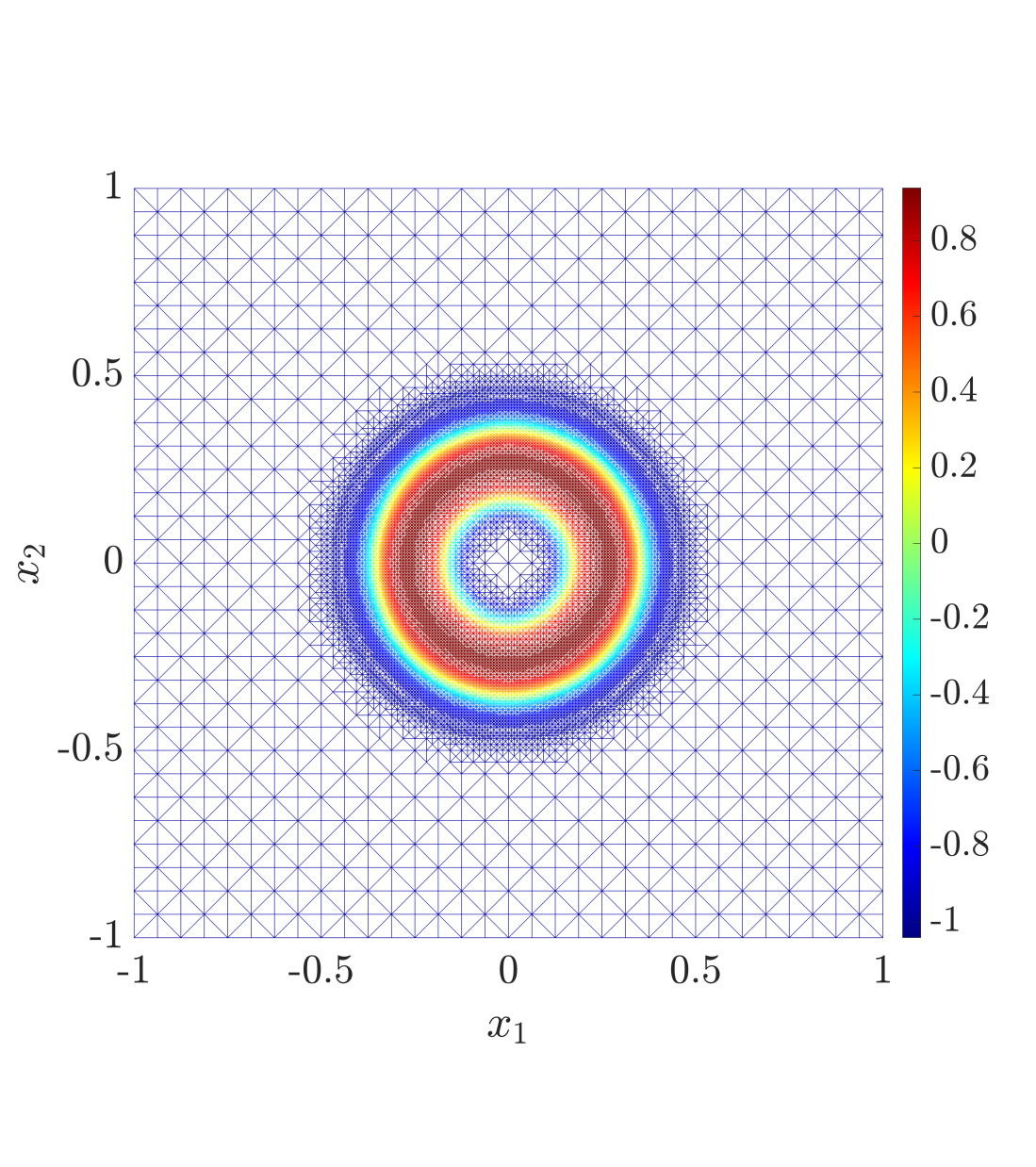}
\end{figure}
\end{minipage}
\hfill
\begin{minipage}{0.48\textwidth}
\begin{figure}[H]
\centering
\includegraphics[scale=0.35, trim=0 35 0 50, clip]{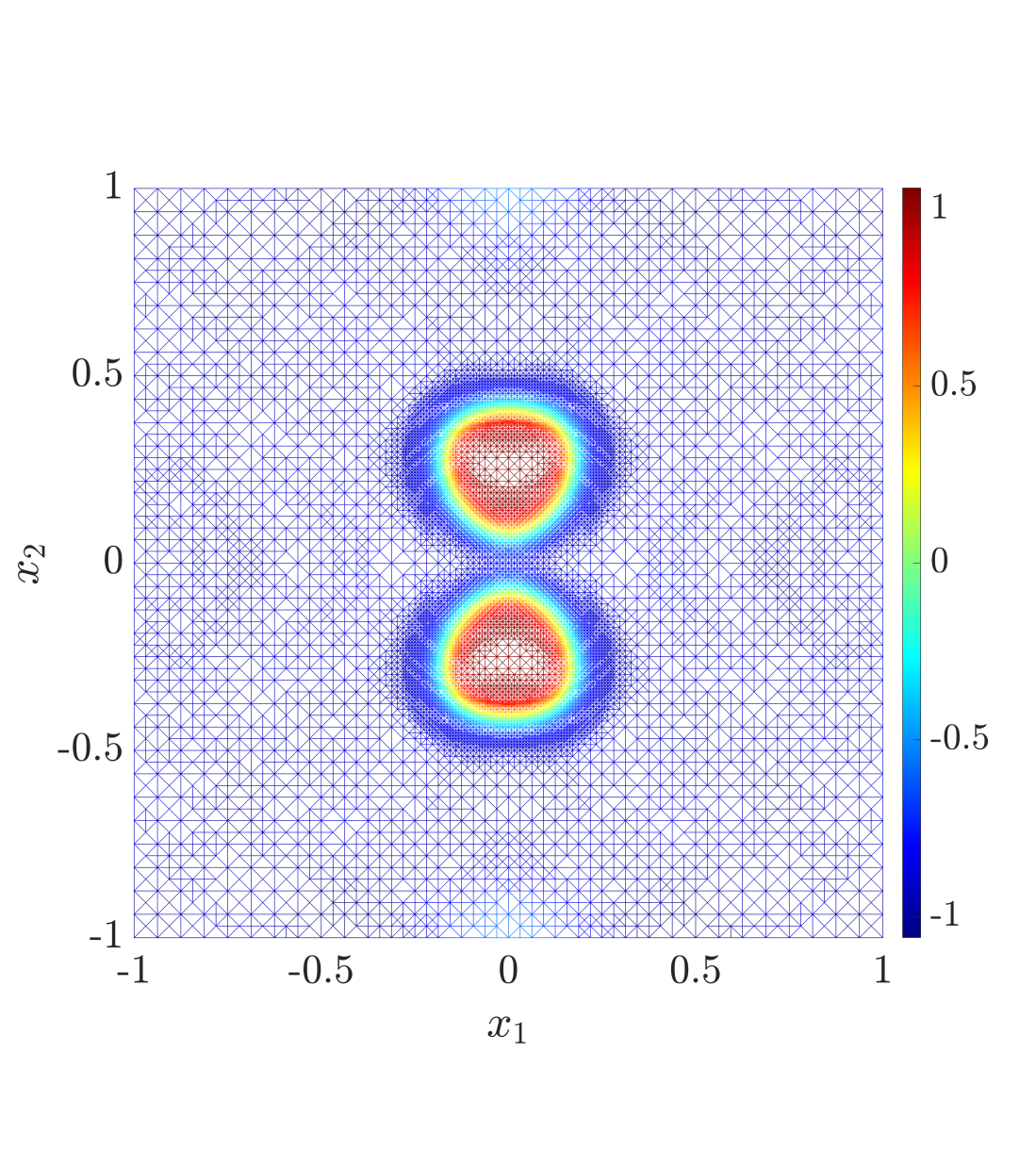}
\end{figure}
\end{minipage}\\
\begin{minipage}{0.48\textwidth}
\begin{figure}[H]
\centering
\includegraphics[scale=0.35, trim=0 35 0 50, clip]{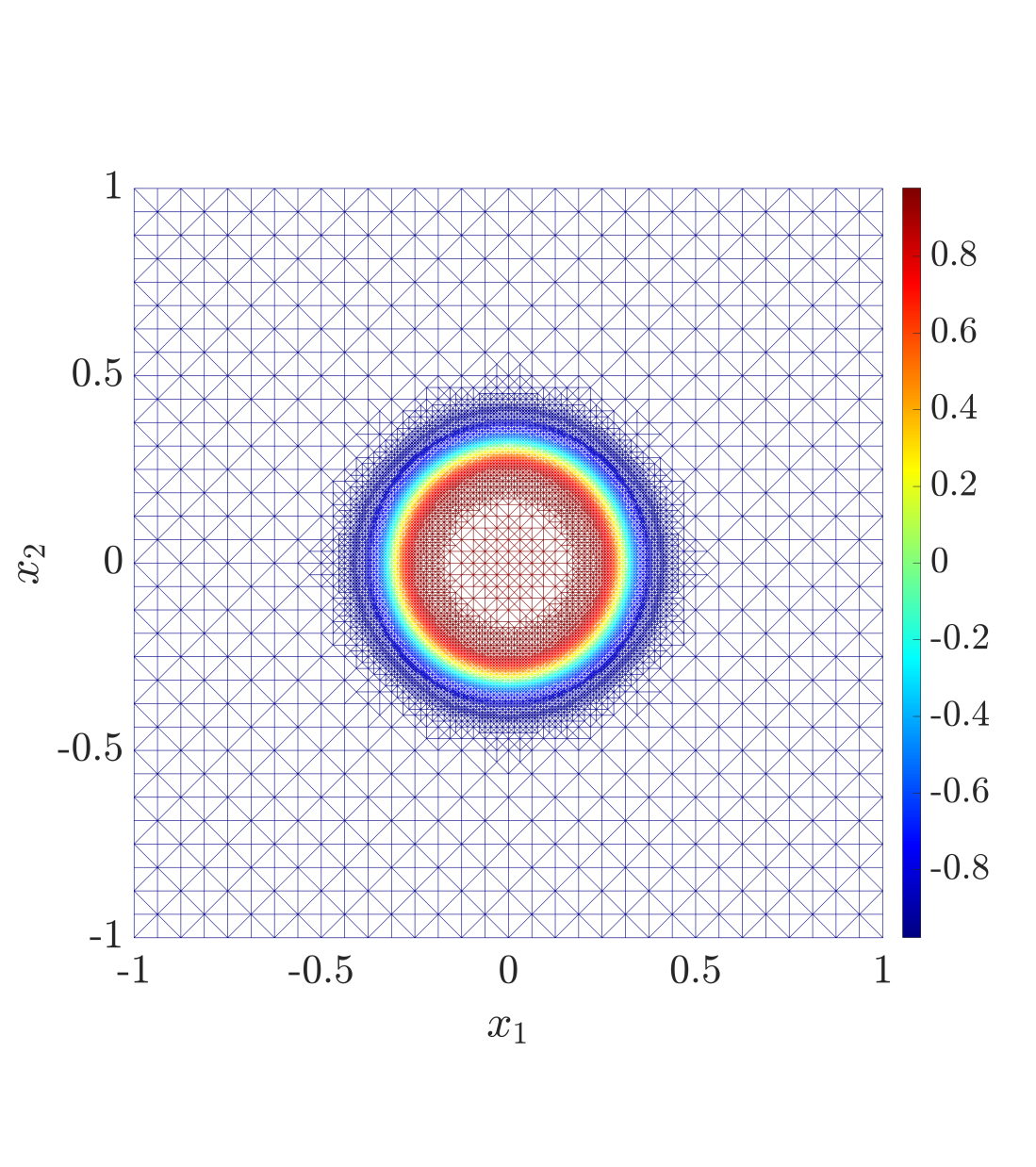}
\end{figure}
\end{minipage}
\hfill
\begin{minipage}{0.48\textwidth}
\begin{figure}[H]
\centering
\includegraphics[scale=0.35, trim=0 35 0 50, clip]{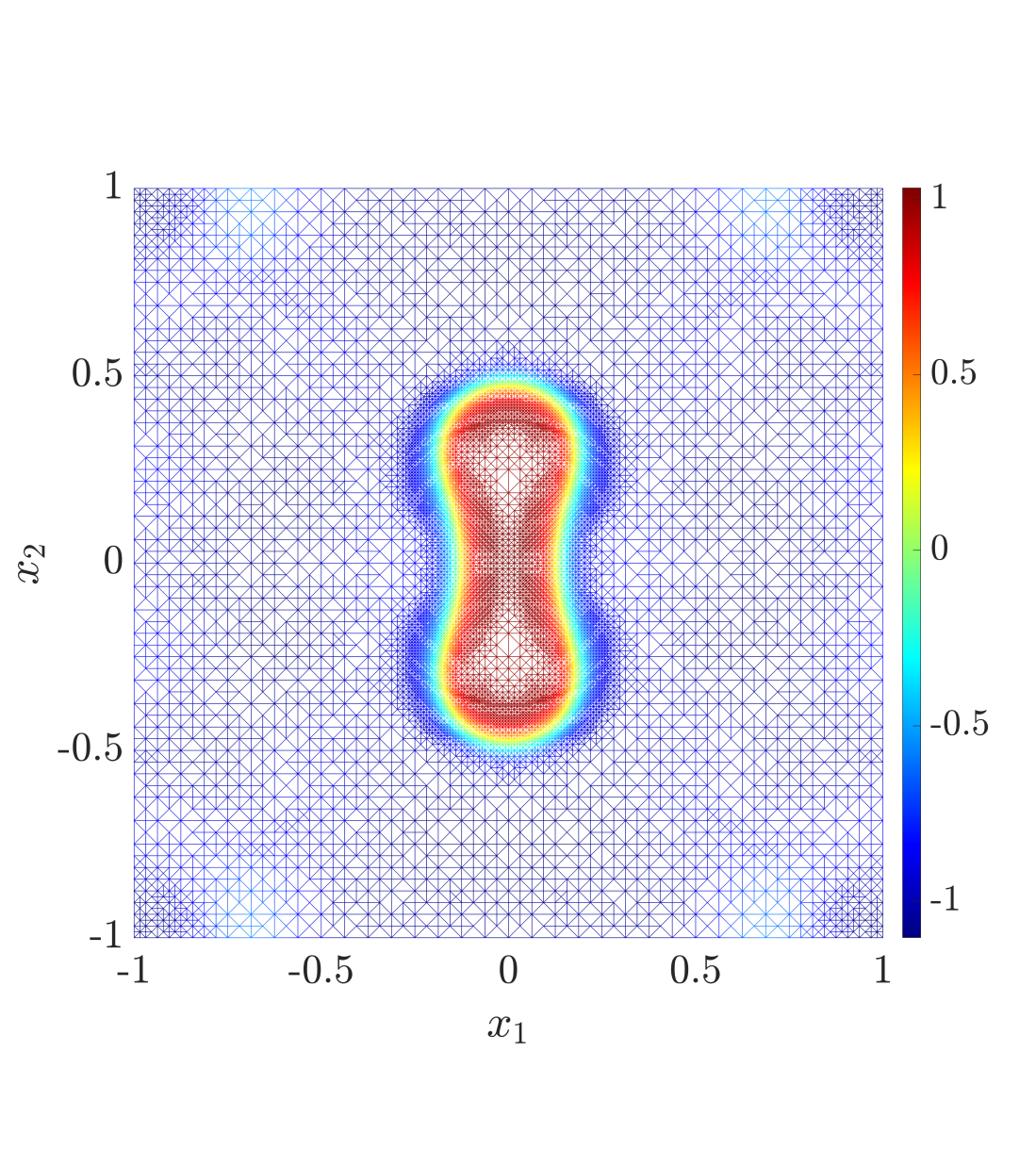}
\end{figure}
\end{minipage}\\
\caption{Snapshots of the mesh for the numerical solution at $t = 0, 0.004, 0.008, 0.012$ (from top to bottom)
deterministic solution (left) and stochastic solution for $\sigma=5$ (right).}
\label{Fig_CHstoch_detstoch_2_mesh}
\end{figure}
\end{center}

\begin{center}\begin{figure}[htp]
\centering
\includegraphics[scale=0.5, trim=0 0 0 0, clip]{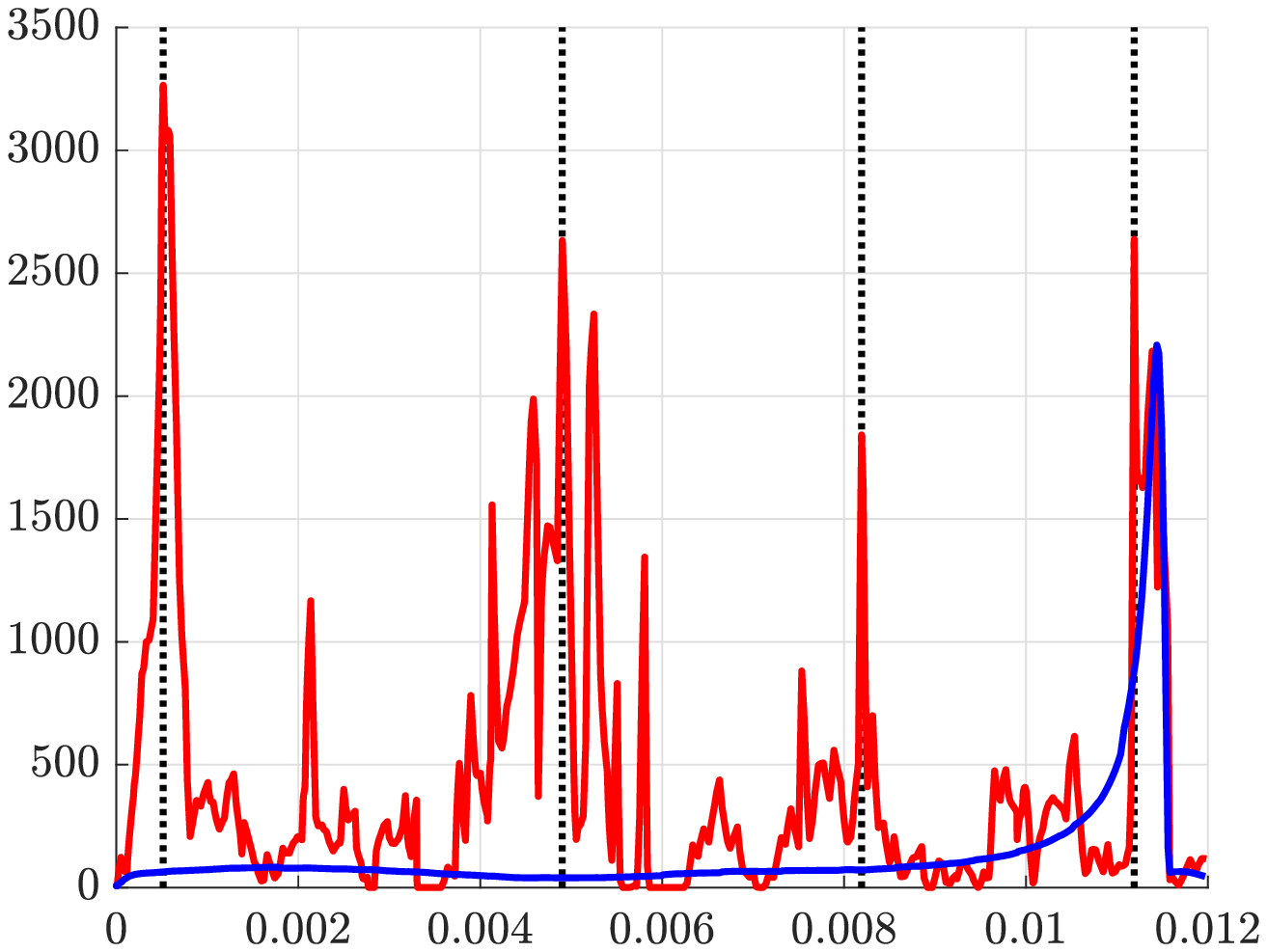}
\caption{Evolution of the principal eigenvalue, deterministic case (blue), stochastic case for $\sigma=5$ (red). 
The dotted vertical lines correspond to the time-levels in Fig.~\ref{Fig_CHstoch_stoch_2_peak_mesh}.}
\label{Fig_CHstoch_2_Lambda}
\end{figure}
\end{center}

\begin{center}\begin{figure}[htp]
\begin{minipage}{0.48\textwidth}
\begin{figure}[H]
\centering
\includegraphics[scale=0.48, trim=0 35 0 50, clip]{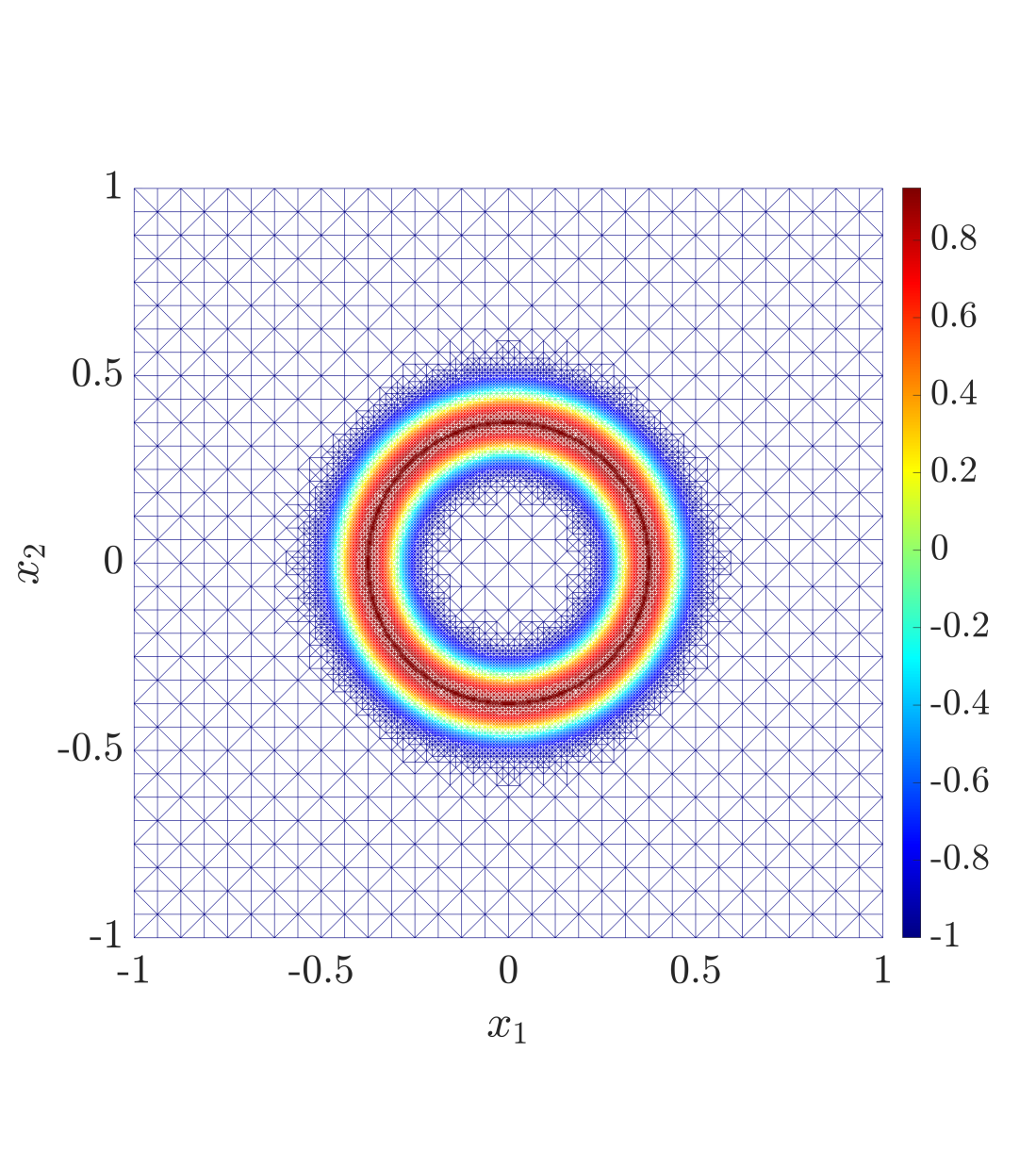}
\end{figure}
\end{minipage}
\hfill
\begin{minipage}{0.48\textwidth}
\begin{figure}[H]
\centering
\includegraphics[scale=0.48, trim=0 35 0 50, clip]{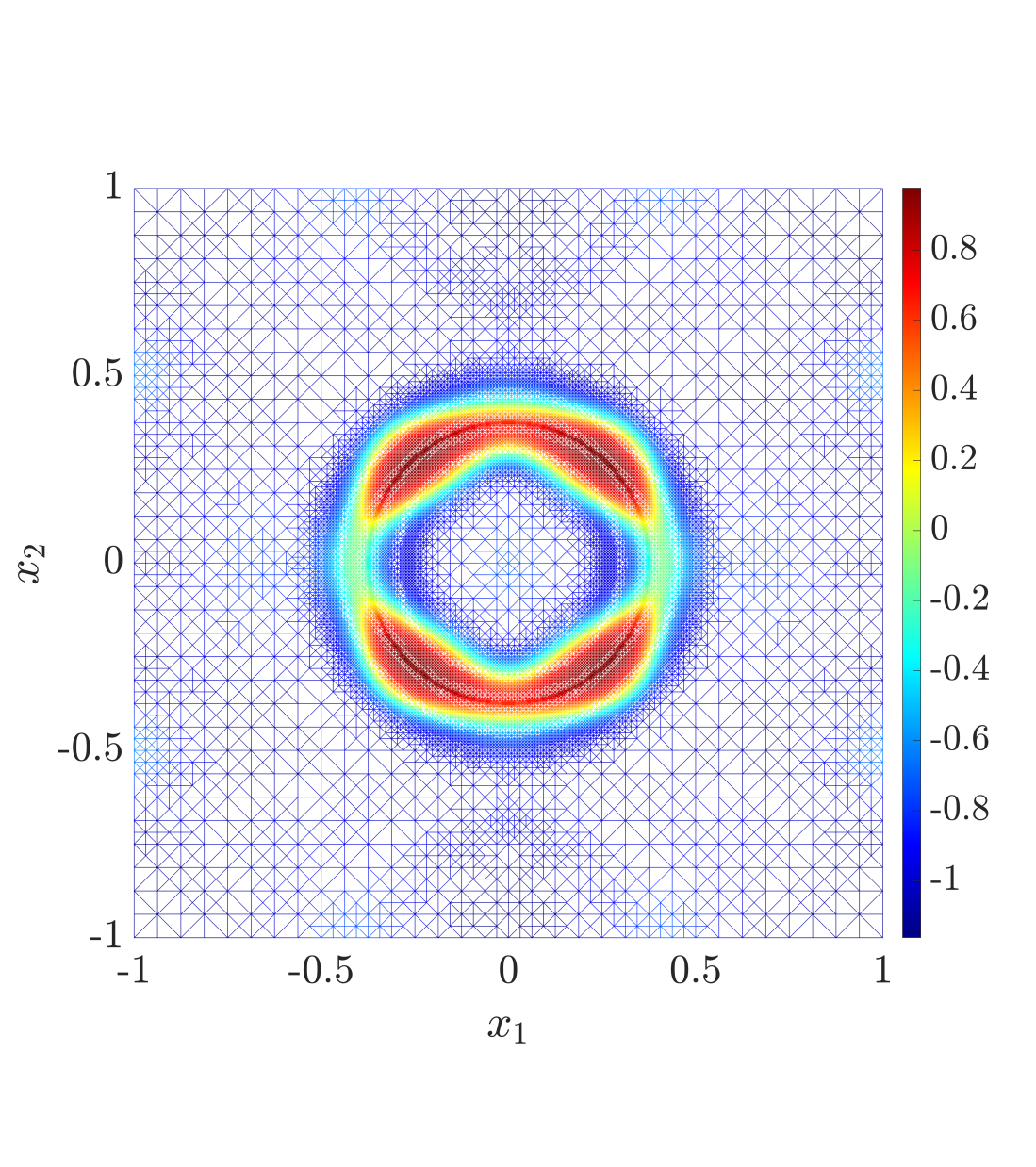}
\end{figure}
\end{minipage}\\
\begin{minipage}{0.48\textwidth}
\begin{figure}[H]
\centering
\includegraphics[scale=0.48, trim=0 35 0 50, clip]{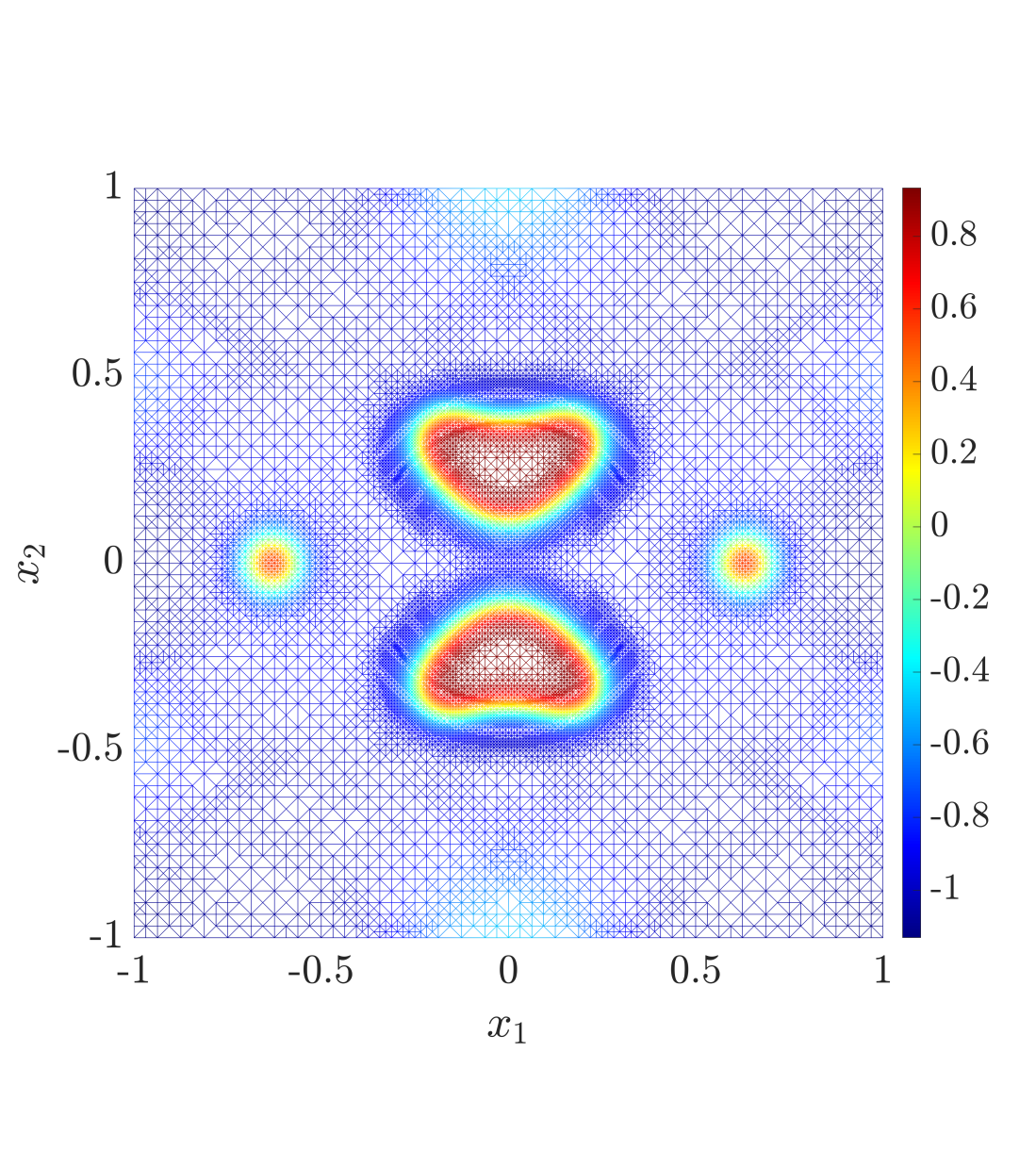}
\end{figure}
\end{minipage}
\hfill
\begin{minipage}{0.48\textwidth}
\begin{figure}[H]
\centering
\includegraphics[scale=0.48, trim=0 35 0 50, clip]{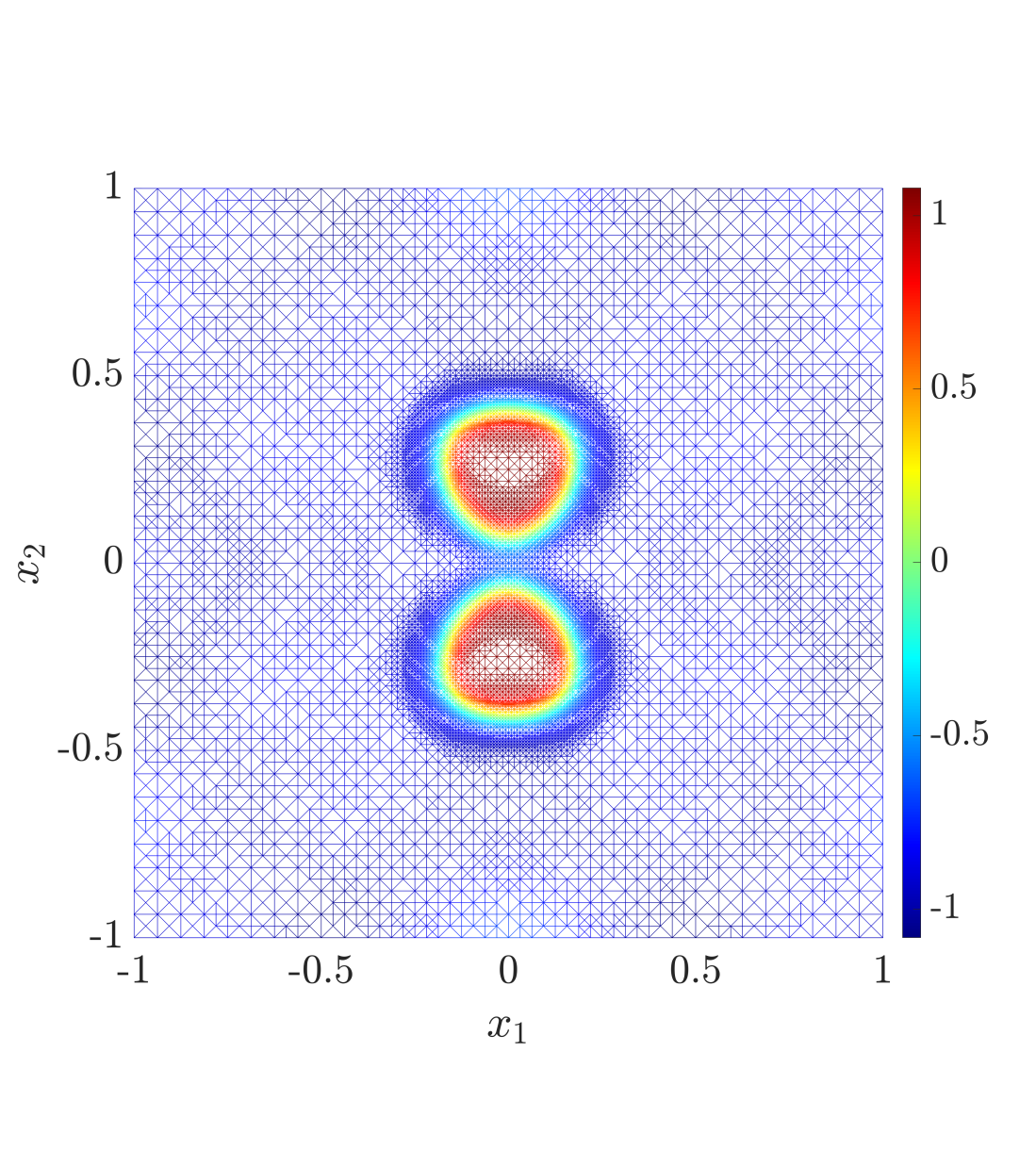}
\end{figure}
\end{minipage}
\caption{Snapshots of the mesh for the numerical solution at the peaks of the stochastic principal eigenvalue for $\sigma=5$ at time $t = 0.0005, 0.0049, 0.0081, 0.0111$.}
\label{Fig_CHstoch_stoch_2_peak_mesh}
\end{figure}
\end{center}

\section*{Acknowledgement}

We thank Martin Ondrej\'at for helping us with the proof of the estimate of the term $I_{1,6}$ in Lemma~\ref{Lemma_CHest_etilde}.

\bibliographystyle{plain}
\bibliography{references}

\end{document}